\documentclass{amsart}
\usepackage{a4}
\usepackage{multicol}

\usepackage{dsfont}

\usepackage{color}

\usepackage{amsmath}
\usepackage{amssymb}

\usepackage[all]{xy}

\usepackage{longtable}
\usepackage{booktabs}

\usepackage[backrefs,lite]{amsrefs}

\usepackage{enumitem}	

\usepackage{tikz}		
\usepackage{tikz-3dplot}
\usepackage{tikz-cd}
\usepackage{graphics}

\usetikzlibrary{patterns,snakes,matrix,arrows,shapes}

\CompileMatrices

\newtheorem{theorem}{Theorem}[section]
\newtheorem{lemma}[theorem]{Lemma}
\newtheorem{proposition}[theorem]{Proposition}
\newtheorem{corollary}[theorem]{Corollary}

\newcommand{\thistheoremname}{}
\newtheorem*{genericthm*}{\thistheoremname}
\newenvironment{namedthm*}[1]
  {\renewcommand{\thistheoremname}{#1}%
   \begin{genericthm*}}
  {\end{genericthm*}}

\theoremstyle{definition}

\newtheorem{definition}[theorem]{Definition}
\newtheorem{example}[theorem]{Example}

\newtheorem{construction}[theorem]{Construction}

\newtheorem{recipe}[theorem]{Recipe}

\newtheorem{remark}[theorem]{Remark}
\theoremstyle{remark}

\numberwithin{equation}{section}

\newcommand\KK{{\mathbb K}}
\newcommand\TT{{\mathbb T}}
\newcommand\ZZ{{\mathbb Z}}

\newcommand\QQ{{\mathbb Q}}
\newcommand\PP{{\mathbb P}}

\newcommand\eins{\mathds{1}}

\newcommand\Eff{{\rm Eff}}
\newcommand\Mov{{\rm Mov}}
\newcommand\Ample{{\rm Ample}}
\newcommand\SAmple{{\rm SAmple}}

\newcommand\rlv{{\rm rlv}}

\renewcommand\div{{\rm div}}
\newcommand\Cl{\operatorname{Cl}}
\newcommand\Pic{\operatorname{Pic}}

\newcommand\cone{{\rm cone}}

\newcommand\Spec{{\rm Spec}}

\newcommand\quot{/\!\!/}
\newcommand\chquot{/\!\!/\!\!/}
\newcommand\tm{{\tau^-}}
\newcommand\tp{{\tau^+}}
\newcommand\tx{{\tau_X}}

\newcommand\Aut{\operatorname{Aut}}

\newcommand\GL{\mathrm{GL}}
\newcommand\PGL{\mathrm{PGL}}
\newcommand\Gr{\mathrm{Gr}}

\newcommand\grad{{\rm grad}}

\newcommand\lin{{\rm lin}}

\newcommand\pr{{\rm pr}}

\newcommand\rk{{\rm rk}\,}

\newcommand\Chi{{\mathbb X}}

\newcommand\im{{\rm im}}

\newcommand\Bl{{\mathrm{Bl}}}

\newcommand\bangle[1]{\langle #1 \rangle} 

\begin{document}
\title[On torus actions of higher complexity]%
{On torus actions of higher complexity}
%
%
%
\author[J.~Hausen]{J\"urgen Hausen} 
\address{Mathematisches Institut, Universit\"at T\"ubingen,
Auf der Morgenstelle 10, 72076 T\"ubingen, Germany}
\email{juergen.hausen@uni-tuebingen.de}
\author[C.~Hische]{Christoff Hische} 
\address{Mathematisches Institut, Universit\"at T\"ubingen,
Auf der Morgenstelle 10, 72076 T\"ubingen, Germany}
\email{hische@math.uni-tuebingen.de}
\author[M.~Wrobel]{Milena Wrobel} 
\address{Max-Planck-Institut f\"ur Mathematik in den Naturwissenschaften, Inselstra\ss e~22, 04103 Leipzig, Germany}
\email{wrobel@mis.mpg.de}
\begin{abstract}
We systematically produce algebraic 
varieties with torus action by 
constructing them as suitably 
embedded subvarieties of toric 
varieties.
The resulting varieties admit 
an explicit treatment in terms 
of toric geometry and graded ring 
theory.
Our approach extends existing 
constructions of rational varieties 
with torus action of complexity one 
and delivers all Mori dream spaces 
with torus action.
We exhibit the example class of 
``general arrangement varieties'' 
and obtain classification results 
in the case of complexity two 
and Picard number at most two, 
extending former work in complexity 
one.
\end{abstract}

\maketitle


\section{Introduction}

This article contributes to the study of 
(algebraic) varieties with torus action.
Here, a torus is an algebraic group $\TT$ 
isomorphic to the $k$-fold direct product
$\TT^k$ of the multiplicative group $\KK^*$
of the ground field $\KK$,
which is assumed to be algebraically closed 
and of characteristic zero. 
The presence of a torus action on a variety
brings combinatorial aspects into the game,
as becomes most evident in the case of 
toric varieties, that means normal varieties
$Z$ containing a torus $\TT_Z$ as an open
subset such that the group structure of 
$\TT_Z$ extends to an action on~$Z$.
The fundamental correspondence between toric 
varieties and fans observed in 1970 by 
Demazure~\cite{Dem} is the starting point 
of toric geometry, a meanwhile highly 
developed theory and continuously active 
field of research.

Beyond toric geometry, there are the
$\TT$-varieties $X$ of higher complexity, 
where one requires~$X$ to be normal,
the $\TT$-action to be effective, meaning
that only the neutral element acts trivially,
and the complexity is the difference 
$\dim(X) - \dim(\TT)$.
Torus actions of complexity one are 
studied as well since the 1970s; 
we mention the combinatorial 
approaches~\cite{KKMS,Tim2}, 
the geometric 
work~\cite{OrWa1,OrWa2,OrWa3,Pi,FiKp1,FiKp2,FlZa} 
on $\KK^*$-surfaces and the more
algebraic point of view based on 
trinomial relations~\cite{HaHeSu,HaHe,HaWr,Mo}.
In arbitrary complexity, we have the general 
approach via polyhedral divisors~\cite{AlHa,AlHaSu},
unifying in particular aspects of~\cite{Tim2,FlZa}.
Torus actions serve also as a model case for 
actions of more general algebraic groups, 
for example reductive ones,
and the approaches just discussed  
reflect in this much more general theory; 
we restrict ourselves to refer to
the seminal work~\cite{LuVu} in complexity 
zero and~\cite{Tim1} as a landmark in 
complexity one.

\goodbreak

In the present article we consider 
 $\TT$-varieties of arbitrary complexity.
The aim is to provide an explicit approach
with close links to toric geometry, 
supporting, for instance, concrete, 
example-oriented work.
Besides basic algebraic geometry and graded
ring theory, only rudiments of toric geometry 
and Cox ring theory are needed; we refer to
Section~\ref{sec:background} for a brief 
summary.

Let us get into the matter by means of 
an example.
Consider the five-dimensional 
smooth projective quadric $X \subseteq \PP_6$.
The automorphism group of $X$ is the 
orthogonal group~$\mathrm{O}(7)$.
Fixing a maximal torus 
$\TT \subseteq \mathrm{O}(7)$, 
we turn $X$ into a $\TT$-variety.
Choosing suitable  
coordinates, we achieve that the 
quadric is given by
$$ 
X
\  = \
V(T_0^2 + T_1T_2 + T_3T_4 + T_5T_6) 
\ \subseteq \ 
\PP_6
$$
and that the elements 
$t = (t_1,t_2,t_3)$ of the (three-dimensional) torus 
$\TT = \TT^3$ act on the points $[z] = [z_0,\ldots,z_6]$ 
of~$\PP_6$ via
$$
t \cdot [z]
\ = \  
[z_0, \, t_1z_1, \, t_1^{-1}z_2, \, t_2z_3, \, 
t_2^{-1}z_4, \, t_3z_5, \, t_3^{-1}z_6].
$$
In particular, $\TT$ acts diagonally.
In order to link the situation 
in an optimal manner to toric geometry, 
we do a further step.
Consider the torus $\TT^6 \subseteq \PP_6$
consisting of all points with only 
nonzero homogeneous coordinates
and the splitting
$$
\TT^6 \ \to \  \TT^3 \times \TT^3,
\qquad
t \ \mapsto \ (t_1t_2, \, t_3t_4, \, t_5t_6, \, t_1, \, t_2 \, ,t_3).
$$ 
In terms of toric geometry, such a change of 
torus coordinates means passing to 
another describing fan of the projective 
space $\PP_6$.
More precisely, we switch over from
the fan~$\Delta_6$ with rays 
generated by the canonical basis vectors 
$e_1, \ldots, e_6 \in \QQ^6$ 
and $e_0 = - e_1 - \ldots - e_6$
to the fan $\Sigma$ with rays 
generated by
$$
e_0-e_1-e_2-e_3, 
\ 
e_1, 
\ 
e_1+e_4,  
\ 
e_2, 
\ 
e_2+e_5, 
\ 
e_3, 
\ e_3+e_6.
$$
In the new picture, our torus $\TT$ is the 
second factor of the splitting  
$\TT^6 = \TT^3 \times \TT^3$.
The projection $\TT^6 \to \TT^3$
onto the first factor mods out the $\TT$-action
and defines a $\TT$-invariant rational 
map $\pi \colon \PP_6 \dasharrow \PP_3$
which in terms of fans arises from
the projection $\ZZ^6 \to \ZZ^3$ onto the 
first factor mapping the rays of $\Sigma$ 
onto the rays of the fan $\Delta_3$ describing
$\PP_3$. 
Observe that the closure of the image $\pi(X)$ 
in $\PP_3$ is
$$ 
Y 
\ = \ 
V(U_0 + U_1 + U_2 + U_3) 
\ \subseteq \ 
\PP_3,
$$ 
the projective plane realized as a general 
hyperplane in the projective space.
The restriction $\pi \colon X \dasharrow Y$
encodes important information on the $\TT$-action
and is what we will call 
a ``maximal orbit quotient''.
Moreover, $X \subseteq \PP_6$ is 
the zero set of a homogeneous quadratic 
polynomial~$g$ and thus we have a homogeneous 
coordinate ring 
$\KK[T_0,\ldots,T_6] / \bangle{g}$.
In case of the quadric $X$, 
the latter equals the \emph{Cox ring}
$$ 
\mathcal{R}(X)
\ =  \
\bigoplus_{\Cl(X)} \Gamma(X,\mathcal{O}(D)),
$$
where the grading via the divisor class group 
$\Cl(X)$ is just the classical $\ZZ$-grading 
of its homogeneous coordinate ring.
Also for $Y \subseteq \PP_3$ we observe that 
the homogeneous coordinate ring equals the Cox 
ring.

Our approach replaces the ambient projective 
spaces of the above example with ambient 
toric varieties.
Given a toric variety~$Z$, say a 
complete one, Cox's quotient presentation 
delivers $Z$ as a quotient of an affine space.
This allows us in particular, to associate
with a closed subvariety $X \subseteq Z$ a 
generalized homogeneous coordinate ring.
We call $X \subseteq Z$ an \emph{explicit variety}
if, roughly speaking, its generalized 
homogeneous coordinate ring equals its 
Cox ring.
In Construction~\ref{constr:t-mds},
we produce systematically 
\emph{explicit $\TT$-varieties}, 
where the idea is to
\emph{reverse the process} of the 
introductory example:
one starts with an explicit 
variety $Y \subseteq Z_Y$ and then 
builds up via an elementary game an 
explicit variety $X \subseteq Z_X$
such that a direct factor~$\TT$ of the torus 
$\TT_Z \subseteq Z$ leaves $X$ invariant and 
turns $X \subseteq Z$ into a
$\TT$-variety with maximal orbit 
quotient $X \dasharrow Y$.
So, in the above example, the projective 
plane $Y \subseteq \PP_3$ is the input 
and as output we obtain the quadric 
$X \subseteq \PP_6$.

Explicit $\TT$-varieties are designed to 
be directly accessible for concrete 
computation.
Their geometry is strongly related to that 
of the ambient toric variety and,
due to finite generation of their Cox ring, 
the combinatorial methods developed 
in~\cite[Sec.~3]{ArDeHaLa} apply;
see Remark~\ref{rem:fan2bunch} for the 
precise interface and 
Section~\ref{sec:taohcfirstprop} for 
a collection of basic geometric properties.
Let us say a few words about what kind 
of $\TT$-varieties we obtain.
First, Construction~\ref{constr:t-mds} 
generalizes the Cox ring based approach to 
rational $\TT$-varieties of complexity 
one developed in~\cites{HaHe,HaWr};
see Remark~\ref{rem:extv2cpl1}.
For a general statement, recall the 
\emph{Mori dream spaces}
introduced by Hu and Keel~\cite{HuKe}:
these varieties behave perfectly 
with respect to the minimal model 
programme and are characterized 
as the $\QQ$-factorial, projective 
varieties with finitely generated 
Cox ring.
As a special case of the more general
Theorem~\ref{thm:t-mds-2},
we obtain the following.

\goodbreak

\begin{theorem}
Every Mori dream space with an effective 
torus action admits a presentation as 
an explicit $\TT$-variety.
\end{theorem}

As a first major example class,
we exhibit in Section~\ref{sec:hypPlaneAr} 
the \emph{general arrangement varieties}.
These are $\TT$-varieties of 
complexity~$c$ with maximal orbit quotient 
$\pi \colon X \dasharrow \PP_c$ such that 
the critical values of $\pi$ form a general 
hyperplane arrangement in the projective space 
$\PP_c$.
We have already seen an example.
The smooth projective quadric~$X$ with 
its maximal torus action discussed 
at the beginning is a general arrangement 
variety of complexity two.
The critical values of the maximal orbit 
quotient $\pi \colon X \dasharrow \PP_2$ 
are the points of the lines
$$
V(T_0), \qquad 
V(T_1), \qquad V(T_2), 
\qquad V(T_0+T_1+T_2).
$$
Theorem~\ref{thm:arrTvar} ensures that, 
for instance, all projective general 
arrangement varieties can be presented 
as explicit $\TT$-varieties.
In Section~\ref{sec:exfirstprops}
we use the methods on explicit 
$\TT$-varieties to investigate the 
geometry of general arrangement varieties.
For example, Proposition~\ref{prop:genarrcandiv} 
gives an explicit formula for the canonical 
divisor and Corollary~\ref{cor:genarrvarsmooth}
provides a purely combinatorial smoothness 
criterion.

Extending recent classification work
in complexity one~\cite{FaHaNi}, we 
take a closer look at smooth general arrangement 
varieties of Picard number at most two.
In Picard number one, we retrieve precisely 
the smooth projective quadrics,
see Proposition~\ref{prop:arrvar-rho-one}.
Similar to the case of complexity one,
the situation in Picard number two is 
much more ample. 
For the case of complexity two, we obtain 
the following explicit descriptions;
below, we say that a torus action on a
variety is of \emph{true} complexity $c$,
if the action is of complexity $c$ and 
the variety does not admit a torus action
of lower complexity.
Note that being Mori dream spaces, the 
varieties listed below are indeed determined 
by their Cox ring together with an ample 
class; see Remark~\ref{rem:Sigmau}.

\begin{theorem}
\label{theo:proj}
Every smooth projective general arrangement 
variety of true complexity two and Picard number two 
is isomorphic to precisely one of the following 
varieties $X$, specified by their Cox ring 
$\mathcal{R}(X)$, the matrix $[w_1, \dots, w_r]$ 
of generator degrees and an ample class
$u \in \Cl(X) = \ZZ^2$.

\medskip

{\centering
{\small
\setlength{\tabcolsep}{-1pt}
\begin{longtable}{ccccc}
No.
&
\small{$\mathcal{R}(X)$}
&
\small{$[w_1,\ldots, w_r]$}
&
\small{$u$}
&
\small{$\dim(X)$}
\\
\toprule
1
&
$
\frac
{\KK[T_1, \ldots , T_9]}
{\langle T_{1}T_{2}T_{3}^2+T_{4}T_{5}+T_6T_7+T_8T_9 \rangle}
$
&
\tiny{
\setlength{\arraycolsep}{1pt}
$
\begin{array}{c}
\left[
\begin{array}{ccccccccc}
0 & 0 & 1 & 1 & 1 & 1 & 1 & 1 & 1
\\
1 & 1 & 0 & a_1 & a_1' & a_2 & a_2' & a_3 & a_3'
\end{array}
\right]
\\[1em]
1 \leq a_1 \leq a_2 \leq a_3,
\ a_i' = 2-a_i
\end{array}
$
}
&
\tiny{
\setlength{\arraycolsep}{1pt}
$
\left[
\begin{array}{c}
1
\\
a_3+1
\end{array}
\right]
$
}
&
\small{$6$}
\\
\midrule
2
&
$
\frac
{\KK[T_1, \ldots , T_9]}
{\langle T_{1}T_{2}T_{3}+T_{4}T_{5}+T_6T_7 +T_8T_9\rangle}
$
&
\tiny{
\setlength{\arraycolsep}{1pt}
$
\left[
\begin{array}{ccccccccc}
0 & 0 & 1 & 1 & 0 & 1 & 0  & 1 & 0
\\ 
1 & 1 & 0 & 1 & 1 & 1 & 1 & 1 & 1
\end{array}
\right]
$
}
&
\tiny{
$
\left[
\begin{array}{c}
1
\\
2
\end{array}
\right]
$
}
&
\small{$6$} 
\\
\midrule
3
&
$
\frac
{\KK[T_1, \ldots , T_8]}
{\langle T_{1}T_{2}T_{3}^2+T_{4}T_{5}+T_6T_7 +T_8^2\rangle}
$
&
\tiny{
\setlength{\arraycolsep}{1pt}
$
\begin{array}{c}
\left[
\begin{array}{cccccccc}
0 & 0 & 1 & 1 & 1 & 1 & 1 & 1
\\
1 & 1 & 0 & a_1 & a_1' & a_2 & a_2' & 1
\end{array}
\right]
\\[1em]
1 \leq a_1 \leq a_2,
\
a_i' = 2-a_i
\end{array}
$
}
&
\tiny{
\setlength{\arraycolsep}{1pt}
$
\left[
\begin{array}{c}
1
\\
a_2+1
\end{array}
\right]
$
}
&
\small{$5$} 
\\
\midrule
4
&
$
\begin{array}{c}
\frac
{\KK[T_1, \ldots , T_8, S_1,\ldots,S_m]}
{\langle T_{1}T_{2}^{l_{2}}+T_{3}T_{4}^{l_{4}}+T_{5}T_{6}^{l_{6}}+T_{7}T_{8}^{l_{8}} \rangle}
\\
\scriptstyle m \geq 0
\end{array}
$
&
\tiny{
\setlength{\arraycolsep}{1pt}
$
\begin{array}{c}
\left[
\begin{array}{cccccccc|ccc}
0 & 1 & a_1 & 1 & a_2 & 1 & a_3 & 1 & d_1 & \ldots & d_m
\\
1 & 0 & 1 & 0 & 1 & 0 & 1 & 0 & 1 & \ldots & 1
\end{array}
\right]
\\[1em]
0 \leq a_1 \leq a_2 \leq a_3,
\
d_1 \leq \ldots \leq d_m,
\\
l_2 = a_1+l_4 = a_2+l_6 = a_3+l_8
\end{array}
$
}
&
\tiny{
\setlength{\arraycolsep}{1pt}
$
\begin{array}{c}
\left[
\begin{array}{c}
d
\\
1
\end{array}
\right]
\\[1em]
d \text{ max}
\\
\text{of } a_3,d_m
\end{array}
$
}
&
\small{$m+5$}
\\

\midrule
5
&
$
\begin{array}{c}
\frac
{\KK[T_1, \ldots , T_8, S_1,\ldots,S_m]}
{\langle T_{1}T_{2}+T_{3}^2T_{4}+T_5^2T_{6}+T_7^2T_8 \rangle}
\\
\scriptstyle m \geq 0
\end{array}
$
&
\tiny{
\setlength{\arraycolsep}{1pt}
$
\begin{array}{c}
\left[
\begin{array}{cccccccc|ccc}
0 & 2a+1 & a & 1 & a & 1 & a & 1  & 1 & \ldots & 1
\\
1 & 1 & 1 & 0 & 1 & 0 & 1 & 0 & 0 & \ldots & 0
\end{array}
\right]
\\[1em]
a \geq 0
\end{array}
$
}
&
\tiny{
\setlength{\arraycolsep}{1pt}
$
\left[
\begin{array}{c}
2a+2
\\
1
\end{array}
\right]
$
}
&
\small{$m+5$}
\\
\midrule
6
&
$
\begin{array}{c}
\frac
{\KK[T_1, \ldots , T_8, S_1,\ldots,S_m]}
{\langle T_{1}T_{2}+T_{3}T_{4}+T_5^2T_{6}+T_7^2T_8 \rangle}
\\
\scriptstyle m \geq 0
\end{array}
$
&
\tiny{
\setlength{\arraycolsep}{1pt}
$
\begin{array}{c}
\left[
\begin{array}{cccccccc|ccc}
0 & 2a_3+1 & a_1 & a_2 & a_3 & 1 & a_3 & 1  & 1 & \ldots & 1
\\
1 & 1 & 1 & 1 & 1 & 0 & 1 & 0 & 0 & \ldots & 0
\end{array}
\right]
\\[1em]
2a_3+1 = a_1+ a_2, \ 0 \leq a_1 \leq a_2
\end{array}
$
}
&
\tiny{
\setlength{\arraycolsep}{1pt}
$
\left[
\begin{array}{c}
2a_3+2
\\
1
\end{array}
\right]
$
}
&
\small{$m+5$}
\\
\midrule
7
&
$
\begin{array}{c}
\frac
{\KK[T_1, \ldots , T_8, S_1,\ldots,S_m]}
{\langle T_{1}T_{2}+T_{3}T_{4}+T_5T_{6}+T_7^2T_8 \rangle}
\\
\scriptstyle m \geq 1
\end{array}
$
&
\tiny{
\setlength{\arraycolsep}{1pt}
$
\begin{array}{c}
\left[
\begin{array}{cccccccc|ccc}
0 & 2a_5+1 & a_1 & a_2 & a_3 & a_4 & a_5 & 1  & 1 & \ldots & 1
\\
1 & 1 & 1 & 1 & 1 & 1 & 1 & 0 & 0 & \ldots & 0
\end{array}
\right]
\\[1em]
2a_5+1 = a_1+ a_2 = a_3 + a_4,
\ a_i \ge 0
\end{array}
$
}
&
\tiny{
\setlength{\arraycolsep}{1pt}
$
\left[
\begin{array}{c}
2a_5+2
\\
1
\end{array}
\right]
$
}
&
\small{$m+5$}
\\
\midrule
8
&
$
\begin{array}{c}
\frac
{\KK[T_1, \ldots , T_8, S_1,\ldots,S_m]}
{\langle T_{1}T_{2}+T_{3}T_{4}+T_5T_{6}+T_7T_8 \rangle}
\\
\scriptstyle m \geq 1
\end{array}
$
&
\tiny{
\setlength{\arraycolsep}{1pt}
$
\left[
\begin{array}{cccccccc|ccc}
0 &  0 & 0 & 0 & 0 & 0 & -1 & 1 & 1 & \ldots & 1
\\
1 & 1 & 1 & 1 & 1 & 1 & 1 & 1 & 0 & \ldots & 0
\end{array}
\right]
$
}
&
\tiny{
\setlength{\arraycolsep}{1pt}
$
\left[
\begin{array}{c}
1
\\
2
\end{array}
\right]
$
}
&
\small{$m+5$}
\\
\midrule
9
&
$
\begin{array}{c}
\frac
{\KK[T_1, \ldots , T_8, S_1,\ldots,S_m]}
{\langle T_{1}T_{2}+T_{3}T_{4}+T_5T_{6}+T_7T_8 \rangle}
\\
\scriptstyle m \geq 2
\end{array}
$
&
\tiny{
\setlength{\arraycolsep}{1pt}
$
\begin{array}{c}
\left[
\begin{array}{cccccccc|ccc}
0 & a_1 & a_2 & a_3 & a_4 & a_5 & a_6 & a_7 & 1 & \ldots & 1
\\
1 & 1 & 1 & 1 & 1 & 1 & 1 & 1 & 0 & \ldots & 0
\end{array}
\right]
\\[1em]
a_1 =a_2+a_3=a_4+a_5=a_6+a_7
\\
a_i \geq 0
\end{array}
$
}
&
\tiny{
\setlength{\arraycolsep}{1pt}
$
\left[
\begin{array}{c}
a_1+1
\\
1
\end{array}
\right]
$
}
&
\small{$m+5$}
\\
\midrule
10
&
$
\begin{array}{c}
\frac
{\KK[T_1, \ldots , T_8, S_1,\ldots,S_m]}
{\langle T_{1}T_{2}+T_{3}T_{4}+T_5T_{6}+T_7T_8 \rangle}
\\
\scriptstyle m \geq 2
\end{array}
$
&
\tiny{
\setlength{\arraycolsep}{1pt}
$
\begin{array}{c}
\left[
\begin{array}{cccccccc|cccc}
0 &  0 & 0 & 0 & 0 & 0 & 0 & 0 & 1 & 1 & \ldots & 1
\\
1 & 1 & 1 & 1 & 1 & 1 & 1 & 1 & 0 & d_2 & \ldots & d_m
\end{array}
\right]
\\[1em]
0 \leq d_2\leq\dots \leq d_m, \ d_m > 0
\end{array}
$
}
&
\tiny{
\setlength{\arraycolsep}{1pt}
$
\left[
\begin{array}{c}
1
\\
d_m +1
\end{array}
\right]
$
}
&
\small{$m+5$}
\\
\midrule
11
&
$
\begin{array}{c}
\frac
{\KK[T_1, \ldots , T_7, S_1,\ldots,S_m]}
{\langle T_{1}T_{2}+T_{3}T_{4}+T_5T_{6}+T_7^2 \rangle}
\\
\scriptstyle m \geq 1
\end{array}
$
&
\tiny{
\setlength{\arraycolsep}{1pt}
$
\left[
\begin{array}{ccccccc|ccc}
-1 & 1 & 0 & 0 & 0 & 0 & 0 & 1 & \ldots & 1
\\
1 & 1 & 1 & 1 & 1 & 1 & 1 & 0& \ldots & 0
\end{array}
\right]
$
}
&
\tiny{
\setlength{\arraycolsep}{1pt}
$
\left[
\begin{array}{c}
1
\\
2
\end{array}
\right]
$
}
&
\small{$m+4$}
\\
\midrule
12
&
$
\begin{array}{c}
\frac
{\KK[T_1, \ldots , T_7, S_1,\ldots,S_m]}
{\langle T_{1}T_{2}+T_{3}T_{4}+T_5T_{6}+T_7^2 \rangle}
\\
\scriptstyle m \geq 2
\end{array}
$
&
\tiny{
\setlength{\arraycolsep}{1pt}
$
\begin{array}{c}
\left[
\begin{array}{ccccccc|ccc}
0 & 2a_5 & a_1 & a_2 & a_3 & a_4 & a_5 &  1  & \ldots & 1
\\
1 & 1 & 1 & 1 & 1 & 1 & 1 & 0 & \ldots & 0
\end{array}
\right]
\\[1em]
a_1+a_2 = a_3+a_4 = 2a_5,
\ a_i \geq 0
\end{array}
$
}
&
\tiny{
\setlength{\arraycolsep}{1pt}
$
\left[
\begin{array}{c}
2a_5+1
\\
1
\end{array}
\right]
$
}
&
\small{$m+4$}
\\
\midrule
13
&
$
\begin{array}{c}
\frac
{\KK[T_1, \ldots , T_7, S_1,\ldots,S_m]}
{\langle T_{1}T_{2}+T_{3}T_{4}+T_5T_{6}+T_7^2 \rangle}
\\
\scriptstyle m \geq 2
\end{array}
$
&
\tiny{
\setlength{\arraycolsep}{1pt}
$
\begin{array}{c}
\left[
\begin{array}{ccccccc|cccc}
0 & 0 & 0 & 0 & 0 &0 &0 &  1  & 1 & \ldots & 1
\\
1 & 1 & 1 & 1 & 1 & 1 & 1 & 0 &d_2 & \ldots & d_m
\end{array}
\right]
\\[1em]
0\leq d_2 \leq \ldots \leq d_m, 
\ d_m >0
\end{array}
$
}
&
\tiny{
\setlength{\arraycolsep}{1pt}
$
\left[
\begin{array}{c}
1
\\
d_m+1
\end{array}
\right]
$
}
&
\small{$m+4$}
\\
\midrule
14
&
\qquad
$
\frac
{\KK[T_1, \ldots , T_{10}]}
{\left\langle\!\!\!\!
\tiny{
\begin{array}{c}
T_{1}T_{2}+T_{3}T_{4}+T_5T_{6}+T_7T_8, 
\\
\lambda_1T_{3}T_{4}+\lambda_2T_5T_{6}+T_7T_8+T_9T_{10} 
\end{array}
}
\!\!\!\! \right\rangle}

$
&
\tiny{
\setlength{\arraycolsep}{1pt}
$
\left[
\begin{array}{cccccccccc}
1 & 0 & 1 & 0 & 1 & 0 & 1 & 0 & 1 & 0
\\
0 & 1 & 0 & 1 & 0 & 1 & 0 & 1 & 0 & 1
\end{array}
\right]
$
}
&
\tiny{
\setlength{\arraycolsep}{1pt}
$
\left[
\begin{array}{c}
1
\\
1
\end{array}
\right]
$
}
&
\small{$6$}
\\
\bottomrule
\end{longtable}
}
}
\noindent
Moreover, each of the listed data defines 
a smooth projective general arrangement variety 
of true complexity two and Picard number two.
\end{theorem}

As a direct application, we can contribute 
to the study of smooth Fano varieties with 
torus action; 
note that by~\cite[Cor.~1.3.2]{BCHM} they all 
are Mori dream spaces and thus the methods of 
this article are applicable.
Existing classification work concerns
the toric case~\cite{Ba1,Ba2,KrNi} 
and the case of complexity one~\cite{FaHaNi}.
Our Theorem~\ref{theo:Fano} classifies in every 
dimension the (finitely many) smooth Fano 
general arrangement varieties of complexity 
two and Picard number two.
Moreover, we study in Section~\ref{sec:FanoClass}
the geometry of the Fano varieties 
listed in Theorem~\ref{theo:Fano},
meaning that we describe their elementary divisorial 
contractions 
and provide small degenerations to singular Fano 
varieties of true complexity one.
We obtain a similar finiteness feature 
as observed in~\cite{FaHaNi} in complexity one:
all varieties of Theorem~\ref{theo:Fano} arise via 
two elementary contractions and a series of 
isomorphisms in codimension one from a finite 
set of smooth projective general arrangement 
varieties of complexity two having dimension~5 to~8, 
see Remark~\ref{rem:duplication}. 
Finally, we list in Theorem~\ref{theo:AlFano} the smooth 
truly almost Fano general arrangement varieties 
of true complexity two and Picard number two, 
where truly almost Fano
means that the anticanonical divisor is semiample
but not ample.

\tableofcontents

\section{Background on toric varieties and Cox rings}
\label{sec:background}

We provide the necessary background and 
fix our notation on toric geometry and 
Cox rings.
Throughout the whole article, the ground 
field $\KK$ is algebraically closed 
and of characteristic zero.
Moreover, the word variety refers to an 
integral separated scheme of finite type 
over $\KK$.
In particular, we assume varieties to be 
irreducible.
By a point we mean a closed point.

When we speak of an \emph{action} of an algebraic
group~$G$ on a variety $X$, then we always
assume the action map $G \times X \to X$,
$(g,x) \mapsto g \cdot x$ to be a morphism 
of varieties.  
A \emph{torus} is an algebraic group~$\TT$ 
isomorphic to a \emph{standard torus} 
$\TT^n = (\KK^*)^n$ and a \emph{$\TT$-variety} 
is a normal variety $X$ with 
an effective torus action,
where effective means that only the neutral 
element $\eins \in \TT$ acts trivially.
The \emph{complexity} $c(X)$ of a $\TT$-variety~$X$ 
is the difference $\dim(X)-\dim(\TT)$. 

Toric geometry treats the case of complexity
zero.
More precisely, a \emph{toric variety} is a 
$\TT$-variety~$Z$ with a base point $z_0 \in Z$ 
such that the orbit map 
$t \mapsto t \cdot z_0$ 
yields an open embedding $\TT \to Z$;
we call $\TT_Z = \TT$ the \emph{acting torus} 
of~$Z$ and write $\TT_Z \subseteq Z$,
identifying $\mathds{1} \in \TT$
with $z_0 \in Z$ and~$\TT_Z$ with its orbit 
$\TT_Z \cdot z_0$.
Toric geometry originates in 
Demazure's work~\cite{Dem} in the 1970s
and connects combinatorics, 
represented by fans, with algebraic geometry,
represented by toric varieties.
As introductory references, we 
mention~\cite{Da, Od, Fu, CoLiSc}.
Here comes the fundamental construction,
which at the end yields a covariant equivalence 
between the categories 
of fans and toric varieties.

\begin{construction}
\label{constr:fan2tv}
A \emph{fan} in $\ZZ^n$ is
a finite collection $\Sigma$ 
of pointed, convex, polyhedral cones 
living in~$\QQ^n$ such 
that for any $\sigma \in \Sigma$ also 
every face $\tau \preccurlyeq \sigma$ 
belongs to~$\Sigma$ and for any two 
$\sigma, \sigma' \in \Sigma$ the intersection 
$\sigma \cap \sigma'$ is a face of both, 
$\sigma$ and $\sigma'$.
Given a fan~$\Sigma$ in~$\ZZ^n$,
the \emph{associated toric variety} $Z$
is built by equivariantly gluing the 
spectra~$Z_\sigma$ 
of the monoid algebras~$\KK[M_\sigma]$ 
of the monoids 
$M_\sigma := \sigma^\vee \cap \ZZ^n$ of 
lattice points inside the dual cones:
$$
Z 
\ = \ 
Z_\Sigma 
\ = \ 
\bigcup_{\sigma \in \Sigma} Z_\sigma,
\qquad
Z_\sigma 
\ = \ 
\Spec \, \KK[M_\sigma],
\qquad
\KK[M_\sigma]
\ = \ 
\bigoplus_{u \in M_\sigma} \KK \chi^u.
$$
The acting torus 
$\TT_Z = \TT^n = \Spec \, \KK[\ZZ^n]$
embeds via $\KK[M_\sigma] \subseteq \KK[\ZZ^n]$ 
canonically into each of the 
$Z_\sigma \subseteq Z$ and 
one takes the neutral element 
$\eins_n \in \TT_Z = \TT^n$ 
as base point $z_0 \in Z$.
The action of $\TT_Z$ on $Z$ then just extends 
the group structure of $\TT_Z \subseteq Z$.
Locally, on the affine open subsets 
$Z_\sigma \subseteq Z$, the $\TT_Z$-action 
is given by its comorphism 
$\chi^u \mapsto \chi^u \otimes \chi^u$.
\end{construction}

\begin{remark}
Let $\Sigma$ be a fan in $\ZZ^n$ and 
$Z$ the associated toric variety.
The cones of $\Sigma$ are in bijection
with the $\TT_Z$-orbits via 
$\sigma \mapsto \TT_Z \cdot z_\sigma$,
where~$z_\sigma$ denotes the common limit 
point for $t \to 0$ of all one-parameter 
groups $t \mapsto (t^{v_1}, \ldots, t^{v_n})$
of $\TT_Z$ with $v \in \ZZ^n$ 
taken from the relative 
interior $\sigma^\circ \subseteq \sigma$.
The dimension of  $\TT_Z \cdot z_\sigma$
equals $n- \dim(\sigma)$.
In particular, the rays
$\varrho_1,\ldots,\varrho_r$ of~$\Sigma$,
that means the one-dimensional cones, 
define the $\TT_Z$-invariant prime divisors
$D_i := \overline{\TT_Z \cdot z_{\varrho_i}}$
of $Z$.
\end{remark}

\emph{Cox's quotient presentation} generalizes 
the classical construction of the projective 
space $\PP_n$ as the quotient of 
$\KK^{n+1} \setminus \{0\}$
by $\KK^*$ acting via scalar multiplication.
It delivers, for instance, any complete
toric variety as a quotient of an open toric 
subset of some affine space 
by a \emph{quasitorus}, that means 
an algebraic group isomorphic to a direct 
product of a torus and a finite abelian group.

\begin{construction}
\label{constr:torcox}
See~\cite{Co}, also~\cite[Sec.~5]{CoLiSc} 
and~\cite[Sec.~2.1.3]{ArDeHaLa}.
Consider a fan~$\Sigma$ in~$\ZZ^n$ and let 
$\varrho_1, \ldots, \varrho_r$
denote its rays.
In each $\varrho_i$ sits a unique
primitive lattice vector~$v_i$, the
generator of the monoid $\varrho_i \cap \ZZ^n$.
The \emph{generator matrix of $\Sigma$}
is the $(n \times r)$-matrix
$$ 
P \ = \ [v_1,\ldots,v_r]
$$
having $v_1, \ldots, v_r$ as its columns,
numbered accordingly to 
$\varrho_1, \ldots, \varrho_r$.
We use the letter~$P$ as well to 
denote the associated linear maps $\ZZ^r \to \ZZ^n$
and $\QQ^r \to \QQ^n$.
As any integral $n \times r$ matrix, $P$ 
defines a homomorphism of tori
$$ 
p \colon \TT^r \ \to \ \TT^n, 
\qquad
t \ \mapsto \ (t^{P_{1*}}, \ldots, t^{P_{n*}}) 
$$
where $t^{P_{i*}} = t_1^{p_{i1}} \cdots t_r^{p_{ir}}$ 
has the $i$-th row of $P = (p_{ij})$ as its exponent 
vector.
Now assume that $v_1, \ldots, v_r$ generate~$\QQ^n$ 
as a vector space, meaning that 
the associated toric variety $Z$ has no 
torus factor.
Consider the orthant $\gamma = \QQ_{\ge 0}^r$ 
and the set
$$ 
\hat \Sigma
\ := \ 
\{\tau \preccurlyeq \gamma; \ 
P(\tau) \subseteq \sigma 
\text{ for some } 
\sigma \in \Sigma\}.
$$
Then $\hat \Sigma$ is a subfan of the fan $\bar \Sigma$ 
of faces of the orthant $\gamma \subseteq \QQ^r$.
Moreover, $P$ sends cones from $\hat \Sigma$ into
cones of $\Sigma$.
Thus, $p \colon \TT^r \to \TT^n$ extends to 
a morphism $p \colon \hat Z \to Z$
of the associated toric varieties. 
We arrive at the following picture
$$ 
\xymatrix{
{\hat Z}
\ar@{}[r]|\subseteq
\ar[d]^{\quot H}_p
&
{\bar Z}
\ar@{}[r]|{:=}
&
{\KK^r}
\\
Z
&
&
}
$$
where $\hat Z \subseteq \bar Z$ is an open 
$\TT^r$-invariant subvariety and 
$H \subseteq \TT^r$ is the kernel of the 
homomorphism $p \colon \TT^r \to \TT^n$ 
of the acting tori.
Being a closed subgroup of a torus, $H$ is 
a quasitorus.
For any cone $\sigma \in \Sigma$, we have
$$ 
p^{-1}(Z_\sigma) 
\ = \ 
\hat Z_{\hat \sigma},
\quad 
\hat \sigma \ := \ \cone(e_i; \ v_i = P(e_i) \in \sigma),
\qquad\quad
p^*\mathcal{O}(Z_{\sigma})
\ = \ 
\mathcal{O}(\hat Z_{\hat \sigma})^H,
$$
where $e_i \in \ZZ^r$ is the $i$-th canonical
basis vector.
Thus, $p \colon \hat Z \to Z$ is an affine
morphism and the pull back functions are 
precisely the $H$-invariants. 
In other words, $p$ is a \emph{good quotient} 
for the $H$-action, as indicated by~``$\quot H$''.
\end{construction}

Generalizing the idea of homogeneous 
coordinates on the projective space, 
one uses Cox's quotient presentation to 
obtain global coordinates on toric varieties.

\begin{remark}
\label{rem:coxcoord}
Let $Z$ be a toric variety with
quotient presentation
$p \colon \hat Z \to Z$ as 
in~\ref{constr:torcox}.
Then every $p$-fiber contains a unique 
closed $H$-orbit.
The presentation in \emph{Cox coordinates}
of a point $x \in Z$ is
$$ 
x \ = \ [z_1,\ldots, z_r],
\qquad
\text{where }
z = (z_1,\ldots, z_r) \in p^{-1}(x)
\text{ with } H \cdot z \subseteq \hat Z
\text{ closed}.
$$
Thus, $[z]$ and $[z']$ represent 
the same point $x \in Z$ if and only if 
$z$ and $z'$ lie in the same closed $H$-orbit 
of~$\hat Z$.
For instance, the points $z_\sigma \in Z$, 
where $\sigma \in \Sigma$, 
are given in Cox coordinates as
$$ 
z_\sigma 
\ = \ 
[\varepsilon_1,\ldots, \varepsilon_r],
\qquad
\varepsilon_i 
\ = \ 
\begin{cases}
0, & P(e_i) \in \sigma,
\\
1, & P(e_i) \not\in \sigma.
\\
\end{cases}
$$
\end{remark}

\begin{remark}
\label{rem:coxcoord2defeqn}
Let $Z$ be a toric variety with
quotient presentation
$p \colon \hat Z \to Z$ as 
in~\ref{constr:torcox}.
Then we obtain an injection
from the closed subvarieties of $X \subseteq Z$ 
to the $H$-invariant closed subvarieties of 
$\bar Z = \KK^r$ via
$$ 
X \ \mapsto \ \bar X \ := \ \overline{p^{-1}(X)} \subseteq \bar Z.
$$
The vanishing ideal 
$I(\bar X) \subseteq \KK[T_1, \ldots,T_r]$
is generated by polynomials 
$g_1, \ldots, g_s$ being $H$-homogeneous in the sense 
that $g_j(h \cdot z) = \chi_j(h)g_j(z)$ holds 
with characters $\chi_j \in \Chi(H)$. 
We call $g_1, \ldots, g_s$ 
\emph{defining equations}
in Cox coordinates for $X \subseteq Z$.
\end{remark}

We turn to Cox rings. Their history starts 
in the 1970s in a geometric setting, 
when Colliot-Th\'{e}l\`ene
and Sansuc introduced the universal 
torsors presenting smooth varieties in
a universal way as quotients~\cites{ColSa}.
In toric geometry, the quotient presentation 
and Cox rings popped up in the 1990s in work 
of Audin~\cite{Au}, Cox~\cite{Co} and others.
In~2000, Hu and Keel observed fundamental 
connections between Cox rings, Mori 
theory and geometric invariant theory~\cite{HuKe}.
As a general introductory reference 
on Cox rings, we mention~\cite{ArDeHaLa}.

We enter the subject. Consider a normal variety~$X$ 
with only constant invertible global functions 
and finitely generated divisor class group $\Cl(X)$.
For a Weil divisor~$D$ on~$X$, denote by 
$\mathcal{O}(D)$ the associated sheaf of sections.
Then the
\emph{Cox sheaf}~$\mathcal{R}$ and the 
\emph{Cox ring} $\mathcal{R}(X)$ of $X$ 
are defined as 
$$ 
\mathcal{R}
\ := \ 
\bigoplus_{[D] \in \Cl(X)}
\mathcal{O}(D),
\qquad\quad
\mathcal{R}(X)
\ := \ 
\Gamma(X,\mathcal{R})
\ = \ 
\bigoplus_{[D] \in \Cl(X)}
\Gamma(X,\mathcal{O}(D)).
$$
Observe that we grade $\mathcal{R}$ 
and $\mathcal{R}(X)$
by divisor classes whereas the homogeneous 
components are defined by divisors. 
If $\Cl(X)$ is torsion free, then this problem 
of well-definedness is solved by just 
regarding~$\mathcal{R}$ as the sheaf of 
multi-section algebras associated with any 
subgroup of the Weil divisor group 
mapping isomorphically onto $\Cl(X)$.
The case of torsion in $\Cl(X)$ requires 
more care, see~\cite{BerHa1,Ha3} 
and~\cite[Sec.~1.1.4]{ArDeHaLa}.

\begin{remark}
\label{ex:toriccr}
See~\cite{Co}, also~\cite[Sec.~5]{CoLiSc} 
and~\cite[Sec.~2.1.3]{ArDeHaLa}.
Let $Z$ be a toric variety without torus factor
and let $D_1, \ldots, D_r$ be the $\TT_Z$-invariant 
prime divisors of $Z$. 
Then the Cox ring of $Z$ and its $\Cl(Z)$-grading 
are given as
$$ 
\mathcal{R}(Z) 
\ = \ 
\KK[T_1,\ldots,T_r],
\qquad\qquad
\deg(T_i) \ =  [D_i] \ \in \ \Cl(Z).
$$
For a more explicit picture, let $Z$ 
arise from a fan $\Sigma$ in $\ZZ^n$. 
Then 
$D_i = \overline{\TT_Z \cdot z_{\varrho_i}} \subseteq Z$
holds with the rays $\varrho_1, \ldots, \varrho_r$ 
of $\Sigma$.
The divisor class group of $Z$ and the divisor 
classes of the $D_i$ are described by
$$ 
\Cl(Z) 
\ = \ 
K
\ := \ 
\ZZ^r / \im(P^*),
\qquad
\qquad
[D_i] \ = \ w_i \ := \ Q(e_i) \ \in \ K,
$$
where we denote
by  $P^*$ the transpose of the generator 
matrix $P$ of~$\Sigma$,
by $Q \colon \ZZ^r \to K$ the projection
and by $e_i \in \ZZ^r$ the $i$-th canonical basis
vector. 
\end{remark}

\begin{remark}
In Construction~\ref{constr:torcox},
the toric variety $Z$ is represented 
as a quotient of $\hat Z \subseteq \KK^r$
by the quasitorus 
$H = \ker(p) \subseteq \TT^r$.
With $K = \Cl(Z) = \ZZ^r / \im(P^*)$ 
from Remark~\ref{ex:toriccr}, 
we can view $H$ also as 
the spectrum of the associated 
group algebra:
$$
H 
\ = \ 
\Spec \, \KK[K],
\qquad \qquad 
\KK[K] 
\ = \ 
\bigoplus_{w \in K} \KK \chi^w.
$$ 
Here, the elements $\chi^w \in \KK[K]$ 
are the characters of $H$ and $w \mapsto \chi^w$ 
defines an isomorphism between $K$ and 
the character group $\Chi(H)$.
Setting $\chi_i := \chi^{w_i}$,
we retrieve the $H$-action from the  
$\Cl(Z)$-grading of the Cox ring $\mathcal{R}(Z)$ 
as
$$ 
h \cdot z 
\ = \ 
(\chi_1(h)z_1, \ldots, \chi_r(h)z_r).
$$  
\end{remark}

In general, the Cox ring $\mathcal{R}(X)$ 
is normal, integral and, as its main 
algebraic feature, it is 
$\Cl(X)$-factorial~\cite{Ar,Ha3}. 
Let us recall the meaning.
A ring $R = \oplus_K R_w$ graded 
by an abelian group~$K$ is \emph{$K$-integral} 
if it has no homogeneous zero divisors.
A nonzero homogeneous non-unit $f \in R$ is 
\emph{$K$-prime} if, whenever $f$ divides a 
product~$gh$ of two homogeneous $g,h \in R$,
then it divides $g$ or $h$.
The ring $R$ is called \emph{$K$-factorial} if
it is $K$-integral and every nonzero homogeneous 
non-unit of $R$ is a product of $K$-primes.
If $\Cl(X)$ is torsion free, then the Cox ring 
admits unique factorization in the usual 
sense, see~\cite[Prop.~8.4]{BerHa1}
and also~\cite[Cor.~1.2]{EKW}.

The bunched ring approach presented
in~\cites{BerHa2,Ha2,ArDeHaLa} uses 
Cox rings to encode algebraic varieties.
The central construction starts with 
a given $K$-factorial ring~$R$ and 
produces varieties $X$ having divisor 
class group $K$ and Cox ring~$R$.
In this article, we will work with 
the following variant being closer to 
toric geometry in the sense that it uses 
fans instead of the bunches of cones 
of~\cites{BerHa2,Ha2,ArDeHaLa}.
Let us fix the necessary notation.
By an~\emph{affine algebra} we mean 
a finitely generated reduced $\KK$-algebra.
If $K$ is an abelian 
group, then we denote by 
$K_\QQ = K \otimes_\ZZ \QQ$ the associated
rational vector space.
Given $w \in K$, we write as well $w$ 
for the element $w \otimes 1 \in K_\QQ$.
Moreover, if $Q \colon K \to K'$ is a 
homomorphism, we denote the associated 
linear map $K_\QQ \to K'_\QQ$ as well 
by~$Q$.

\begin{construction}
\label{constr:explvar}
Let $K$ be a finitely generated abelian group and
$R = \oplus_K R_w$ a $K$-factorial, normal, 
integral, affine $\KK$-algebra with only 
constant homogeneous units.
Suppose that $f_1,\ldots,f_r$ are pairwise 
non-associated $K$-prime generators of~$R$ 
such that any $r-1$ of the degrees 
$w_i := \deg(f_i)$ generate $K$ as a group
and for $\tau_i := \cone(w_j; \ j \ne i)$,
the intersection $\tau_1 \cap \ldots \cap \tau_r$ 
is of full dimension in $K_\QQ$.
Consider the closed embedding
$$ 
\xymatrix{
{\Spec \, R}
 =:  
{\bar X}
\
\ar[rrr]^{x \mapsto (f_1(x),\ldots,f_r(x))}
&&&
\
{\bar Z}
:= 
{\KK^r}.
}
$$
The quasitorus $H = \Spec \, \KK[K]$ acts 
on $\bar Z$ via 
$h \cdot z = (\chi_1(h)z_1, \ldots, \chi_r(h)z_r)$,
where $\chi_i \in \Chi(H)$ is the character 
corresponding to $w_i \in K$.
This action leaves the subvariety 
$\bar X \subseteq \bar Z$
invariant.
Now, consider the degree map
$$
Q \colon \ZZ^r \ \to \ K,
\qquad
e_i \ \mapsto \ w_i.
$$
Let $P$ be an integral $(n \times r)$-matrix, 
the rows of which generate $\ker(Q) \subseteq \ZZ^r$.
Then the assumptions on $f_1, \ldots, f_r$ ensure that 
the columns of $P$ are pairwise different and 
primitive, see~\cite{ArDeHaLa}*{Thm.~2.2.2.6 
and Lemma~2.1.4.1}.
Fix any fan $\Sigma$ having $P$ as generator 
matrix.
The associated toric variety $Z$
and $p \colon \hat Z \to Z$ from
Construction~\ref{constr:torcox}
fit into a commutative diagram
$$ 
\xymatrix{
{\hat X}
\ar@{}[r]|\subseteq
\ar[d]_{\quot H}^p
&
{\hat Z}
\ar[d]^{\quot H}_p
\\
X
\ar@{}[r]|\subseteq
&
Z}
$$
where we set $\hat X := \bar X \cap \hat Z$ and 
$X := \hat X \quot H$ is a normal,
closed subvariety of $Z$.
We speak of $X \subseteq Z$ as an \emph{explicit variety},
refer to $\alpha = (f_1,\ldots,f_r)$ as 
the \emph{embedding system} of $X \subseteq Z$ and 
call any system of $K$-homogeneous generators
$g_1, \ldots, g_s$ of the vanishing ideal
$I(\bar X) \subseteq \KK[T_1,\ldots,T_r]$ 
\emph{defining equations} 
for $X$. 
\end{construction}

\begin{remark}
If, in Construction~\ref{constr:explvar},
the ambient toric variety $Z$ 
is affine (complete, projective),
then the resulting $X$ is affine 
(complete, projective).
\end{remark}

The interface from explicit varieties 
to bunched rings relies on linear 
Gale duality.
Here comes how this concretely works 
in our situation. 

\begin{remark}
\label{rem:fan2bunch}
Notation as in~\ref{constr:torcox}, 
\ref{ex:toriccr} and~\ref{constr:explvar}.
To any explicit variety $X \subseteq Z$, 
we can apply the machinery of bunched 
rings~\cite[Chap.~3]{ArDeHaLa}.
The translation into the latter setting 
runs as follows.
Consider the homomorphisms
$$ 
P \colon \ZZ^r \to \ZZ^n, \ e_i \mapsto v_i,
\qquad \qquad
Q \colon \ZZ^r \to K, \ e_i \mapsto w_i,
$$
where $K = \ZZ^r / \mathrm{im}(P^*)$.
For $\sigma \in \Sigma$, the face
$\hat \sigma \preccurlyeq \gamma \subseteq \QQ^r$ 
of the orthant is generated by 
the $e_i$ with $v_i \in \sigma$.
The \emph{complementary face}
$\hat \sigma^* \preccurlyeq \gamma$ 
of 
$\hat \sigma \preccurlyeq \gamma$
is generated by the~$e_i$ with 
$e_i \not\in \hat \sigma$.
Set
$$ 
\Phi
\ := \ 
\{Q(\hat \sigma^*); 
\ 
\sigma \in \Sigma 
\text{ with } 
X \cap \TT_z \cdot z_\sigma \ne \emptyset
\}.
$$
Then $\Phi$ is a collection of cones in $K_\QQ$ 
with pairwise intersecting relative interiors.
With the system of generators 
$\mathfrak{F} := (f_1,\ldots,f_r)$,
we obtain a \emph{bunched ring}
$(R,\mathfrak{F},\Phi)$
in the sense of~\cite[Def.~3.2.1.2]{ArDeHaLa}.
We have an open inclusion
$$ 
X \ \subseteq \ X(R,\mathfrak{F},\Phi)
$$
into the variety associated with the 
bunched ring~\cite[Def.~3.2.1.3]{ArDeHaLa}
such that the complement of 
$X$ in $X(R,\mathfrak{F},\Phi)$ is of 
codimension at least two.
If $X$ is affine or complete, then the
above inclusion is even an equality.
\end{remark}

Based on this translation, we will 
import several statements on the 
geometry of explicit ($\TT$-)varieties 
in Section~\ref{sec:taohcfirstprop}.
For the moment, we just mention the 
following.

\begin{remark}
See~\cite[Thm.~3.2.1.4]{ArDeHaLa}.
For every explicit variety $X \subseteq Z$, 
the divisor class group and 
the Cox ring of $X$ are given as 
$$ 
\Cl(X) 
\ = \ 
K
\ = \ 
\Cl(Z),
\qquad \qquad 
\mathcal{R}(X) 
\ = \ 
R
\ = \ 
\mathcal{R}(Z) / I(\bar X).
$$
Moreover, $\hat X$ is the relative spectrum of 
the Cox sheaf on $X$, which in turn is given 
as the direct image $\mathcal{R} = p_* \mathcal{O}_{\hat X}$.
Finally, we have the prime divisors
$$
D_i^X
\ = \ 
X \cap D_i^Z
\ \subseteq \
X
$$
induced by the toric prime 
divisors~$D_1^Z, \ldots, D_r^Z$.
Here each $D_i^X$ is determined by the property
$p^*D_i^X = V_{\bar X}(T_i)$.
\end{remark}

Coming embedded into a toric variety,
every explicit variety $X \subseteq Z$ 
inherits the \emph{$A_2$-property}: any 
two points of $X$ admit a common affine 
neighborhood.
The normal $A_2$-varieties are precisely
the normal varieties that are embeddable into 
a toric variety, see~\cite{Wl}. 
An $A_2$-variety $Y$ is \emph{$A_2$-maximal}
if it does not allow open embeddings into 
$A_2$-varieties $Y'$ such that 
$Y' \setminus Y$ is non-empty of 
codimension at least two.
For example, affine and projective varieties
are $A_2$-maximal.

\begin{remark}
See~\cite[Thm.~3.2.1.9]{ArDeHaLa}.
Every $A_2$-maximal variety with only
constant invertible global functions,
finitely generated divisor class group
and finitely generated Cox ring 
can be represented as an explicit 
variety.
\end{remark}

In~\cite{HuKe}, Hu and Keel introduced 
the~\emph{Mori dream spaces} as 
$\QQ$-factorial projective varieties 
with a Mori chamber decomposition
satisfying suitable finiteness properties
which in particular guarantee an 
optimal behavior with respect 
to the minimal model programme.

\begin{remark}
According to~\cite{HuKe}*{Prop.~2.9},
the Mori dream spaces are precisely 
the $\QQ$-factorial projective varieties 
with a finitely generated Cox ring.
In particular, every Mori dream space 
can be represented as an explicit variety.
\end{remark}

\section{Explicit $\TT$-varieties}
\label{sec:expltvar}

We present our method of producing systematically 
explicit $\TT$-varieties,
see Construction~\ref{constr:t-mds}, 
and we formulate basic properties, see 
Proposition~\ref{thm:t-mds-1},
Theorem~\ref{thm:t-mds-2} and 
Proposition~\ref{thm:t-mds-3}.
The proofs of the latter results are given in 
the subsequent section.
We begin by indicating the ideas behind 
Construction~\ref{constr:t-mds}. 
First, take a glance at the following 
naive way to produce varieties with torus 
action sitting inside a given toric variety.

\begin{recipe}
\label{rem:emb2tvar}
Let $Z$ be a toric variety, 
$\TT_Z = \TT' \times \TT$ a splitting  
of the acting torus into closed subtori 
and $Y \subseteq \TT'$ a closed subvariety. 
Consider the closure
$$ 
X \ := \ \overline{Y \times \TT} \ \subseteq \ Z. 
$$
Then the variety $X \subseteq Z$ is invariant under 
the action of $\TT$ on $Z$ and thus 
we obtain an effective algebraic torus action 
$\TT \times X \to X$. 
By construction, we have 
$$ 
X \cap \TT_Z 
\ = \ 
Y \times \TT,
\qquad\qquad
\KK(X)^{\TT}
 \ = \ 
\KK(Y).
$$ 
In particular, $Y$ represents the field of 
$\TT$-invariant rational functions of $X$ 
and thus the projection $\TT_Z \to \TT'$ 
defines a  \emph{rational quotient} 
$X \dasharrow Y$ 
for the $\TT$-action on $X$.
\end{recipe}

So far, Recipe~\ref{rem:emb2tvar} provides no
specifically close relations between the geometry 
of~$X$ and that of its ambient toric variety $Z$.
Nevertheless, we know in advance the rational
quotient $Y$ and, stemming from a subtorus action 
on~$Z$, the $\TT$-action on~$X$ can be 
studied by toric methods.
Moreover, Recipe~\ref{rem:emb2tvar} produces
for instance all projective $\TT$-varieties, 
as we infer from the following.

\begin{remark}
Any $\TT$-variety $X$ that admits an equivariant
embedding into a toric variety $Z$ with $\TT$
acting as a subtorus of $\TT_Z$ can be 
represented as in Recipe~\ref{rem:emb2tvar}.
The techniques from~\cites{Ha1,Ha2} yield 
such equivariant embeddings for $\TT$-varieties 
$X$ with the $A_2$-property provided they 
are $\QQ$-factorial or, more generally, 
divisorial in the sense of~\cite{Bo}, 
or have a Cox sheaf of locally finite 
type.
\end{remark}

Our aim is to bring together the features of 
Recipe~\ref{rem:emb2tvar} with those of 
Construction~\ref{constr:explvar}.
Let us first look at a concrete example, 
indicating the main rules of the subsequent 
construction game and illustrating the 
notation used there.
The example we are going to treat is a well 
known $\KK^*$-surface,
occurring as an important step in resolving 
the $E_6$-singular cubic surface;
see~\cite[Sec.~4]{HaTschi} and, for links 
to various other aspects, also~\cite[p.~522]{ArDeHaLa}.

\begin{example}
\label{ex:recipe}
Our initial data is a projective 
line $Y \subseteq \PP_2$ 
given in homogeneous coordinates
by the following equation:
$$ 
Y \ = \ V(T_0+T_1+T_2) \ \subseteq \ \PP_2.
$$
We regard $\PP_2$ as the toric variety defined 
by the complete fan $\Delta$ with the 
generator matrix
$$ 
B 
\ = \
[u_0,u_1,u_2]
\ = \ 
\left[
\begin{array}{rrr}
-1 & 1 & 0
\\
-1 & 0 & 1
\end{array}
\right] .
$$
Now we start the game that builds up the 
generator matrix $P$ of the fan of the 
prospective ambient 
toric variety~$Z$ of our final $X$.
First produce a matrix
$$ 
P_0 
\ = \ 
[u_{01,},u_{02},u_{11},u_{21}]
\ = \ 
\left[
\begin{array}{rrrr}
-3 & -1 & 3 & 0 
\\
-3 & -1 & 0 & 2  
\end{array}
\right],
$$
the columns $u_{ij}$ of which are positive multiples 
of the columns $u_i$ of~$B$.
Then append a zero column to $P_0$, a block $d$ 
below~$P_0$ and a block~$d'$ below the zero 
column:
$$ 
P 
\ = \ 
[v_{01,},v_{02},v_{11},v_{21},v_1]
\ = \
\left[
\begin{array}{rrrrr}
-3 & -1 & 3 & 0 & 0 
\\
-3 & -1 & 0 & 2 & 0 
\\
-2 & -1 & 1 & 1 & 1
\end{array}
\right]  
\ = \ 
\left[
\begin{array}{cc}
P_0 & 0
\\
d & d'
\end{array}
\right].  
$$
Let $\Sigma$ be any complete fan in $\ZZ^3$
having $P$ as its generator matrix and let $Z$ 
be the associated toric variety.
Then the acting torus $\TT_Z$ of $Z$ splits as
$$ 
\TT_Z 
\ =  \
\TT^3 
\ =  \ 
\TT^2 \ \times \ \KK^*.
$$
Moreover, $Y \cap \TT^2 \subseteq \TT^2$
is the zero set of $1+T_1/T_0+T_2/T_0$.
Proceeding exactly as in Recipe~\ref{rem:emb2tvar}
yields a surface~$X$ coming with an effective 
$\KK^*$-action:
$$ 
X 
\ := \ 
\overline{(Y \cap \TT^2) \times \KK^*}
\ \subseteq \ 
Z.
$$
We have $X \cap \TT^3 = V(1+T_1/T_0+T_2/T_0)$.
Pulling back that equation via the 
homomorphism $p \colon \TT^5 \to \TT^3$ 
given by $P$ leads to the equation 
for $X$ in Cox coordinates:
$$ 
\bar X
\ =  \
V(T_{01}^3T_{02} + T_{11}^3 + T_{21}^2)
\ \subseteq \
\bar Z
\ = \ 
\KK^5,
$$  
where the variables $T_{ij}$ represent columns of 
$P_0$ and the index $ij$ tells us that we have the 
$j$-th repetition of the $i$-th column of $B$, 
scaled by the exponent $l_{ij}$ of~$T_{ij}$. 
\end{example}

\begin{remark}
In Example~\ref{ex:recipe}, we encountered
two explicit varieties in the sense of
Construction~\ref{constr:explvar}:
first, the projective line $Y \subseteq \PP_2$
and second, the $\KK^*$-surface 
$X \subseteq Z$.
In particular, divisor class group and 
Cox ring of $X$ are given as  
\begin{align*}
\Cl(X) 
& \ = \ K  \ = \ \ZZ^5/\mathrm{im}(P^*) \ = \ \ZZ^2,
\\
\mathcal{R}(X)
& \ = \ 
\KK[T_{01},T_{02},T_{11},T_{21},T_1] 
/
\langle 
T_{01}^3T_{02} + T_{11}^3 + T_{21}^2
\rangle.
\end{align*}
Observe that the manipulations on the 
matrix $B$ turned the redundant 
defining relation $T_0+T_1+T_2$ of~$Y$ 
into the serious relation
$T_{01}^3T_{02} + T_{11}^3 + T_{21}^2$,
defining the resulting~$X$.
\end{remark}

We come to the general construction
of explicit $\TT$-varieties.
It starts with a given explicit variety
$Y \subseteq Z_\Delta$ provided
by Construction~\ref{constr:explvar}
and delivers an explicit variety 
$X \subseteq Z_\Sigma$ which is 
invariant under a direct factor 
$\TT \subseteq \TT_\Sigma$ of the 
acting torus 
$\TT_\Sigma \subseteq Z_\Sigma$. 

\begin{construction}
\label{constr:t-mds}
Let $Y \subseteq Z_\Delta$ be an explicit variety
with embedding system $\alpha = (f_0, \ldots, f_r)$.
The defining fan~$\Delta$ of $Z_\Delta$ lives
in some $\ZZ^{t}$ and has a $t \times (r+1)$ 
generator matrix 
$$
B
\ = \ 
[u_{0}, \ldots, u_{r}].
$$
In particular, $\Cl(Y) = \Cl(Z_\Delta)$ equals 
$K_B := \ZZ^{r+1} / \mathrm{im}(B^{*})$
and the Cox ring of~$Y$ equals the 
$K_B$-factorial input ring~$R_Y$
of Construction~\ref{constr:explvar}.
We build up a new matrix from~$B$
and the following data
\begin{itemize}
\item
positive integers $n_{0},\ldots,n_{r}$
and non-negative integers $m,s$ 
with $t+s \le n+m$,
where $n := n_{0} + \ldots + n_{r}$,
\item
for any two $i,j$, where $i = 0, \ldots, r$ and 
$j = 1, \ldots, n_{i}$, a positive integer~$l_{ij}$ 
and a vector $d_{ij} \in \ZZ^{s}$,
\item
for any $k$, where $ 1 \le k \le m$,
a vector $d_{k}' \in \ZZ^{s}$,
\end{itemize}
where, with the multiples $u_{ij} := l_{ij}u_{i} \in \ZZ^{t}$
of the columns of $B$, we require that the vectors  
$$
v_{ij} \ = \ (u_{ij},d_{ij}) \ \in \ \ZZ^{t+s},
\qquad\qquad
v_{k} \ = \ (0,d_{k}') \ \in \ \ZZ^{t+s}
$$  
are all primitive, any two of them are distinct
and altogether they generate $\QQ^{t+s}$ as a 
vector space.  
Store the $v_{ij}$ and $v_k$ as columns in a 
$(t+s)\times(n+m)$ matrix
$$
P
\ = \
[v_{ij},v_k] 
\ = \ 
\begin{bmatrix}
u_{01} 
& 
\ldots 
& 
u_{0n_{0}} 
& 
\ldots 
&  u_{r1} 
& 
\ldots 
& 
u_{rn_{r}}
&
0
&
\ldots 
&
0
\\
d_{01} 
& 
\ldots 
& 
d_{0n_{0}} 
& 
\ldots &  
d_{r1} 
& 
\ldots 
& 
d_{rn_{r}}
&
d_{1}'
&
\ldots 
&
d_{m}'
\end{bmatrix}.
$$
Now, let $\Sigma$ be any fan in $\ZZ^{t+s}$ 
having $P$ as its generator matrix
and denote by~$Z_{\Sigma}$ 
the associated toric variety. 
Then we obtain a commutative diagram
$$ 
\xymatrix{
X 
\ar@{}[r]|\subseteq
\ar@{-->}[d]
&
Z_{\Sigma}
\ar@{-->}[d]
\\
Y
\ar@{}[r]|\subseteq 
&
Z_{\Delta}
}
$$ 
where the rational map $Z_{\Sigma} \dasharrow Z_{\Delta}$
is given by the projection $\TT^{t} \times \TT^{s} \to \TT^{t}$
of the respective acting tori $\TT_\Sigma = \TT^{t} \times \TT^{s}$ 
and $\TT_\Delta = \TT^{t}$ and we define
$$
X 
\ := \ 
X(\alpha,P,\Sigma) 
\ := \ 
\overline{(Y \cap \TT^{t}) \times \TT^{s}}
\ \subseteq \ 
Z_{\Sigma}.
$$ 
Then $X \subseteq Z_\Sigma$ is invariant under the action 
of the subtorus 
$\TT = \{\mathds{1}_t\} \times \TT^{s}$ of 
the acting torus $\TT_\Sigma = \TT^{t} \times \TT^{s}$
of $Z_\Sigma$.
Moreover, set
$$ 
T_{i}^{l_{i}} := T_{i1}^{l_{i1}} \cdots T_{in_{i}}^{l_{in_{i}}} \in \KK[T_{ij},S_k],
\qquad 
K_P
:= 
\ZZ^{n+m}/ \mathrm{im}(P^{*})
= 
\Cl(Z_\Sigma).
$$ 
Let $h_1, \ldots, h_q$ be defining 
equations of~$Y$, that means $K_B$-homogeneous
generators for the ideal of relations between 
$f_0, \ldots, f_r$.
Consider the factor ring 
$$ 
R(\alpha,P) 
\ := \
\KK[T_{ij},S_{k}] 
/ 
\bangle{
h_1(T_{0}^{l_{0}}, \ldots, T_{r}^{l_{r}}), 
\ldots, 
h_q(T_{0}^{l_{0}}, \ldots, T_{r}^{l_{r}})}
$$ 
and denote by $Q_P \colon \ZZ^{n+m} \to K_P$ the projection.
We turn $R(\alpha,P)$ into a $K_P$-graded algebra via 
$$
\deg(T_{ij}) \ := \ w_{ij} \ := \ Q_P(e_{ij}),
\qquad
\deg(T_k) \ := \ w_k \ := \ Q_P(e_k),
$$   
where $e_{ij},e_{k} \in \ZZ^{n+m}$ are the canonical basis vectors.
Observe that we have a unique homomorphism of graded algebras 
$R_Y \to R(\alpha,P)$ sending $f_{i}$ to $T_{i}^{l_{i}}$.
\end{construction}

\begin{remark}
\label{rem:expltx2explv}
For $t=n=0$, the above construction 
yields the usual construction of a 
toric variety from a fan.
Moreover, for $s = m = 0$ and $n = r+1$,
we arrive at Construction~\ref{constr:explvar}.
\end{remark}

We come to the first basic property of
Construction~\ref{constr:t-mds}.
Note that in concrete cases the assumptions 
of this proposition on $R(\alpha,P)$ and the 
$T_{ij}$ made below can be checked 
algorithmically via absolute factorization; 
see~\cite{HaKeLa1}*{Rem.~3.8}.

\begin{proposition}
\label{thm:t-mds-1}
Let $X = X(\alpha,P,\Sigma)$ 
arise from Construction~\ref{constr:t-mds}.
If $R(\alpha,P)$ is a $K_P$-integral affine 
algebra with only constant homogeneous 
units and the~$T_{ij}$ define 
pairwise non-associated 
$K_P$-primes in $R(\alpha,P)$,
then $X \subseteq Z_\Sigma$ is an
explicit variety.
\end{proposition}

\begin{definition}
\label{def:expltvar}
By an \emph{explicit $\TT$-variety}
$X \subseteq Z$ 
we mean a variety $X = X(\alpha,P,\Sigma)$ 
in $Z = Z_\Sigma$ together with the 
action of
$\TT = \{\mathds{1}_t\} \times \TT^{s}$
arising from 
Construction~\ref{constr:t-mds}
such that the assumptions of 
Proposition~\ref{thm:t-mds-1}
are satisfied.
\end{definition}

\begin{corollary}
Let $X \subseteq Z$ be an explicit 
$\TT$-variety.
Then $X$ is a normal variety
with only constant invertible global 
functions.
Moreover, dimension, complexity, 
divisor class group and Cox ring of $X$ 
are given by
$$ 
\dim(X) = s + \dim(Y),
\quad
c(X) = \dim(Y),
\quad
\Cl(X) \ = \ K_P,
\quad
\mathcal{R}(X) \ = \ R(\alpha,P).
$$
\end{corollary}

We say that a $\TT$-variety $X'$ 
admits a \emph{presentation as an explicit 
$\TT$-variety} if there is a 
$\TT$-equivariant isomorphism $X' \to X$ 
with some explicit $\TT$-variety
$X \subseteq Z$.

\begin{theorem}
\label{thm:t-mds-2}
Let $X$ be an $A_2$-maximal 
$\TT$-variety having only constant 
invertible global functions, 
finitely generated divisor class group and 
finitely generated Cox ring.
Then $X$ admits a presentation as an
explicit $\TT$-variety.
\end{theorem}

In the rest of the section, we  
discuss the geometry of the torus 
action of an explicit $\TT$-variety 
$X \subseteq Z$, aiming for 
a suitable quotient.
First, we continue 
Example~\ref{ex:recipe}.

\begin{example}
\label{ex:exprops2}
Consider again the explicit 
$\KK^*$-surface $X \subseteq Z$ 
from~\ref{ex:recipe}.
An important source of information is 
the location of the columns of~$P$ 
over those of~$B$ with respect to the 
projection $\pr \colon \ZZ^3 \to \ZZ^2$ 
onto the first two coordinates:
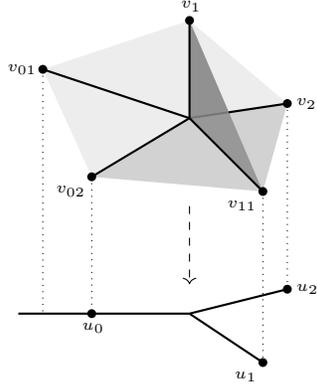
\begin{center}
\begin{tikzpicture}[scale=1.3]

\coordinate (v) at (0,0);
\coordinate (v01) at (-1.5,.5);
\coordinate (v02) at (-1,-.6);
\coordinate (v11) at (.75,-.75);
\coordinate (v11a) at (.75,-.76);
\coordinate (v21) at (1,.15);
\coordinate (v1) at (0,1);
\coordinate (u) at (0,-2);
\coordinate (u0) at (-1,-2);
\coordinate (u01) at (-1.5,-2);
\coordinate (u01+) at (-1.75,-2);
\coordinate (u1) at (.75,-2.5);
\coordinate (u2) at (1,-1.75);
\coordinate (s) at   ($(v)!.4!(v21)$);

\fill[gray!70,opacity=0.2] (v) -- (v1) -- (v01) -- (v02) -- cycle; 
\fill[gray!90,opacity=1.2] (v) -- (v1) -- (v11) -- cycle; 
\fill[gray!70,opacity=0.2] (v) -- (v1) -- (v21) -- cycle; 
\fill[gray!70,opacity=0.5] (v) -- (v02) -- (v11) -- cycle; 
\fill[gray!70,opacity=0.5] (v) -- (v11) -- (v21) -- cycle;

\draw[->,dashed,domain=-0.9:-1.7,variable=\x] plot ({0},{\x});

\path[fill, color=black] (v01) circle (0.3ex) node[left]{\tiny $v_{01} \!$};
\path[fill, color=black] (v02) circle (0.3ex) node[below left]{\tiny $v_{02} \!$};
\path[fill, color=black] (v11) circle (0.3ex) node[below left]{\tiny $v_{11} \!$};
\path[fill, color=black] (v21) circle (0.3ex) node[right]{\tiny $v_{21} \!$};
\path[fill, color=black] (v1) circle (0.3ex) node[above]{\tiny $v_1 \!$};
\draw[thick] (v)--(v01);
\draw[thick] (v)--(v02);
\draw[thick] (v)--(v11);
\draw[thick,opacity=0.2] (v)--(s);
\draw[thick] (s)--(v21);
\draw[thick] (v)--(v1);

\path[fill, color=black] (u0) circle (0.3ex) node[below]{\tiny $u_{0} \!$};
\path[fill, color=black] (u1) circle (0.3ex) node[below left]{\tiny $u_{1} \!$};
\path[fill, color=black] (u2) circle (0.3ex) node[right]{\tiny $u_{2} \!$};
\draw[thick] (u)--(u01+);
\draw[thick] (u)--(u1);
\draw[thick] (u)--(u2);

\draw[dotted] (v01)--(u01);
\draw[dotted] (v02)--(u0);
\draw[dotted] (v11)--(u1);
\draw[dotted] (v21)--(u2);
 
\end{tikzpicture}   
\end{center}
Each column~$v_{ij}$ projects into the ray through $u_i$ 
and $v_1$ lies in the kernel of the projection.
The rays $\varrho_{ij}$ through $v_{ij}$ and $\varrho_1$ 
through $v_1$ define prime divisors $D_{ij}^Z$
and $E_{1}^Z$ of $Z$, respectively.
Cutting down to $X$ gives us prime divisors
$$ 
D_{ij}^X \ := \ X \cap D_{ij}^Z \ \subseteq \ X,
\qquad\qquad
E_1^X \ := \ X \cap E_1^Z \ \subseteq \ X,
$$
where the basic reason for primality is that the divisors
are given in Cox coordinates by $K$-prime ideals; 
for instance $D_{01}$ is defined by 
$\langle T_{01}, \, T_{11}^3+T_{21}^2 \rangle$.
We are interested in the isotropy groups.
Recall that $\KK^*$ acts on $X$ as the subtorus
$$ 
\TT
\ := \ 
\{\mathds{1}_2\} \times \KK^*
\ \subseteq \ 
\TT^3
\ = \ 
\TT_Z.
$$
In particular, the isotropy groups of the 
$\TT$-action are constant along the 
$\TT_Z$-orbits.
Consider the kernel 
$L = \{0\} \times \ZZ$ of
$\pr \colon \ZZ^3 \to \ZZ^2$.
Then~\cite[Prop.~2.1.4.2]{ArDeHaLa} yields
for any $\sigma \in \Sigma$ that the isotropy 
group of $\TT$ at $z_\sigma \in Z$ has character 
group
$$ 
\Chi(\TT_{z_\sigma}) 
\ \cong \ 
(L \cap \lin(\sigma))
\, \oplus \,
(\pr(\lin(\sigma)) \cap \ZZ^2) / (\pr(\lin(\sigma) \cap \ZZ^3),
$$
where $\lin(\sigma) \subseteq \QQ^3$ denotes the 
$\QQ$-linear hull.  
Looking at $\sigma = \varrho_{ij}$, we see 
that the isotropy group $\TT_x$ of the general 
point $x \in D_{ij}^X$ is cyclic of order 
$l_{ij}$, where $l_{ij}$ is the exponent
of $T_{ij}$ in the defining relation of $X$,
that means
$$ 
l_{01} = 3, \qquad
l_{02} = 1, \qquad
l_{11} = 3, \qquad
l_{21} = 2.
$$
Moreover, the curve $E_1^X$ consists of 
fixed points of the $\TT$-action and there
are two isolated fixed points, forming the 
intersections of $X$ with the toric orbits 
$\TT_Z \cdot z_\sigma$ 
for $\sigma = \cone(v_{01},v_{02})$ 
and $\sigma = \cone(v_{02},v_{11},v_{21})$,
respectively.
In particular, we see that
$$ 
X_0  
\ = \ 
X 
\cap 
\left( 
\TT_Z 
\cup 
\TT_Z \cdot z_{\varrho_{01}}
\cup 
\TT_Z \cdot z_{\varrho_{02}}
\cup 
\TT_Z \cdot z_{\varrho_{11}}
\cup 
\TT_Z \cdot z_{\varrho_{21}}
\right)
\ \subseteq \
X
$$
is the open subset of $X$ consisting of 
all points $x \in X$ having finite isotropy
group~$\TT_x$.
The projection $\pr \colon  \ZZ^3 \to \ZZ^2$
defines a rational quotient $Z \dasharrow \PP_2$ 
for the $\TT$-action
inducing a rational quotient $X \dasharrow Y$ 
which in turn is defined on $X_0 \subseteq X$
and gives a surjective morphism $X_0 \to Y$,
where $Y = \PP_1$.
\end{example}

Before entering the general case, 
let us give the precise definitions
of the necessary concepts of quotients.
For the moment, $X$ may be any variety 
with an action of an algebraic group 
$G$.
As already indicated, a
\emph{rational quotient} for the 
$G$-variety~$X$ is a dominant 
rational map $\pi \colon X \dasharrow Y$ 
such that $\pi^*\KK(Y) = \KK(X)^G$
holds.
A~\emph{representative} of a rational 
quotient $\pi \colon X \dasharrow Y$ 
is a surjective morphism $W \to V$ 
representing~$\pi$ 
on a non-empty open $G$-invariant 
subset $W \subseteq X$ and an open 
subset $V \subseteq Y$.
By results of Rosenlicht, rational
quotients always exist and admit a 
representative having $G$-orbits 
as its fibers~\cite{Ro}. 
  
Behind Construction~\ref{constr:t-mds} 
there is a specific rational quotient, 
the maximal orbit quotient.
Recall that a \emph{geometric quotient}
of a $\TT$-variety $X$ is a good quotient
$X \to Y$ having precisely the $\TT$-orbits
as its fibers. 
Moreover, for any $\TT$-variety $X$, we denote 
by $X_0 \subseteq X$ the open subset 
consisting of all points $x \in X$ 
with finite isotropy group.

\begin{definition}
\label{def:maxorbquot}
A \emph{maximal orbit quotient}
for a $\TT$-variety $X$ 
is a rational quotient 
$\pi \colon X \dasharrow Y$  
admitting a representative
$\psi \colon W \to V$
and prime divisors 
$C_0, \ldots, C_r$ on $Y$ 
such that the following 
properties are satisfied:
\begin{enumerate}
\item
one has $W \subseteq X_0$
and the complements 
$X_0 \setminus W \subseteq X_0$ 
and 
$Y \setminus V \subseteq Y$, 
both are of codimension at least two,
\item
for every $i = 0,\ldots, r$, the 
inverse image 
$\psi^{-1}(C_i) \subseteq W$ 
is a union of prime divisors 
$D_{i1}, \ldots, D_{in_i} \subseteq W$,
\item
all $\TT$-invariant prime divisors of $X_0$ 
with non-trivial generic isotropy group
occur among the $D_{ij}$,
\item 
every sequence $J = (j_0, \ldots, j_r)$
with $1 \le j_i \le n_i$
defines a geometric quotient $\psi \colon W_J \to V$ 
for the $\TT$-action,
where $W_J := W \setminus \cup_{j \ne j_i} D_{ij}$. 
\end{enumerate}
We call $C_0, \ldots, C_r \subseteq Y$ 
a collection of \emph{doubling divisors}
for $\pi \colon X \dasharrow Y$.
The closure of any~$D_{ij}$ in~$X$ 
is a $\TT$-invariant prime divisor 
of $X$, again denoted by~$D_{ij}$
and called a \emph{multiple divisor}.
Moreover, we denote by
$E_1, \ldots, E_m$ 
the prime divisors in 
the complement $X \setminus X_0$
and call them the \emph{boundary divisors}.
Finally, we call $\psi \colon W \to V$ 
a \emph{big representative} for 
$\pi \colon X \dasharrow Y$.
\end{definition}

\begin{example}
We continue~\ref{ex:recipe}
and~\ref{ex:exprops2}.
The rational quotient $X \dasharrow Y$ 
arising from the projection 
$\TT^3 \to \TT^2$ of tori is a maximal 
orbit quotient.
The intersection points 
$c_i$ of $Y \subseteq \PP_2$ with 
the coordinate axes $V(T_i)$
yield a collection of doubling divisors, 
the multiple divisors over $c_i$ are the 
$D_{ij}^X$ and the (only) 
boundary divisor is~$E_1^X$.
\end{example}

\begin{remark}
\label{rem:doubdiv}
Observe that Definition~\ref{def:maxorbquot}
leaves some freedom for choosing 
the doubling divisors $C_0, \ldots, C_r$. 
Some divisors necessarily appear:
the images of divisors with non-trivial 
finite generic isotropy group and the 
images of invariant divisors which cannot 
be separated by $\psi$.
Beyond those, we are free to choose 
further doubling divisors $C_i$, which
then means to insert $D_{i1}$ accordingly.
\end{remark}

\begin{remark}
\label{rem:maxorbquotunique}
Let $\pi \colon X \dasharrow Y$ and 
$\pi' \colon X \dasharrow Y'$ 
be maximal orbit quotients for a 
$\TT$-variety $X$. 
Then there are open subsets 
$U \subseteq Y$ and $U' \subseteq Y'$
having complements of codimension 
at least two and an isomorphism
$U \to U'$
which sends any collection of 
doubling divisors for $\pi$ 
to a collection of doubling 
divisors of~$\pi'$.
\end{remark}

\begin{proposition}
\label{thm:t-mds-3}
Let $X \subseteq Z_\Sigma$ be an explicit 
$\TT$-variety.
Let $Z_\Sigma^1 \subseteq Z_\Sigma$ be the 
union of $\TT_\Sigma$ and all toric orbits 
$\TT_\Sigma \cdot z_{\varrho_{ij}}$ 
and $Z_\Delta^1 \subseteq Z_\Delta$
the union of all toric orbits of
codimension at most one.
Then, for $X_1 := X \cap Z_\Sigma^1$
and $Y_1 := Y \cap Z_\Delta^1$, we have 
a commutative diagram
$$ 
\xymatrix{
X_1 
\ar@{}[r]|\subseteq
\ar[d]
&
Z_\Sigma^1
\ar[d]
\\
Y_1 
\ar@{}[r]|\subseteq
&
Z_\Delta^1
}
$$ 
where the downwards maps are maximal orbit
quotients for the action of $\TT$.
Denoting by $D_{ij}^\Sigma$ and 
$D_k^\Sigma$ the toric prime divisors of $Z_\Sigma$
corresponding to the rays 
$\varrho_{ij} = \cone(v_{ij})$
and $\varrho_k = \cone(v_k)$, 
we obtain the multiple divisors and the 
boundary divisors of $X$ as
$$ 
D_{ij}^X \ = \ X \cap D_{ij}^\Sigma,
\qquad\qquad
E_k^X \ = \ X \cap D_k^\Sigma.
$$
The generic isotropy group of $E_k^X$ is 
a one-dimensional torus and the generic 
isotropy group of $D_{ij}^X$ is finite 
of order $l_{ij}$.
The doubling divisors are the intersections 
of $C_i = Y \cap D_i^\Delta$ with the toric 
prime divisors of $D_i^\Delta \subseteq Z_\Delta$.
\end{proposition}

We briefly discuss relations to 
polyhedral divisors.
First we have the following recipe 
to convert explicit $\TT$-varieties 
into the setting of polyhedral 
divisors.

\begin{remark}
\label{rem:expl2poldiv}
Given an explicit $\TT$-variety 
$X \subseteq Z_\Sigma$,
we indicate how to obtain a describing 
divisorial fan in the sense 
of~\cite{AlHa,AlHaSu}.
First follow~\cite{AlHa}*{Sec.~11}.
For every $\sigma \in \Sigma$, let 
$\Delta_\sigma$ be the fan in $\ZZ^t$
obtained as the coarsest common 
refinement of the projections 
$\pr(\tau) \subseteq \QQ^t$ of 
all faces 
$\tau \preccurlyeq \sigma \subseteq \QQ^{s+t}$.
The toric variety associated with~$\Delta_\sigma$ 
is the normalized Chow quotient 
$Z_\sigma \chquot \TT$, 
see~\cite{KaStZe,CrMl}.
Let $Y'_\sigma$ be the normalization 
of the closure of the image of $X \cap \TT_Z$ 
in~$Z_\sigma \chquot \TT$ and write $D_{\sigma,\varrho}$ 
for the pull back of the toric prime divisor 
$D_\varrho$ of $Z_\sigma \chquot \TT$ to~$Y'_\sigma$.
Then 
$$ 
\mathcal{D}_\sigma 
\ := \ 
\sum A_\varrho \otimes D_{\sigma,\varrho},
\qquad
\mathrm{A}_\varrho 
\ := \ 
\sigma \cap \pr^{-1}(v_\varrho)
\ \subseteq \
\QQ^{s+t}
$$
defines a polyhedral divisor on $Y'_\sigma$
describing the $\TT$-action on 
$X_\sigma := X \cap Z_\sigma$.
Now, follow the proof of~\cite{AlHaSu}*{Thm.~5.6}
to bring the local pictures together.
Choose projective closures 
$Y'_\sigma \subseteq Y_\sigma''$ and,
via resolving indeterminacies of the 
birational maps between the $Y_\sigma''$
induced by those between the $X_\sigma$,
construct a normal projective variety~$Y''$ 
dominating birationally all the $Y_\sigma''$.
Pulling back the $\mathcal{D}_\sigma$
to~$Y''$ yields the desired divisorial fan 
describing the $\TT$-action on $X$.
\end{remark}

\begin{remark}
In general, maximal orbit quotient and 
Chow quotient of a $\TT$-variety differ
from each other.
For example, let $\TT = \KK^*$ act
on~$X = \KK^4$ via
$$
t \cdot z \ = \ (t^{-1}z_1,t^{-1}z_2,tz_3,tz_4).
$$
Working for instance in terms of fans
we see that in this particular case
we obtain a maximal orbit quotient just 
by taking the good quotient
$$ 
\pi \colon X \ \to \ X \quot \TT, 
\qquad\qquad
z \ \mapsto \ (z_1z_3,z_1z_4,z_2z_3,z_2z_4),
$$
where $X \quot \TT = \{w \in \KK^4; \ w_1w_4=w_2w_3\}$,
and the canonical map $X \chquot \TT \to X \quot \TT$ 
from the Chow quotient onto the good quotient 
resolves the singularity $0 \in X \quot \TT$. 
\end{remark}

\section{Proofs to Section~\ref{sec:expltvar}}

Here we prove the statements made in
Construction~\ref{constr:t-mds}, 
Proposition~\ref{thm:t-mds-1},
Theorem~\ref{thm:t-mds-2} and 
Proposition~\ref{thm:t-mds-3}.
We will make use of Bechtold's normality 
criterion~\cite{Be2}*{Cor.~6}; for 
convenience we give a direct proof here.

\begin{proposition}
\label{prop:benjamin2}
Let $K$ be a finitely generated abelian 
group,
$R$ a $K$-factorial affine $\KK$-algebra 
with only constant $K$-homogeneous units
and $f_1, \ldots, f_r$ a system of 
pairwise non-associated $K$-prime 
generators for~$R$.
If any $r-1$ of the $\deg(f_i)$ generate 
$K$ as a group, then $R$ is integral 
and normal.
\end{proposition}

\begin{proof}
The $K$-grading of $R$ defines an action 
of the quasitorus $H := \Spec \, \KK[K]$ 
on $\bar X := \Spec \, R$
such that the homogeneous elements $f \in R$ 
of degree $w \in K$ are precisely the 
functions on $\bar X$ which are homogeneous 
with respect to $\chi^w \in \Chi(H)$.
Set $g_i := \prod_{j \ne i} f_j$
and consider the $H$-invariant open 
subset
$$
\qquad\qquad 
\hat X 
\ := \ 
\bar X_{g_1} \cup \ldots \cup \bar X_{g_r}
\ \subseteq \ 
\bar X.
$$
Since the $f_i$ are pairwise non-associated 
$K$-primes, $\hat X$ has complement 
of codimension at least two in $\bar X$.
According to~\cite{Be}*{Thm.~1.3},
each $\bar X_{g_i}/H$ is factorial and 
hence normal.
By~\cite[Prop.~1.2.2.8]{ArDeHaLa}, 
the $H$-action on $\bar X_{g_i}$ is free.
Thus, Luna's slice theorem~\cite{Lu}*{Thm.~III.1} 
tells us that the quotient map 
$\bar X_{g_i} \to \bar X_{g_i}/H$ 
is an \'{e}tale $H$-principal bundle.
As \'{e}tale morphisms preserve normality, 
see~\cite{Mil}*{Prop.~8.1},
we conclude that each $\bar X_{g_i}$ 
and hence $\hat X$ is normal.
Now, observe
$$ 
R 
\ = \
\mathcal{O}(\bar X)
\ \subseteq \
\mathcal{O}(\hat X).
$$ 
We claim that the last inclusion is in fact 
an equality.
Let $g \in \mathcal{O}(\hat X)$ be 
an $H$-homogeneous function.
Since $g$ is a regular homogeneous 
function on $\bar X_{g_1}$, we have 
$g = g'/g_1^{l}$ with a homogeneous function 
$g' \in R$. 
Using $K$-factoriality, we find 
pairwise non-associated $K$-primes $p_i, f_j \in R$,
where $f_2, \ldots, f_r$ are the generators fixed
before, such that
$$
g 
\  = \
\frac{p_1^{\nu_1} \cdots p_s^{\nu_s}}{f_2^{\mu_2} \cdots f_r^{\mu_r}}.
$$
Since $g$ is regular on the normal variety 
$\hat X$ and $\bar X \setminus \hat X$ 
is of codimension at least two in $\bar X$, 
we must have $\mu_2 = \ldots = \mu_r = 0$. 
Consequently, $g \in R$ holds.
Now, every regular function on $\hat X$ is a sum of 
$K$-homogeneous ones and thus extends to a regular 
function on $\bar X$.
In particular, $R = \mathcal{O}(\hat X)$ is 
normal.

To see that $R$ is integral, we have to 
show that $\bar X = \Spec \, R$ is 
irreducible. 
Due to normality, the irreducible components of 
$\bar X$ coincide with its 
connected components 
$\bar X_1, \ldots, \bar X_k$.
Indeed, if two distinct irreducible components
have a common point, then the corresponding 
local ring has zero divisors, contradicting normality.   
The assumption that $R$ is $K$-integral
means on the geometric side that~$H$ 
permutes transitively the $\bar X_i$.
So, we can choose $h_i \in H$ with 
$\bar X_i = h_i\bar X_1$ and  
a non-trivial character $\chi \in \mathbb{X}(H)$
vanishing along the stabilizer of $\bar X_1$. 
Then, setting $f(z) := \chi(h_i)$ 
for $z \in \bar X_i$ defines a  
homogeneous unit on $\bar X$,
which is non-constant as soon as $k > 1$ holds.
We conclude $k=1$ and thus $\bar X$ is 
irreducible. 
\end{proof}

\begin{proof}[Proof of Construction~\ref{constr:t-mds},
Proposition~\ref{thm:t-mds-1}
and Proposition~\ref{thm:t-mds-3}]
The generator matrix $B$ of the 
fan~$\Delta$ and the generator matrix~$P$ 
of the fan $\Sigma$ fit into the commutative 
diagram 
$$ 
\xymatrix{
{\ZZ^{n+m}}
\ar[r]^{P}
\ar[d]_{A}
&
{\ZZ^{t+s}}
\ar[d]
\\
{\ZZ^{r+1}}
\ar[r]_{B}
&
{\ZZ^{t}}
}
$$
where the lifting $A \colon \ZZ^{n+m} \to \ZZ^{r+1}$ 
of the projection $\ZZ^{t+s} \to \ZZ^{t}$
sends the canonical basis vectors 
$e_{ij} \in \ZZ^{n+m}$
to $l_{ij} e_{i} \in \ZZ^{r+1}$ 
and $e_{k} \in \ZZ^{n+m}$ 
to $0 \in \ZZ^{r+1}$.
Dualizing leads to a commutative ladder 
of abelian groups with exact rows
$$ 
\xymatrix{
0 
\ar@{<-}[r]
&
K_P\ar@{<-}[r]^{Q_P}
\ar@{<-}[d]_{\imath}
&
{\ZZ^{n+m}}
\ar@{<-}[r]^{P^{*}}
\ar@{<-}[d]_{A^{*}}
&
{\ZZ^{t+s}}
\ar@{<-}[r]
\ar@{<-}[d]
&
0
\\
0 
\ar@{<-}[r]
&
K_B
\ar@{<-}[r]_{Q_B}
&
{\ZZ^{r+1}}
\ar@{<-}[r]_{B^{*}}
&
{\ZZ^{t}}
\ar@{<-}[r]
&
0
}
$$

We validate Construction~\ref{constr:t-mds}.
According to Construction~\ref{constr:explvar}, 
the canonical basis vector $e_i \in \ZZ^{r+1}$ 
is sent by $Q_B$ to $\deg(f_i) \in K_B$.
Thus, the induced map 
$\imath \colon K_B \rightarrow K_P$
sends $\deg(f_{i}) \in K_B$ 
$Q_P(l_{i1}e_{i1} + \ldots + l_{in_{i}}e_{in_{i}}) \in K_P$.
Define a $K_B$-grading on the 
polynomial ring $\KK[F_0,\ldots,F_r]$
by $\deg(F_i) = \deg(f_i)$ and 
a $K_B$-grading on $\KK[T_{ij},S_k]$
by $\deg(T_{ij}) = Q_P(e_{ij})$ and $\deg(S_k) = Q_P(e_k)$.
Then the homomorphism
$$ 
\KK[F_0,\ldots,F_r] \ \to \ \KK[T_{ij},S_k],
\qquad 
f_i \ \mapsto \ T_{i}^{l_{i}}
$$
sends homogeneous elements of degree $w \in K_B$
to homogeneous elements of degree $\imath(w) \in K_P$.
In particular, the defining relations
$h_j(T_{0}^{l_{0}}, \ldots, T_{r}^{l_{r}})$ are 
$K_P$-homogeneous,
the $K_P$-grading of $R(\alpha,P)$ is well 
defined and, moreover, we have 
the induced homomorphism 
of the graded algebras
$R_Y \to R(\alpha,P)$ 
sending $f_{i}$ to~$T_{i}^{l_{i}}$
as desired.

We turn to Proposition~\ref{thm:t-mds-1}.
Let $\bar{Y} \subseteq \KK^{r+1}$ 
and $\bar X \subseteq \KK^{n+m}$ 
denote the closures of the inverse images 
of $Y \cap \TT^{t}$ 
and $X \cap \TT^{t+s}$ under the homomorphisms of tori
$b \colon \TT^{r+1} \to \TT^{t}$ and 
$p \colon \TT^{n+m} \to \TT^{t+s}$
defined by $B$ and $P$ respectively.
Observe that
$\bar Y = \Spec \, R_Y$
holds.
With the quasitori 
$H_{Y} := \Spec\, \KK[K_B]$
and $H_{X} := \Spec\, \KK[K_P]$
and the homomorphism of tori 
$a \colon  \TT^{n+m}  \to \TT^{r+1}$ defined 
by $A$, we have a commutative diagram
$$ 
\xymatrix{
\bar X \cap \TT^{n+m}
\ar[r]^{/H_{X}}_p
\ar[d]_a
&
X \cap \ \TT^{t+s}
\ar[d]^{/\TT^{s}}
\\
\bar{Y} \cap \TT^{r+1}
\ar[r]^{/H_{Y}}_b
&
Y \cap \TT^{t}.
}
$$
Consider the product $f \in R_Y$
over all the generators $f_i$ of $R_Y$ 
and the product $g \in R(\alpha,P)$ over all the 
generators 
$T_{ij}$ and $S_k$ of $R(\alpha,P)$. 
Then, using the above diagram, we see
$$ 
\left((R_Y)_f \right)^{H_Y}
\ \cong \ 
a^* \left((R_Y)_f \right)^{H_Y}
\ = \ 
\left(\left(R(\alpha,P)_g\right)^{H_X}\right)^{\TT^{s}}.
$$
Since the left hand side ring is factorial, also 
the right hand side ring is so.
By assumption, $R(\alpha,P)$ is 
$K_P$-integral and the generators 
$T_{ij}$ are $K_P$-prime. 
Using~\cite{Be}*{Thm.~1.3},
we see that $R(\alpha,P)$ is factorially 
$K_P$-graded
and Proposition~\ref{prop:benjamin2}
shows that $R(\alpha,P)$ is integral 
and normal.
Consequently, we are in the setting 
of Construction~\ref{constr:explvar}
which establishes
Proposition~\ref{thm:t-mds-1}.

Finally, we show Proposition~\ref{thm:t-mds-3}.
First note that  $Z_\Sigma^1 \to Z_\Delta^1$
defines a maximal orbit quotient of 
the $\TT^{s}$-action on $Z_\Sigma$.
The toric prime divisors of $Z_\Sigma^1$ 
cut down to the prime divisors 
$D_{ij}^X$ and $D_k^X$ of $X_1$ 
and those of~$Z_\Delta^1$ to the 
prime divisors $C_i$ of~$Y_1$. 
Thus, we can infer the statements on 
the isotropy groups 
from~\cite{ArDeHaLa}*{Prop.~2.1.4.2}
and conclude that $X_1 \to Y_1$ is a 
big representative of a maximal orbit 
quotient of the $\TT^{s}$-variety~$X$.
\end{proof}

We come to the proof of Theorem~\ref{thm:t-mds-2}.
The task is to provide for any abstractly 
given $A_2$-maximal $\TT$-variety $X$ with only 
constant invertible global functions, 
finitely generated divisor class group $\Cl(X)$ 
and finitely Cox ring $\mathcal{R}(X)$ 
a presentation as an explicit $\TT$-variety 
$X \subseteq Z$.
This runs via general Cox ring theory.
Let us recall the necessary background.
Mimicking Cox's quotient presentation~\ref{constr:torcox}, 
one looks at 
$$
\bar X \ = \ \Spec \, \mathcal{R}(X),
\qquad\qquad
H \ = \ \Spec \, \KK[\Cl(X)],
$$ 
the \emph{total coordinate space} and the 
\emph{characteristic quasitorus} of $X$.
Then $H$ acts on~$\bar X$, where 
this action is defined via its comorphism,
sending a homogeneous  element $f \in \mathcal{R}(X)$ 
of degree $[D]$ to the element 
$\chi^{[D]} \otimes f$ of $\KK[\Cl(X)] \otimes \mathcal{R}(X)$. 
Moreover, $X$ can be reconstructed as a good quotient
$$ 
\xymatrix{
{\Spec_X \, \mathcal{R}}
\ar@{}[r]|{\quad =}
&
{\hat X}
\ar@{}[r]|\subseteq
\ar[d]_{\quot H}^p
&
{\bar X}
\ar@{}[r]|{= \qquad}
&
\Spec \, \mathcal{R}(X)
\\
&
X
&
&
}
$$
Here the relative spectrum $\hat X$ of 
the Cox sheaf $\mathcal{R}$ is called 
the \emph{characteristic space} over $X$.
It is an open $H$-invariant subset of $\bar X$ 
and the complement $\bar X \setminus \hat X$ 
is small in the sense that it is 
of codimension at least two in $\bar X$;
see~\cite{ArDeHaLa}*{Sec.~1.6.1}.

We will deal with canonical sections, 
which in the context of Cox rings means 
the following.
For any effective representative~$D$ 
of a class $[D] \in \Cl(X)$, 
there is, up to scalars, a unique 
$f  \in \mathcal{R}(X)_{[D]}$ 
with $\div(f) = p^*D$ on $\hat X$.
In this situation, we call $f$ a 
\emph{canonical section} of $D$
and write $f = 1_D$.
A canonical section $1_D$ is a $\Cl(X)$-prime
element of $\mathcal{R}(X)$ if and only 
if $D$ is a prime divisor on $X$.
See~\cite{ArDeHaLa}*{Prop.~1.5.3.5 and Lemma~1.5.3.6} 
for the full details.

\begin{proof}[Proof of Theorem~\ref{thm:t-mds-2}]
Write for short $K := \Cl(X)$ and 
$R := \mathcal{R}(X)$.
In a first step, we lift the action of the 
torus $\TT$ to the total coordinate space 
$\bar X$.
Consider the characteristic space 
$p \colon \hat X \to X$ over $X$.
By~\cite{ArDeHaLa}*{Thm.~4.2.3.2}, 
there are a $\TT$-action on $\hat X$ and 
a positive integer~$b$ such that
for all $t \in \TT$, $h \in H$ and 
$x \in \hat X$, we have
$$
t \cdot h \cdot x \ = h \cdot t \cdot x,
\qquad\qquad
p(t \cdot x) \ = \ t^b \cdot p(x).
$$
Since $\hat X \subseteq \bar X$ has 
a small complement and~$\bar X$ is 
normal, the $\TT$-action on
$\hat X$ extends to~$\bar X$.
The fact that the actions of $\TT$ and 
$H$ commute means that we have an action 
of $\TT \times H$ on $\bar X$.
Thus, the $K$-grading of $R$ refines to 
a $(M \times K)$-grading for $M = \Chi(\TT)$.
As $M$ is torsion free, \cite{Be}*{Thm.~1.5}
yields that $R$ is $(M \times K)$-factorial
and $(M \times K)$-primality coincides with  
$K$-primality in $R$.

Now, let $\mathfrak{F} = (f_1, \ldots, f_q)$ 
be a system of pairwise non-associated 
$K$-prime generators of $R$ such that 
every $\TT$-invariant prime divisor of 
$X$ having non-trivial generic isotropy 
group has a canonical section among 
the~$f_i$.
Then, as mentioned before, the $f_i$ 
are $(M \times K)$-prime, and thus
in particular $(M \times K)$-homogeneous.
Similarly as in Construction~\ref{constr:explvar},
we obtain a $(\TT \times H)$-equivariant 
closed embedding
$$ 
\xymatrix{
{\Spec \, R}
 =:  
{\bar X}
\
\ar[rrr]^{x \mapsto (f_1(x),\ldots,f_q(x))}
&&&
\
{\bar Z}
:= 
{\KK^q}.
}
$$
Let $Q \colon \ZZ^q \to K$,  $e_i \mapsto \deg(f_i)$
be the degree map of the $K$-grading.
Then, in the language 
of~\cite{ArDeHaLa}*{Thm.~3.1.4.4},
we have a maximal bunch of orbit cones 
$$ 
\Phi 
\ = \ 
\{
Q(\gamma_x); \ 
x \in \hat X \text{ with } 
H \cdot x \subseteq \hat X \text{ closed}
\},
\qquad
\gamma_x 
\ = \ 
\cone(e_i; \ f_i(x) \ne 0).
$$
Moreover,~\cite{ArDeHaLa}*{Props.~3.2.2.2, 3.2.2.5}
ensure that we obtain a bunched ring 
$(R,\mathfrak{F},\Phi)$ in the sense 
of~\cite{ArDeHaLa}*{Def.~3.2.1.1}.
Now we reverse the translation performed
in Remark~\ref{rem:fan2bunch}.
Fix a $q' \times q$ matrix $P$, 
the rows of which form a 
lattice basis for $\ker(Q)$.
For a face $\gamma_0 \preccurlyeq \gamma$
of the orthant $\gamma = \QQ_{\ge 0}^q$,
let $\gamma_0^* \preccurlyeq \gamma$ be the 
complementary face. 
Then 
$$ 
\Sigma_X 
\ := \ 
\{
P(\gamma_x^*); \ 
x \in \hat X \text{ with } 
H \cdot x \subseteq \hat X \text{ closed}
\}
$$
is a set of cones intersecting in common
faces; see~\cite{ArDeHaLa}*{Thm.~2.2.1.14}.
Let $\Sigma$ be any fan in $\ZZ^{q'}$ 
such that $\Sigma_X \subseteq \Sigma$ holds.
Consider the associated toric variety  
$Z = Z_\Sigma$ and Cox's quotient 
presentation $\hat Z \to Z$.
We will build up the following commutative 
diagram
$$ 
\xymatrix{
{\hat X}
\ar@{}[r]|\subseteq
\ar[d]_{\quot H}^p
&
\hat{Z}
\ar[d]^{\quot H}
\\
X 
\ar[r]
\ar@{-->}[d]_\pi
&
Z
\ar@{-->}[d]
\\
Y
\ar[r]
&
Z_\Delta.
}
$$ 
By $A_2$-maximality of $X$ and the 
choice of $\Sigma$, we have  
$\hat X = \bar X \cap \hat Z$
and the induced morphism $X \to Z$ 
of quotient spaces is a closed 
embedding.
Moreover, $X \to Z$ is 
$\TT$-equivariant, where~$\TT$ 
acts on $Z$ as a subtorus of 
$\TT_Z \subseteq Z$.
Choose a splitting 
$\TT_Z = \TT^t \times \TT$. 
Accordingly, the lattice hosting 
the fan~$\Sigma$ splits as 
$\ZZ^{q'} = \ZZ^t \times \ZZ^s$.
Let~$\Delta$ be the fan in $\ZZ^t$ 
consisting of the zero cone and the 
projections of the rays of~$\Sigma$.
Then the projection $\TT_Z \to \TT^t$
of acting tori defines the rational
map $Z \dasharrow Z_\Delta$.
Defining $Y \subseteq Z_\Delta$ to 
be the closure of the image of 
$X \cap \TT_Z$, we complete the 
commutative diagram.

We investigate the shape of the 
generator matrices $B$ of $\Delta$
and $P$ of $\Sigma$.
Numbering its columns as 
$u_0, \ldots, u_r$,
we turn $B$ into a $t \times (r+1)$
matrix.
For every $i = 0, \ldots, r$, 
denote by $v_{i1}, \ldots, v_{in_i}$
the columns of $P$ such that the ray 
$\varrho_{ij} = \cone(v_{ij})$
projects onto $\cone(u_i)$.
Moreover, denote by $v_1, \ldots, v_m$ the 
columns of $P$ such that the ray 
$\varrho_k = \cone(v_k)$ 
lies in the kernel of the 
projection $\QQ^t \times \QQ^s \to \QQ^t$.
Then $P$ is a $(n+m) \times (t+s)$ 
matrix, where $n = n_0+\ldots+n_r$.
Consider the toric prime divisors 
$D_{ij}^Z \subseteq Z$ and  $E_{k}^Z \subseteq Z$
corresponding to the rays $\varrho_{ij}$ and 
$\varrho_k$ respectively. 
Computing the generic isotropy groups $\TT_x$ 
of these divisors according 
to~\cite{ArDeHaLa}*{Prop.~2.1.4.2}, 
we see that the $E_{k}^Z$ are the boundary divisors
of the $\TT$-action and 
that the~$v_{ij}$ have a non-trivial 
$\ZZ^t$-part being the $l_{ij}$-fold multiple 
of the primitive generator $u_i \in \ZZ^t$.
Thus, $B$ and $P$ look as in 
Construction~\ref{constr:t-mds}.

We claim that the dashed arrows are 
maximal orbit quotients for the 
$\TT$-actions on~$Z$ and $X$ respectively.
Consider the union $Z^1 \subseteq Z$ 
of $\TT_Z$ and all toric orbits 
$\TT_Z \cdot z_{\varrho_{ij}}$.
Then $Z^1 \subseteq Z_0$ is an open 
subset with complement of
codimension at least two in the 
set $Z_0 \subseteq Z$ consisting of 
all points $z \in Z$ with finite 
isotropy group $\TT_z$.
Let~$C_i$ be the prime divisor of $Z_\Delta$
corresponding to  $\cone(u_i)$,
where $i = 0, \ldots, r$.
Then $C_0, \ldots, C_r$ serve as doubling 
divisors and $Z^1 \to Z_\Delta$ is a 
big representative 
for the rational quotient $Z \dasharrow Z_\Delta$,
where Property~\ref{def:maxorbquot}~(iv)
is due to~\cite{ArDeHaLa}*{Cor.~2.3.1.7}.
Cutting down to $X$ gives an open 
subset $X^1 = X \cap Z^1$ of~$X_0$
and a morphism $\psi \colon X^1 \to Y$,
which inherits the properties of 
a big representative from $Z^1 \to Z_\Delta$.
In particular, $Y$ is normal.
A collection of doubling divisors 
is given by $C_i = Y \cap D_i$, 
where $D_0, \ldots, D_r \subseteq Z_\Delta$ 
are the invariant prime divisors of $Z_\Delta$.
Observe that each~$C_i$ is prime, 
because it is the image of 
$X \cap D$ for any $\TT_Z$-invariant 
prime divisor $D \subseteq Z$ 
lying over $D_i \subseteq Z_\Delta$.

To conclude the proof, we still have to show 
that $Y \subseteq Z_\Delta$ is an explicit variety,
that means that $1_{C_0}, \ldots, 1_{C_r}$ generate
the Cox ring $\mathcal{R}(Y)$.
Consider the commutative diagram 
$$ 
\xymatrix{
{\ZZ^{n+m}}
\ar[r]^{P}
\ar[d]_{A}
&
{\ZZ^{t+s}}
\ar[d]
\\
{\ZZ^{r+1}}
\ar[r]_{B}
&
{\ZZ^{t}}
}
$$
where the matrix $A \colon \ZZ^{n+m} \to \ZZ^{r+1}$ 
defines the homomorphism of 
$a \colon \TT^{n+m} \to \TT^{r+1}$
which in turn uniquely extends 
to the monomial map
$$ 
a \colon \KK^{n+m} \ \to \ \KK^{r+1},
\qquad\qquad
(z,w) \ \mapsto \ (z_0^{l_0}, \ldots, z_r^{l_r}).
$$
Note that $a$ is the good quotient for 
the action of the quasitorus $\ker(a)$ on 
$\KK^{n+m}$.
The total coordinate space 
$\bar X \subseteq \KK^{n+m}$ 
is invariant and thus maps
onto a closed normal subvariety 
$\bar Y \subseteq \KK^{r+1}$.
Moreover, $\bar Y$ inherits from $X$ 
the property that the coordinate
functions of $\KK^{r+1}$
define pairwise non-associated 
elements on $\mathcal{O}(\bar Y)$.
By construction, $\bar Y  \cap \TT^{r+1}$ 
dominates $Y \subseteq Z_\Delta$.
Thus, using \cite[Lemmas~3.4.1.7, 3.4.1.9 and 
Cor.~3.4.1.6]{ArDeHaLa}, we see that 
$\mathcal{O}(\bar Y)$ is the 
Cox ring of $Y$.
\end{proof}

As a consequence of the above proofs we 
retrieve~\cite[Thm.~1.2]{HaSu} 
for the special case of $\TT$-varieties 
with finitely generated Cox ring.

\begin{corollary}
\label{cor:RYinRX}
Let $X$ be a $\TT$-variety with finitely 
generated Cox ring $\mathcal{R}(X)$.
Then~$X$ admits a maximal orbit quotient 
$\pi \colon X \dasharrow Y$
and a collection $C_0, \ldots, C_r$ of doubling 
divisors such that we have an isomorphism
of $\Cl(X)$-graded rings
$$ 
\mathcal{R}(X) 
\ \cong \ 
\mathcal{R}(Y)[T_{ij},S_k] / \bangle{T_i^{l_i}-U_i; \ i = 0, \ldots, r},
$$
where $T_{ij}, S_k \in \mathcal{R}(X)$ and 
$U_i \in \mathcal{R}(Y)$ are canonical 
sections of the multiple divisors~$D_{ij}$, 
boundary divisors~$E_k$ and doubling divisors~$C_i$
respectively and the 
$\Cl(X)$-grading on the right hand side 
is given 
by 
$$ 
\deg(U_i) 
= 
[l_{i1}D_{i1} + \ldots + l_{in_i}D_{in_i}],
\qquad
\deg(D_{ij}) 
= 
[D_{ij}],
\qquad
\deg(S_k) 
= 
[E_k].
$$
Moreover, we have $\bar Y = \bar X \quot H_{X,Y}$,
where the quasitorus $H_{X,Y} \subseteq \TT^{n+m}$ 
is the kernel of the homomorphism of tori 
$\TT^{n+m} \to \TT^{r+1}$ sending $(t,s)$ 
to $(t_0^{l_0}, \ldots,t_r^{l_r})$.
\end{corollary}

\section{First properties and examples}
\label{sec:taohcfirstprop}

We discuss basic geometric properties 
of explicit $\TT$-varieties.
First we provide a collection of general 
statements directly imported 
from~\cite{ArDeHaLa}*{Chap.~3}, 
concerning singularities, 
the Picard group and various cones of 
divisor classes.
Then we present more specific statements
involving the $\TT$-action.
The second part of the section is
devoted to examples.
We indicate how to apply the results 
in practice by means of a concrete 
(new) example, 
we show how the construction of rational
$\TT$-varieties of complexity one 
from~\cite{HaHe,HaWr}
fits into the framework of 
explicit $\TT$-varieties
and finally, we present the Grassmannian
$\Gr(2,n)$ with its maximal torus
action as an explicit $\TT$-variety.

When we speak about an explicit $\TT$-variety 
$X \subseteq Z$ or, more specifically,
about an explicit $\TT$-variety
$X(\alpha,P,\Sigma)$ in $Z_\Sigma$, 
then we allow ourselves to make free 
use of the notation introduced in 
Construction~\ref{constr:t-mds}.
Recall from Remark~\ref{rem:expltx2explv} 
that the case of a trivial 
$\TT$-action, that means the explicit 
varieties from Construction~\ref{constr:explvar},
is included via $s=m=0$ and $n=r+1$.

\begin{remark}
Let $X \subseteq Z$
be an explicit $\TT$-variety.
The total coordinate spaces
$\bar{X}$ and $\bar{Z}$, 
that means the spectra of the 
Cox rings~$\mathcal{R}(X)$ 
and~$\mathcal{R}(Z)$, 
are given as 
$$
\bar{X}
\ := \ 
\bar{X}(\alpha,P) 
\ := \
V(
h_1(T_{0}^{l_{0}}, \ldots, T_{r}^{l_{r}}), 
\ldots, 
h_q(T_{0}^{l_{0}}, \ldots, T_{r}^{l_{r}})) 
\ \subseteq \
\KK^{n+m}
\ =: \ 
\bar{Z}.
$$
The embedding $\bar X \subseteq \bar Z$ 
is equivariant with respect to
the actions of the characteristic 
quasitorus
$H = \Spec \, \KK[K_P]$
defined by the gradings of 
$\mathcal{R}(X)$ 
and $\mathcal{R}(Z)$ by 
$K_P = \Cl(X) = \Cl(Z)$.
Moreover, we have a commutative diagram
$$ 
\xymatrix{
{\hat{X}}
\ar@{}[r]|\subseteq
\ar[d]_{\quot H}
& 
{\hat{Z}}
\ar[d]^{\quot H}
\\
X
\ar@{}[r]|\subseteq
&
Z
}
$$
where $\hat{Z} \to Z$ is Cox's quotient
presentation~\ref{constr:torcox} 
and $\hat{X} = \bar{X} \cap \hat{Z}$
holds. 
The good quotients $\hat{X} \to X$
and $\hat{Z} \to Z$ are
the characteristic spaces over $X$ and $Z$,
respectively.
\end{remark}

Every explicit $\TT$-variety $X \subseteq Z$ 
inherits a decomposition into locally closed 
subsets by cutting down the toric orbit 
decomposition of~$Z$.
Generalizing well-known basic facts of toric 
geometry, one can express several geometric  
properties of~$X$ in terms of this inherited
decomposition.
Let us introduce the necessary notation
for precise statements.

\begin{definition}
\label{def:Xcone}
Let $X \subseteq Z$
be an explicit $\TT$-variety.
Set $\gamma := \QQ^{n+m}_{\ge 0}$.
An \emph{$\bar{X}$-face} is a face 
$\gamma_0 \preccurlyeq \QQ^{n+m}$ such that 
the complementary face 
$\gamma_0^* \preccurlyeq \gamma$
satisfies
$$
\KK^{n+m}
\ \supseteq \
\bar{X}(\gamma_0)
\ := \ 
\bar{X}
\ \cap \ 
\TT^{n+m} \cdot z_{\gamma_0^*} 
\ \ne \
\emptyset.
$$
For $\sigma \in \Sigma$ and the 
\emph{corresponding} face 
$\gamma_0 \preccurlyeq \gamma$,
that means the face 
with $P(\gamma_0^*) = \sigma$,
consider the intersection of 
$X$ and the associated toric 
orbit of $Z = Z_\Sigma$:
$$ 
X(\gamma_0) 
\ := \ 
X(\sigma) 
\ := \ 
X 
\cap 
\TT^{t+s} \cdot z_\sigma
\ \subseteq \
Z.
$$
We call $\sigma \in \Sigma$
an \emph{$X$-cone} and 
$\gamma_0 \preccurlyeq \gamma$ 
an \emph{$X$-face}
if $X(\gamma_0) = X(\sigma)$
is non-empty.
Moreover, we denote 
$$ 
\rlv(X)
\ := \ 
\{
\gamma_0 \preccurlyeq \gamma; 
\ \gamma \text{ is an $X$-face}
\}.
$$
Finally, we call the subsets 
$X(\gamma_0) \subseteq X$,
where $\gamma_0$ is an $X$-face, 
the \emph{pieces}
of the explicit $\TT$-variety 
$X \subseteq Z$.
\end{definition}

\begin{remark}
Let $X \subseteq Z$ be an explicit 
$\TT$-variety.
Then every piece 
$X(\gamma_0) \subseteq X$
is locally closed and 
$X$ is the disjoint union of 
its pieces:
$$ 
X 
\ = \ 
\bigsqcup_{\gamma_0 \in \rlv(X)} X(\gamma_0).
$$
Moreover, $\gamma_0 \preccurlyeq \gamma$
is an $X$-face if and only if it is
an $\bar X$-face and we have 
$P(\gamma_0^*) \in \Sigma$.
If $\gamma_0 \preccurlyeq \gamma$
is an $X$-face, then  $\bar{X}(\gamma_0)$
maps onto $X(\gamma_0)$.
\end{remark}

We describe basic local properties in terms 
of the pieces.
Consider for the moment any normal 
variety~$X$.
A point $x \in X$ is \emph{factorial} 
if every Weil divisor of $X$ is Cartier 
near~$x$.
Moreover, $x \in X$ is 
\emph{$\QQ$-factorial} if for every 
Weil divisor of $X$ some nonzero
multiple is Cartier near $x$.

\begin{proposition}
\label{prop:Qfactchar}
Let $X \subseteq Z$ be an explicit 
$\TT$-variety.
Consider an $X$-face $\gamma_0 \preccurlyeq \gamma$ 
and $\sigma = P(\gamma_0^*) \in \Sigma$.
Then the following statements 
are equivalent.
\begin{enumerate}
\item
The piece $X(\sigma)$ consists of 
$\QQ$-factorial points of $X$.
\item
The cone $\sigma$ is simplicial.
\item
The cone 
$Q(\gamma_0) \subseteq K_\QQ$
is of full dimension.
\end{enumerate}
\end{proposition}

\begin{proof}
Translate via Remark~\ref{rem:fan2bunch}  and 
apply~\cite{ArDeHaLa}*{3.3.1.8 and~3.3.1.12}.
\end{proof}

\begin{proposition}
\label{prop:smoothchar}
Let $X \subseteq Z$ be an explicit 
$\TT$-variety.
Consider an $X$-face $\gamma_0 \preccurlyeq \gamma$ 
and $\sigma = P(\gamma_0^*) \in \Sigma$.
Then the following statements are equivalent.
\begin{enumerate}
\item
The piece $X(\sigma)$ consists of 
factorial points of $X$.
\item
The cone $\sigma$ is regular.
\item
The set $Q(\gamma_0 \cap \ZZ^{n+m})$ 
generates $K$ as a group.
\end{enumerate}
Moreover, $X(\sigma)$ consists of smooth points of $X$ 
if and only if 
one of the above statements holds and
$\bar{X}(\gamma_0)$ consists of smooth 
points of $\bar{X}$.
\end{proposition}

\begin{proof}
Translate via Remark~\ref{rem:fan2bunch} and 
apply~\cite{ArDeHaLa}*{3.3.1.8, 3.3.1.9 and~3.3.1.12}.
\end{proof}

We turn to the Picard group and the various cones 
of divisor classes.
For a subset $A \subseteq V$ of a $\QQ$-vector space,
we denote by $\lin(A) \subseteq V$ its $\QQ$-linear hull.

\begin{proposition}
\label{prop:PicandCones}
Let $X \subseteq Z$
be an explicit $\TT$-variety.
Then, in $K_P = \Cl(X)$, the Picard group of 
$X$ is given by 
$$ 
\Pic(X)
\ = \ 
\bigcap_{\gamma_0 \in \rlv(X)} Q(\lin(\gamma_0) \cap \ZZ^{n+m}).
$$
Moreover, in $(K_P)_\QQ = \Cl_\QQ(X)$, the cones of 
effective, movable, semiample and ample
divisor classes are given by
$$ 
\Eff(X) 
\ = \ 
Q(\gamma), 
\qquad\qquad
\Mov(X)
\ = \ 
\bigcap_{\gamma_0 \preccurlyeq \gamma \text{ facet}}
Q(\gamma_0),
$$
$$ 
\SAmple(X)
\ = \ 
\bigcap_{\gamma_0 \in \rlv(X)} Q(\gamma_0),
\qquad\qquad
\Ample(X)
\ = \ 
\bigcap_{\gamma_0 \in \rlv(X)} Q(\gamma_0)^\circ.
$$
\end{proposition}

\begin{proof}
Translate via Remark~\ref{rem:fan2bunch} and 
apply~\cite{ArDeHaLa}*{Cor.~3.3.1.6 and Prop.~3.3.2.9}.
\end{proof}

\begin{remark}
\label{rem:Sigmau}
Let $X = X(\alpha,P,\Sigma)$ in $Z = Z_\Sigma$
be an explicit $\TT$-variety.
If the fan~$\Sigma$ is the normal fan of a 
polytope in $\QQ^{t+s}$, then $Z$ and hence 
$X$ are projective.
Conversely, if $X$ is projective,
choose any class $u \in \Ample(X)$, 
an element $e \in \QQ^{n+m}$ with 
$Q(e) = u$ and consider the polytope
$$
B(u)
\ = \ 
(P^*)^{-1}(Q^{-1}(u) \cap \gamma) - e)
\ \subseteq \ 
\QQ^{t+s}.
$$
Then, with the normal fan $\Sigma(u)$ of 
$B(u)$, we have $X = X(\alpha,P,\Sigma(u))$,
whereas the toric ambient variety 
$Z(u)$ associated with $\Sigma(u)$ may 
differ from the original $Z = Z_\Sigma$.
Note that in terms of the faces 
$\gamma_0 \preccurlyeq \gamma$, the normal 
fan is given as 
$$
\Sigma(u)
\ = \ 
\left\{P(\gamma_0^*);
\ \gamma_0 \preccurlyeq \gamma 
\text{ with } 
u \in Q(\gamma_0)^\circ\right\}.
$$
\end{remark}

We indicate, in our setting, the fundamental 
connection between geometric invariant theory 
and Mori theory found by Hu and Keel~\cite[Thm.~2.3]{HuKe};
we refer to~\cite{ArDeHaLa}*{Sections~3.1.2 and~3.3.4} 
for additional background.

\begin{remark}
\label{rem:Moridecomp}
Let $X =  X(\alpha,P,\Sigma(u))$ in $Z = Z_\Sigma$ 
be a projective explicit $\TT$-variety.
For every $u \in \Eff(X)$, denote by
$\Gamma(u)$ the collection of $\bar X$-faces 
$\gamma_0 \preceq \gamma$ with 
$u \in Q(\gamma_0)$
and define a convex polyhedral cone 
$$ 
\lambda_u 
\ = \ 
\bigcap_{\gamma_0 \in \Gamma(u)} Q(\gamma_0)
\ \subseteq \ 
K_\QQ 
\ = \ 
\Cl_\QQ(X).
$$
The cones $\lambda_u$ form a fan 
$\Lambda$ subdividing $\Eff(X)$.
The fan $\Lambda$ maintains the 
geometric invariant theory of 
the characteristic quasitorus action
on~$\bar X$ in the sense that
the sets $X^{ss}(u)$ of semistable 
points associated with the 
characters 
$\chi^u \in \Chi(H)$ satisfy
$$ 
\bar X^{ss}(u) 
\ \subseteq \ 
\bar X^{ss}(u')
\quad \iff \quad
\lambda_u 
\ \succcurlyeq \ 
\lambda_{u'} .
$$
Observe that if one of the conditions 
holds, then we have an induced morphism 
$X_u \to X_{u'}$ of the associated quotients 
by~$H$.
Now, look at the $u \in \Mov(X)^\circ$.
These define projective explicit 
$\TT$-varieties.
More precisely, we have
$$
X_u \ = \ X(\alpha,P,\Sigma(u)),
\qquad\qquad
\hat X_u \ = \ \bar X^{ss}(u).
$$
Each $X_u$ has $\lambda_u$ as its 
semiample cone and is 
$\QQ$-factorial if and only if 
$\dim(\lambda_u)$ equals $\dim(K_\QQ)$.
Moreover, $\Lambda$ reflects 
the Mori equivalence:
the birational map $X_u \dasharrow X_u'$ 
is an isomorphism if and only if 
$u, u' \in \lambda^\circ$ holds 
for some $\lambda \in \Lambda$.
\end{remark}

Now we discuss more specific properties of 
explicit $\TT$-varieties $X \subseteq Z$ 
involving in particular the torus action. 
Recall that $\TT = \TT^s$, being a factor
of $\TT_Z = \TT^t \times \TT^s$, acts on
$Z$ and leaves $X \subseteq Z$ invariant.
Moreover, the projection $\TT_Z \to \TT^t$
defines the maximal orbit quotient 
$Z \dasharrow Z_\Delta$ for the $\TT$-action 
on $Z$ and by restricting we obtain a 
maximal orbit quotient $\pi \colon X \dasharrow Y$
for the $\TT$-action on $X$.

\begin{proposition}
\label{prop:isogroups}
Let $X \subseteq Z$ be an explicit 
$\TT$-variety.
Let $L \subseteq \ZZ^{t+s}$ be the kernel 
of the projection 
$\pr \colon \ZZ^{t+s} \to \ZZ^t$.
Then, for every $X$-cone $\sigma \in \Sigma$ 
and every $x \in X(\sigma)$, the 
isotropy group $\TT_x$ satisfies
$$ 
\Chi(\TT_x) 
\ \cong \ 
(L \cap \lin(\sigma))
\, \oplus \,
(\pr(\lin(\sigma)) \cap \ZZ^{t}) / (\pr(\lin(\sigma) \cap \ZZ^{t+s}).
$$
\end{proposition}

\begin{proof}
From~\cite[Prop.~2.1.4.2]{ArDeHaLa} we infer
the formula for the isotropy group 
of $\TT \subseteq \TT_Z$ at the point
$z_\sigma \in Z$.
Since the isotropy groups of the $\TT$-action 
are constant along the toric orbits,
this is all we need.
\end{proof}

\begin{proposition}
\label{prop:ci}
Let $X \subseteq Z$ be an explicit $\TT$-variety.
Suppose that the Cox ring presentation
$
\mathcal{R}(Y) = \KK[f_1,\ldots,f_r] 
/ 
\bangle{h_1 ,\ldots, h_q}
$ 
is a complete intersection.
Then, 
with $h_u' :=  h_u(T_0^{l_0}, \ldots, T_r^{l_r})$,
also the Cox ring presentation
$$
\mathcal{R}(X) 
\ = \ 
\KK[T_{ij},S_k] 
/ 
\bangle{h_1', \ldots, h_q'}
$$
is a complete intersection.
Moreover, in the latter case, 
the canonical divisor class of 
$X$ is given by 
$$ 
\mathcal{K}_X
\ = \ 
- 
\sum_{i=0}^r \sum_{j = 1}^{n_i} \deg(T_{ij}) 
- 
\sum_{k=1}^m \deg(S_k) 
+ 
\sum_{u=1}^q \deg(h_u')
\ \in \
K_P
\ = \ 
\Cl(X).
$$
In particular, with the canonical divisor 
class $\mathcal{K}_Y \in K_B = \Cl(Y)$
and the maximal orbit quotient 
$\pi \colon X \dasharrow Y$, we have 
$$ 
\mathcal{K}_X - \pi^*(\mathcal{K}_Y) 
\ = \ 
\sum_{i=0}^r \sum_{j = 1}^{n_i} (l_{ij}-1)\deg(T_{ij}) 
-
\sum_{k=1}^m \deg(S_k).
$$
\end{proposition}

\begin{proof}
The second and third statement follow 
from~\cite{ArDeHaLa}*{Prop.~3.3.3.2}.
The first one is seen via a simple dimension computation:
\begin{eqnarray*}
\dim(\bar{X})
&  = &
\dim(X) + \rk(\Cl(X))
\\
& = &
s + \dim(Y) + \rk(\Cl(X))
\\
& = &
s + \dim(\bar{Y}) - \rk(\Cl(Y)) + \rk(\Cl(X))
\\
& = &
s + (r+1 - q) - (r+1 - t) + (n+m - t - s)
\\
& = & 
n+m - q.
\end{eqnarray*}
\end{proof}

For the next observation, note that 
in Construction~\ref{constr:t-mds},
we may remove successively all maximal 
cones from the fan $\Sigma$ that are 
not $X$-cones. 
The result is a minimal fan $\Sigma$ 
defining still the initial $X$.
We call $Z = Z_\Sigma$ in this case the
\emph{minimal ambient toric variety}
of $X$.

\begin{proposition}
\label{prop:genorbclos}
Let $X \subseteq Z$ be an explicit $\TT$-variety
and assume that $Z$ is the minimal toric 
ambient variety of~$X$.
Let $L \subseteq \ZZ^{t+s}$ be the kernel 
of the projection $\ZZ^{t+s} \to \ZZ^t$.
\begin{enumerate}
\item
The normalization of the general $\TT$-orbit 
closure of $X$ is the toric variety defined 
by the fan $\Sigma_L$ in $L$, where
$$ 
\Sigma_L
\ := \ 
\{\tau; \ \tau \preccurlyeq (\sigma \cap L_\QQ), \ \sigma \in \Sigma\}.
$$
\item
If the maximal orbit quotient $\pi \colon X \dasharrow Y$ 
is a morphism, then $\Sigma_L$ is a subfan of $\Sigma$.
\end{enumerate}
\end{proposition}

\begin{proof}
As $Z$ is the minimal toric embedding,
the general $\TT$-orbit closure of $X$ 
equals the general $\TT$-orbit closure of $Z$.
This reduces the problem to standard toric geometry. 
\end{proof}

\begin{corollary}
\label{cor:upperrhobound1}
Assumptions as in Proposition~\ref{prop:genorbclos}.
If $X$ is complete and $\Sigma_L$ is a subfan 
of $\Sigma$, then we have 
$$
\rk(\Cl(X)) - \rk(\Cl(Y))
\ > \ 
n-r-1.
$$
\end{corollary}

\begin{proof}
According to Proposition~\ref{prop:genorbclos},
the general $\TT$-orbit closure of $X$
has divisor class group of rank $m-s > 0$. 
Thus, the assertion follows from 
$$
\rk(\Cl(X))
\ = \ 
n+m-t-s,
\qquad\qquad
\rk(\Cl(Y))
\ = \ 
r +1-t. 
$$
\end{proof}

We come to the announced example discussions.
First, we use Construction~\ref{constr:t-mds}
to produce a concrete example of a 
$\QQ$-factorial Fano variety with torus action of 
complexity two and maximal orbit quotient 
$X \dasharrow \PP_1 \times \PP_1$.

\begin{example}
\label{ex:genconstrex}
Consider the surface $Y := \PP_1 \times \PP_1$.
Then we have $\Cl(Y) = \ZZ^2$ and the 
Cox ring of $Y$ is the polynomial ring 
$\KK[T_0,T_1,T_2,T_3]$, where the 
$\ZZ^2$-grading is given by
$$ 
\deg(T_0) = \deg(T_1) = (1,0),
\qquad
\deg(T_2) = \deg(T_3) = (0,1).
$$
Consider the redundant system 
$\alpha = (f_0, \ldots, f_5)$ of 
generators for $\mathcal{R}(Y)$ 
consisting of $f_i := T_i$ for 
$i = 0, \ldots, 3$ and
the defining equations of the 
diagonals
$$
f_4 \ := \ T_0T_3-T_1T_2,
\qquad\qquad
f_5 \ := \ T_0T_2-T_1T_3,
$$
both being of degree $(1,1)$.
A matrix $B$ of relations between 
the degrees of generators $f_0, \ldots, f_5$
is given by
$$ 
B
\ := \ 
\left[
\begin{array}{rrrrrr}
-1 & 1 & 0 & 0 & 0 & 0
\\
0 & 0 & -1 & 1 & 0 & 0
\\
-1 & 0 & -1 & 0 & 1 & 0 
\\
-1 & 0 & -1 & 0 & 0 & 1
\end{array}
\right].
$$
Then $Y$ is embedded into the toric variety
$Z_\Delta$, the fan $\Delta$ of which lives in
$\ZZ^4$ and has the following four cones 
as its maximal ones
$$
\cone(v_i,v_j,v_k,v_4,v_5),
\qquad
0 \le i < j < k \le 3,
$$
where $v_i$ denotes the $i$-th column of $B$.
Note that $Y$ is given in Cox coordinates 
by the equation $f_4 = f_0f_3-f_1f_2$
and $f_5 = f_0f_2-f_1f_3$.
To build the variety $X$, consider the 
matrix
$$ 
P
\ := \ 
\left[
\begin{array}{rrrrrrrr}
-1 & 1 & 0 & 0 & 0 & 0 & 0 & 0
\\
0 & 0 & -1 & 1 & 0 & 0 & 0 & 0
\\
-1 & 0 & -1 & 0 & 2 & 0 & 0 & 0
\\
-1 & 0 & -1 & 0 & 0 & 1 & 2 & 0
\\
-1 & -1 & 1 & 1 & -1 & 1 & -1 & -1
\\
0 & 1 & 0 & 1 & 0 & 1 & 2 & -1
\end{array}
\right]
$$
obtained from $B$ by 
firstly doubling the last column,
then multiplying
its last and third last columns with $2$,
adding a zero column
and, after that, adding two new rows
as $d,d'$ part.
We gain polynomials by modifying
the variables of the describing 
relations of $Y \subseteq Z_\Delta$ 
accordingly to the column
modifications:
$$
g_1 := T_{41}^2 -T_{01}T_{31}+T_{11}T_ {21},
\qquad\qquad
g_2 := T_{51}T_{52}^2 -T_{01}T_{21}+T_{11}T_{31}.
$$
By construction, the polynomials $g_i$ 
are homogeneous with respect to the grading 
of $\KK[T_ {ij},S_1]$
given by 
$$ 
\deg(T_{ij}) 
\ := \ 
Q(e_{ij})
\ \in \
K,
\qquad\qquad
\deg(S_1) 
\ := \ 
Q(e_{1})
\ \in \ 
K,
$$
where $Q \colon \ZZ^8 \to K := \ZZ^8 / \im(P^*) \cong \ZZ^2$,
is the projection 
and $e_{ij},e_1 \in \ZZ^8$
are the canonical basis vectors,
numbered according to the variables
$T_{ij}$ and $S_1$.
Let $\Sigma = \Sigma(u)$ in $\ZZ^6$ be the 
normal fan of the polytope
$$
(P^*)^{-1}(Q^{-1}(u) \cap \gamma) - e)
\ \subseteq \ 
\QQ^{6},
$$
where $u := (8,-4) \in K$ and 
$e \in \ZZ^8$ is any point with 
$Q(e) = u$.
Then $\Sigma$ has the columns of $P$ 
as its primitive  generators.
Moreover, the projection $\ZZ^8 \to \ZZ^6$
onto the first six coordinates
sends the rays of $\Sigma$ into
the rays of $\Delta$.
This gives a rational toric map 
$\pi \colon Z_\Sigma \dasharrow Z_\Delta$.
Now, define a variety
$$ 
X 
\ = \ 
X(\alpha,P,\Sigma)
\ := \ 
\overline{ \pi^{-1} (Y \cap \TT^4) }
\ \subseteq \ 
Z_{\Sigma}.
$$
Then $X$ is invariant under the 
action of the subtorus $\TT := \{\mathds{1}_4\} \times \TT^2$ 
of the acting torus $\TT^6$ of $Z$.
The $\TT$-variety $X$ is normal, of dimension 
four with divisor class group and Cox ring given 
by
$$ 
\Cl(X) 
\ = \ 
\ZZ^2,
\qquad\qquad
\mathcal{R}(X)
\ = \ 
\KK[T_{ij},S_1] / \bangle{g_1,g_2},
$$
where the grading of the Cox ring is the 
one given above.
This involves application of 
Proposition~\ref{thm:t-mds-1};
the necessary assumptions 
are directly verified.
Now, applying
Propositions~\ref{prop:Qfactchar},
\ref{prop:PicandCones}, \ref{prop:ci}
and their implementation in~\cite{HaKe},
we see that $X$ is a $\QQ$-factorial
Fano variety of Gorenstein index 30.
\end{example}

\begin{remark}
\label{rem:extv2cpl1}
The Cox ring based approach of~\cites{HaHe,HaWr}
produces all $A_2$-maximal rational $\TT$-varieties 
$X$ of complexity one with only constant invertible 
global functions via a construction having a 
pair of matrices and a bunch of cones as input data.
Let us see how to retrieve these $X$
via Construction~\ref{constr:t-mds}.
Two types of $\TT$-varieties are distinguished: 
the first admits non-constant $\TT$-invariant 
functions, the second does not.

\medskip
\noindent
\emph{Type~1.}
The starting variety is $Y = \KK$ in 
$Z_\Delta = \KK^{r+1}$,
where the generator matrix of~$\Delta$ 
is $B = \mathds{E}_{r+1}$,
we embed via the system $\alpha = (f_0, \ldots,f_r)$
given by $f_i = T - a_i$ with pairwise different 
$a_i \in \KK$ and the defining relations for $Y$ are
$$
h_i 
\ = \ U_i-U_{i+1} - (a_i+a_{i+1}) 
\ \in \ 
\KK[U_0, \ldots, U_r],
\quad
i = 0, \ldots, r-1.
$$

\medskip
\noindent
\emph{Type~2.}
The starting variety 
is $Y = \PP_1$ in $Z_\Delta = \PP_r$,
where~$\Delta$ has generator matrix 
$B = [-\mathds{1}_r,\mathds{E}_r]$,
the embedding system 
$\alpha = (f_0, \ldots,f_r)$ with
$f_i:=a_{i,1}T_1 +a_{i,2}T_2$ such that 
$[a_{i,1},a_{i,2}] \in \PP_1$ are
pairwise different 
and the defining relations for $Y$ are
$$
h_i 
\ := \
\det
\left[
\begin{array}{ccc}
a_{i,1} & a_{i+1,1} & a_{i+2,1}
\\
a_{i,2} & a_{i+1,2} & a_{i+2,2}
\\
U_i & U_{i+1} & U_{i+2}
\end{array}
\right]
\ \in \ 
\KK[U_0, \ldots, U_r],
\quad
i = 0, \ldots, r-1.
$$

\medskip
\noindent
Now run Construction~\ref{constr:t-mds}
for both types.
The assumptions of Proposition~\ref{thm:t-mds-1}
are satisfied by~\cite{HaHe}*{Thm.~10.4} 
and~\cite{HaWr}*{Thm.~1.5}.
Thus, Theorem~\ref{thm:t-mds-2}
gives the desired result.
The input data $A$, $P$ 
and $\Phi$ of~\cite{HaHe,HaWr} are
recovered as follows:
the matrix $A$ is $[a_0,\ldots, a_r]$, 
where for Type~1 the $a_i$ are as above 
and for Type~2 we set
$a_i = (a_{i,1},a_{i,2})$,
the matrix $P$ is the one produced 
by Construction~\ref{constr:t-mds}
and the bunch of cones $\Phi$ is 
related to the fan $\Sigma$ 
via Gale duality as outlined in 
Remark~\ref{rem:fan2bunch}.
\end{remark}

\begin{example}
\label{ex:gr2n}
Fix $n \ge 5$ and let $X_{2,n} = \Gr(2,n)$ 
be the Grassmannian of two-dimensional 
vector subspaces of $\KK^n$.
Then $X_{2,n}$ has the projective
linear group $\PGL(n)$ as its 
automorphism group~\cite{Chow}. 
We will pick a maximal torus 
$\TT \subseteq \PGL(n)$ and show
how to obtain the $\TT$-variety $X_{2,n}$ 
via Construction~\ref{constr:t-mds}.
Set 
$$ 
r \ := \ {n \choose 2}.
$$
Identify the Pl\"ucker coordinate space 
$\KK^n \wedge \KK^n$ with $\KK^r$ such that 
the basis $(e_i \wedge e_j)$ 
of $\KK^n \wedge \KK^n$ corresponds to 
the basis $(e_{ij})$ of $\KK^r$, where $e_{ij} \in \KK^r$ 
has  $ij$-th Pl\"ucker coordinate equal to 
one and all others zero.  
We order the bases $(e_i \wedge e_j)$ 
and~$(e_{ij})$ lexicographically.
Accordingly, we have the Pl\"ucker ideal 
and the affine cone 
$$ 
I_{2,n} 
\ \subseteq \
\KK[T_{ij}; \ 1 \le i < j \le n],
\qquad\qquad
\bar X_{2,n} 
\ = \ 
V(I_{2,n})
\ \subseteq \
\KK^r.
$$
Look at the largest diagonal torus 
$\bar \TT \subseteq \GL(r)$
leaving~$\bar X_{2,n}$ invariant.
With $t := r - n$,
we obtain $\bar \TT$ as the kernel of the 
homomorphism $b \colon \TT^r \to \TT^t$
defined by the following $t \times r$ matrix,
the first $n$ columns of which are 
defined by the remaining ones as 
indicated:
$$
B
 = 
\left[
v_{12}, \ldots, v_{n-1 \, n}
\right]
 =  
\left[
v_{12}, \ldots, v_{1n},  \mathds{E}_t,  -\mathds{1}_t  
\right],
\quad
v_{1i} 
 :=  
{\textstyle
\sum_{\left\{j, k\right\} \cap \left\{1,i\right\} = \emptyset} v_{jk}.
}
$$
Observe that $\bar \TT \subseteq \GL(r)$ 
is of dimension $n$.
The corresponding torus $\TT \in \PGL(r)$ 
is of dimension $n-1$.
Moreover, $\TT$ acts effectively on 
$X_{2,n} \subseteq \PP_{r-1}$ and thus 
$\TT$ defines a maximal torus of the 
automorphism group $\PGL(n)$.
Now, look at the $(r-1) \times r$
stack matrix 
$$
P
\ := \
\left[
\begin{array}{c}
     B  \\
     d 
\end{array}
\right]
=
\left[
\begin{array}{ccl}
v_{12}, \ldots, v_{1n}  &  \mathds{E}_t &  -\mathds{1}_t  
\\
 \mathds{E}_{n-1}   & 0 &  -\mathds{1}_{n-1}  
\end{array}
\right].
$$
Observe that the kernel of $P$ is generated by 
the vector $\mathds{1}_r \in \ZZ^r$.
In particular,~$P$ differs from 
$[\mathds{E}_r,-\mathds{1}_r]$ 
by multiplication with a unimodular 
matrix from the left.
Let~$\Sigma$ be the unique complete 
fan in $\ZZ^r$ having $P$ as generator matrix
and let $\Delta$ be the fan in~$\ZZ^t$
having the rays through the columns of
$B$ as its maximal cones.
Then we obtain a commutative diagram
$$ 
\xymatrix{
{\bar X_{2,n}} 
\ar@{}[r]|\subseteq
\ar@{-->}[d]^p
\ar@{-->}@/_2pc/[dd]_{b}
& 
{\KK^r}
\ar@{-->}[d]_p
\ar@{-->}@/^2pc/[dd]^{b}
\\
X_{2,n}
\ar@{}[r]|\subseteq
\ar@{-->}[d]^{\quot \TT}
&
{Z_\Sigma}
\ar@{-->}[d]_{\quot \TT}
\\
Y_{2,n}
\ar@{}[r]|\subseteq
&
Z_{\Delta}
}
$$
where $Y_{2,n} \subseteq Z_\Delta$ is 
the closure of the image 
$b(\bar X_{2,n} \cap \TT^r)$. 
We have $Z_\Sigma = \PP_{r-1}$ and 
$X_{2,n} \subseteq Z_\Sigma$ is the Pl\"ucker 
embedding.
Moreover, $R_{2,n} = \KK[T_{ij}]/I_{2,n}$ 
is a unique 
factorization domain~\cite[Prop.~8.5]{Samuel} and the 
variables $T_{ij}$ define prime elements.
Thus, $Y_{2,n} \subseteq Z_\Delta$ is 
an explicit variety and the $\TT$-variety
$X_{2,n} \subseteq Z_\Sigma$ is an 
explicit $\TT$-variety.
Moreover, the $\TT$-action has maximal 
orbit quotient 
$$ 
\pi \colon X_{2,n} \ \dasharrow \ Y_{2,n}.
$$ 
We claim that, up to codimension
two, $Y_{2,n}$ equals the
blowing up $\Bl_{n-1}(\PP_{n-3})$ 
of $\PP_{n-3}$ at $n-1$ points in 
general position.
Indeed, we may first blow up 
the $n-2$ toric fixed points of 
$\PP_{n-3}$ and then the base point 
of the resulting toric variety.
Doing the latter 
via~\cite[Alg.~5.7]{HaKeLa1},
one directly checks that the 
procedure terminates after the first 
step and delivers $R_{2,n}$ as 
Cox ring of $\Bl_{n-1}(\PP_{n-3})$.
Thus, we may also take 
$X_{2,n} \dasharrow \Bl_{n-1}(\PP_{n-3})$ 
as a maximal orbit quotient for the 
$\TT$-action. 
\end{example}

\begin{remark}
\label{rem:tschauchow}
In order to describe a Mori dream space with 
torus action via divisorial fans~\cites{AlHa,AlHaSu}, 
it happens that one has to start with 
a non Mori dream space as prospective
Chow quotient. 
For example, the maximal torus 
action on the Grassmannian $\Gr(2,n)$ has the
moduli space $\overline{M}_{0,n}$ as its Chow 
quotient~\cite{Kap} and for $n \ge 10$, 
it is known that $\overline{M}_{0,n}$ and hence 
all its blow ups have a non-finitely 
generated Cox ring~\cites{CaTe,GoKa,HaKeLa}. 
Note that the Chow quotient $\overline{M}_{0,n}$
starts differing  at $n = 6$ from the maximal 
orbit quotient discussed just before.
Altmann and Hein gave in~\cite{AlHe} a description 
of the maximal torus action on $\Gr(2,n)$ 
by means of a divisorial fan living on 
$\overline{M}_{0,n}$.
\end{remark}

\section{Arrangement varieties}
\label{sec:hypPlaneAr}

We introduce general arrangement varieties
as certain $\TT$-varieties with a projective 
space as maximal orbit quotient, naturally
generalizing the rational projective 
$\TT$-varieties of complexity one.
Construction~\ref{constr:R(A,P)} produces 
examples which are explicit in the sense of 
Construction~\ref{constr:t-mds}
and Theorem~\ref{thm:arrTvar}
shows that we obtain all $A_2$-maximal 
general arrangement varieties in this 
way.

For the precise definition, let us recall 
the necessary notions on
projective hyperplane arrangements.
A \emph{hyperplane} $H$ in the projective 
space $\PP_n$ is the zero set 
of a nonzero homogeneous polynomial 
of degree one.
A \emph{hyperplane arrangement} in $\PP_n$
is a finite collection $H_1, \ldots, H_r$
of hyperplanes in $\PP_n$.
A hyperplane arrangement in~$\PP_n$ is called
\emph{general} if for any 
$1 \le i_1 < \ldots < i_k \le r$,
the intersection $H_{i_1} \cap \ldots \cap H_{i_k}$
is of dimension $(n-k)$.

\begin{definition}
\label{def:genarrvardef}
A \emph{(general) arrangement variety}
of complexity $c$ is a $\TT$-variety~$X$ 
with maximal orbit quotient $X \dasharrow \PP_c$
such that the doubling divisors 
$C_0, \ldots, C_r$ form a (general) hyperplane 
arrangement in $\PP_c$. 
\end{definition}

\begin{remark}
The projective general arrangement 
varieties of complexity $c=1$ are 
precisely the rational projective 
$\TT$-varieties of complexity one.
Indeed, any rational projective 
$\TT$-variety $X$ of complexity one
has maximal orbit quotient 
$\pi \colon X \dasharrow \PP_1$.
The doubling divisors form a point 
configuration in $\PP_1$, which trivially 
satisfies the conditions of a general 
hyperplane arrangement.
\end{remark}

We enter the construction of general 
arrangement varieties.
As in the case of complexity 
one~\cites{HaHeSu,HaHe,ArDeHaLa},
we first write down  the prospective 
Cox rings in terms of generators and 
relations, then investigate their 
algebraic properties and after all 
that construct the varieties we are 
aiming for.

\begin{construction}
\label{constr:R(A,P_0)}
Fix integers $r \ge c > 0$
and $n_0, \ldots, n_r > 0$ 
as well as $m \ge 0$. 
Set $n := n_0 + \ldots + n_r$.
The input data is a pair $(A,P_0)$, 
where 
\begin{itemize}
\item 
$A$ is a $(c+1) \times (r+1)$ matrix over $\KK$ 
such that any $c+1$ of its columns 
$a_0, \ldots, a_r$ are linearly independent,
\item
$P_0$ is an integral $r \times (n+m)$ matrix 
built from tuples of positive integers 
$l_i = (l_{i1},\dots,l_{in_i})$, where $i = 0, \ldots, r$,
as follows
$$
P_{0}
\ := \
\left[
\begin{array}{ccccccc}
-l_{0} & l_{1} &  & 0 & 0  &  \ldots & 0
\\
\vdots & \vdots & \ddots & \vdots & \vdots &  & \vdots
\\
-l_{0} & 0 &  & l_{r} & 0  &  \ldots & 0
\end{array}
\right].
$$
\end{itemize}
Write  $\KK[T_{ij},S_k]$ for the polynomial ring 
in the variables $T_{ij}$, where $i = 0, \ldots, r$, 
$j = 1, \ldots, n_i$, 
and $S_k$, where $k = 1, \ldots, m$.
Every $l_i$ defines a monomial 
$$
T_i^{l_i}
\ := \ 
T_{i1}^{l_{i1}} \cdots T_{in_i}^{l_{in_i}}
\ \in \ 
\KK[T_{ij},S_k].
$$
Moreover, for every $t = 1, \ldots, r-c$, 
we obtain a polynomial $g_t$ by computing 
the following  $(c+2) \times (c+2)$ 
determinant 
$$ 
g_t 
\ := \
\det
\left[
\begin{array}{cccc}
a_0 & \ldots & a_c & a_{c+t}
\\
T_0^{l_0} & \ldots & T_c^{l_{c}} & T_{c+t}^{l_{c+t}}
\end{array}
\right]
\ \in \ 
\KK[T_{ij},S_k].
$$
Now, let $e_{ij} \in \ZZ^{n}$ 
and $e_k \in \ZZ^{m}$ denote the 
canonical basis vectors
and consider the projection 
$$
Q_0 \colon \ZZ^{n+m} 
\ \to \ 
K_0 := \ZZ^{n+m} / \im(P_0^*)
$$ 
onto the factor group
by the row lattice of $P_0$.
Then the 
\emph{$K_0$-graded $\KK$-algebra
associated with $(A,P_0)$} 
is defined by
$$ 
R(A,P_0)
\ := \ 
\KK[T_{ij},S_k] / \bangle{g_1,\ldots,g_{r-c}},
$$
$$
\deg(T_{ij}) :=  Q_0(e_{ij}),
\qquad
\deg(S_k) :=  Q_0(e_k).
$$
\end{construction}

\begin{example}
\label{ex:R(A,P_0)}
Let us take $c=2$ and $r=3$. 
Thus, we will work with a $3 \times 4$ matrix~$A$.
Moreover, let $n_0 = 2$ and $n_1=n_2=n_3=1$ 
and fix $m=0$.
This amounts to $n=5$ and a $3 \times 5$ matrix $P_0$.
We choose
$$ 
A 
\ =  \
\left[
\begin{array}{rrrr}
1 & 0 & 0 & -1
\\
0 & 1 & 0 & -1
\\
0 & 0 & 1 & -1
\end{array}
\right],
\qquad\qquad
P_0 
\ =  \
\left[
\begin{array}{rrrrr}
-1 & -2 & 2 & 0 & 0 
\\
-1 & -2 & 0 & 2 & 0
\\
-1 & -2 & 0 & 0 & 4
\end{array}
\right].
$$
So, the exponent vectors $l_i$ are $l_0 = (1,2)$, 
$l_1 = l_2 = (2)$ and $l_3 = (4)$.
Accordingly, we obtain the four monomials
$$
T_0^{l_0} \ = \ T_{01}T_{02}^2,
\qquad
T_1^{l_1} \ = \ T_{11}^2,
\qquad
T_2^{l_2} \ = \ T_{21}^2,
\qquad
T_3^{l_3} \ = \ T_{31}^4.
$$
We arrive at $r-c = 1$ relation 
$g_1 \in \KK[T_{01},T_{02},T_{11},T_{21},T_{31}]$, 
obtained by computing the following $4 \times 4$ 
determinant
$$ 
g_1
\ = \ 
\det
\left[
\begin{array}{cccc}
1 & 0 & 0 & -1
\\
0 & 1 & 0 & -1
\\
0 & 0 & 1 & -1
\\
T_{01}T_{02}^2 & T_{11}^2 & T_{21}^2 & T_{31}^4
\end{array}
\right]
\ = \ 
T_{01}T_{02}^2 + T_{11}^2 + T_{21}^2 + T_{31}^4.
$$
The canonical basis vectors 
of the row space of $P_0$ are 
indexed in accordance with 
the variables $T_{ij}$, that means
that we write
$$
e_{01},e_{02},e_{11},e_{21},e_{31} 
\ \in \ \ZZ^5 
\ = \ 
\ZZ^{n_0+n_1+n_2+n_3}.
$$
We have $K_0 = \ZZ^5 / \im(P_0^*) = 
\ZZ^2 \oplus \ZZ/2\ZZ \oplus \ZZ/2\ZZ$.
The projection $Q_0 \colon \ZZ^5 \to K_0$
sending $e_{ij}$ to its class in $K_0$ 
is made concrete by the degree matrix
$$ 
Q_0 
\ =  \
[Q_0(e_{ij})]
\ =  \
\left[
\begin{array}{ccccc}
2 & 1 & 2 & 2 & 1
\\
0 & 2 & 2 & 2 & 1
\\
\bar 0 & \bar 0 & \bar 1 & \bar 1 & \bar 0 
\\
\bar 0 & \bar 0 & \bar 1 & \bar 0 & \bar 0 
\end{array}
\right].
$$
Consequently, for the initial data $A$ and $P_0$ 
of this example, the resulting $K_0$-graded 
algebra $R(A,P_0)$ is given by
$$
\KK[T_{01},T_{02},T_{11},T_{21},T_{31}]
\ / \ 
\langle 
T_{01}T_{02}^2 + T_{11}^2 + T_{21}^2 + T_{31}^4 
\rangle,
\qquad 
\deg(T_{ij}) \  =  \ Q_0(e_{ij}). 
$$
\end{example}

We present the basic properties of the 
graded algebra $R(A,P_0)$. 
Recall that a grading of a $\KK$-algebra 
$R = \oplus_K R_w$ by a finitely generated 
abelian group is \emph{effective} if 
the weights $w \in K$ with $R_w \ne \{0\}$ 
generate $K$ as a group and \emph{pointed}, 
if $R_0 = \KK$ holds and $R_w \ne \{0\} \ne R_{-w}$ 
is only possible for torsion elements $w \in \KK$.
Finally, we say that an effective grading is of 
\emph{complexity $c$} if $\dim(R) - \rk(K) = c$
holds.

\begin{theorem}
\label{thm:R(A,P_0)}
Let $R(A,P_0)$ be a $K_0$-graded $\KK$-algebra 
arising from Construction~\ref{constr:R(A,P_0)}.
Then $R(A,P_0)$ is an integral, normal, complete intersection 
ring satisfying 
$$ 
\dim(R(A,P_0)) \ = \ n+m-r+c,
\qquad 
R(A,P_0)^* \ = \ \KK^*.
$$
The $K_0$-grading of $R(A,P_0)$ 
is effective, pointed, factorial and of complexity~$c$.
The variables $T_{ij}$, $S_k$ define pairwise 
non-associated $K_0$-primes in $R(A,P_0)$,
and for $c\geq 2$, they define even primes.
\end{theorem}

The following auxiliary statements for the proof
of this theorem are also used later.
We begin with discussing the specific nature 
of the matrix $A$ and its impact on the ideal 
of relations of $R(A,P)$.

\begin{remark}
\label{rem:kerA}
Situation as in 
Construction~\ref{constr:R(A,P_0)}.
For any tuple $I = (i_1, \ldots, i_{c+2})$
of strictly increasing integers 
from $[0,r]$, consider the matrix 
$$
A(I)
\ := \ 
\left[ a_{i_1}, \ldots , a_{i_{c+2}} \right],
$$
Let $w(I) \in \KK^{c+2}$ denote 
the cross product of the rows of 
$A(I)$ and define a vector
$v(I) \in  \KK^{r+1}$
by putting the entries of 
$w(I)$ at the right places:
$$ 
v(I)_i
\ := \ 
\begin{cases}
w(I)_j, & \quad i = i_j \text{ occurs in } I = (i_1, \ldots, i_{c+2}),
\\
0, & \quad \text{else}.
\end{cases}
$$
Then any linearly independent choice of vectors 
$v(I_1), \ldots, v(I_{r-c})$ is a basis for 
$\ker(A)$.
Note that any nonzero $v \in \ker(A)$ 
has at least $c+2$ nonzero coordinates.
\end{remark}

\begin{remark}
\label{rem:kerA2rel}
Situation as in 
Construction~\ref{constr:R(A,P_0)}.
Every vector $v \in \ker(A) \subseteq \KK^{r+1}$
defines a polynomial 
$$
g_v
\ := \
v_0 T_0^{l_0} + \ldots + v_r T_r^{l_r}
\ \in \ 
\bangle{g_1,\ldots,g_{r-c}}.
$$
Moreover, if a subset $B \subseteq \ker(A)$ 
generates $\ker(A)$ as a vector space, 
then the polynomials $g_v$, $v \in B$,
generate the ideal $\bangle{g_1,\ldots,g_{r-c}}$.
In particular, we have 
$$ 
\bangle{g_1,\ldots,g_{r-c}}
\ = \ 
\bangle{g_{v(I)}; \ 
I = (i_1, \ldots, i_{c+2}), \ 
0 \le i_1 < \ldots < i_{c+2} \le r},
$$
with the tuples $I$ from Remark~\ref{rem:kerA2rel}.
Observe that each $g_v$, $0 \ne v \in \ker(A)$,
has at least $c+2$ of the monomials $T_i^{l_i}$
and all the $g_v$ share the same $K_0$-degree.
\end{remark}

\begin{lemma}
\label{lem:RAPreductions}
Let $R(A,P_0)$ be a graded algebra 
arising from 
Construction~\ref{constr:R(A,P_0)}.
\begin{enumerate}
\item
If we have $l_{i1} + \ldots + l_{in_i} = 1$
for some $i$, then $R(A,P_0)$ 
is isomorphic to a ring $R(A',P_0')$ 
with data $r' = r-1$ and $c' = c$.
\item
For any generator $T_{ij}$, the 
factor ring $R(A,P_0)/\bangle{T_{ij}}$ 
is isomorphic to a ring $R(A',P_0')$ 
with data $r' = r-1$ and $c' = c-1$.
\end{enumerate}
\end{lemma}

\begin{proof}
To obtain~(i), let $A'$ be the matrix obtained 
by deleting the $i$-th column from $A$.
Then the respective ideals defined by $A$ 
and $A'$ produce isomorphic rings.
Adapting the matrix $P_0$ accordingly,
gives the desired~$P_0'$.

We show~(ii).
As elementary row operations on $A$ neither 
change the required properties of $A$ 
nor the defining ideal of $R(A,P)$,
we may assume that $a_{i1} \ne 0$ holds
and all other entries of the $i$-th column
of $A$ equal zero.
Then the matrix $A'$ obtained 
by deleting the first row and the $i$-th 
column from~$A$ satisfies the assumptions 
of Construction~\ref{constr:R(A,P_0)}
with $r' = r-1$ and $c' = c-1$.
Using Remarks~\ref{rem:kerA} 
and~\ref{rem:kerA2rel}, we see that 
the ideal defined by $A'$ corresponds 
to the defining ideal of 
$R(A,P_0)/\bangle{T_{ij}}$.
Again, adapting the matrix $P_0$ accordingly,
gives the desired~$P_0'$. 
\end{proof}

\begin{definition}
\label{def:bigleaftype}
Situation as in Construction~\ref{constr:R(A,P_0)}.
We say that a point $z \in \KK^{n+m}$ with coordinates 
$z_{ij}$, $z_k$ is of 
\begin{enumerate}
\item
\emph{big type}, if for every $i = 0, \ldots, r$,
there is an index $1 \le j_i \le n_i$
such that $z_{ij_i} = 0$ holds,
\item
\emph{leaf type}, if there is a set
$I_z = \{i_1,\ldots,i_c\}$ of indices  
$0 \le i_1 < \ldots < i_c \le r$,
such that for all $i$, $j$, we have 
$z_{ij} = 0 \Rightarrow i \in I_z$.
\end{enumerate}
\end{definition}

\begin{remark}
Situation as in Construction~\ref{constr:R(A,P_0)}.
Consider $\gamma = \QQ^{n+m}$, a face 
$\gamma_0 \preccurlyeq \gamma$ and the 
complementary face $\gamma_0^* \preccurlyeq \gamma$.
Then any coordinate $z_{ij}$, $z_{k}$ 
of $z = z_{\gamma_0^*} \in \KK^{n+m}$ 
equals zero or one and we have 
$$ 
z_{ij} =  0 \ \iff \ e_{ij} \in \gamma_0,
\qquad\qquad
z_k  =  0 \ \iff \ e_k \in \gamma_0.
$$
In particular, there is a point of 
big (leaf) type in
$\TT^{n+m} \cdot z_{\gamma_0^*} \subseteq \KK^{n+m}$
if and only if all points of this toric
orbit are of big (leaf) type.
Moreover, in terms of the cone~$P_0(\gamma_0^*)$
with~$P_0$ from Construction~\ref{constr:R(A,P_0)},
we obtain the following characterizations:
\begin{enumerate}
\item
$z_{\gamma_0^*}$ is of big type if and only if  
$P_0(\gamma_0^*) = \QQ^r$ holds,
\item
$z_{\gamma_0^*}$ is of leaf type if and only if  
$P_0(\gamma_0^*) \ne \QQ^r$ and 
$\dim(P_0(\gamma_0^*)) \le c$.
\end{enumerate}
Observe that if one of the conditions of~(ii) holds, 
then the image cone $P_0(\gamma_0^*)$
is generated by at most $r$ vectors from 
$-\mathds{1}_r,e_1\ldots,e_r \in \ZZ^r$ and thus 
is pointed.
\end{remark}

\begin{lemma}
\label{lem:pointsofbarX}
For $\bar{X} = V(g_1,\ldots, g_{r-c}) \subseteq \KK^{n+m}$
from Construction~\ref{constr:R(A,P_0)},
we have the following statements. 
\begin{enumerate}
\item
Every point $z \in \bar{X}$ is either of big type or 
it is of leaf type.
\item
Every $z \in \KK^{n+m}$ of big type is contained in 
$\bar{X}$.
\item
For every $z \in \KK^{n+m}$ of leaf type,
there is a $t \in \TT^{n+m}$ with 
$t \cdot z \in \bar{X}$.
\end{enumerate}
\end{lemma}

\begin{proof}
To obtain~(i), we have to show that 
any $z \in \bar{X}$ which is not 
of big type must be of leaf type.
Otherwise, there are 
indices~$i_1 < \ldots < i_{c+1}$ 
and associated~$j_q$ 
with $z_{i_qj_q} = 0$.
As $z$ is not of big type, there is
at least one index $i_0$ with $z_{i_0j} \ne 0$ 
for all $j = 1, \ldots, n_{i_0}$.
Remarks~\ref{rem:kerA} and~\ref{rem:kerA2rel}
provide us with a relation 
$g \in \bangle{g_1,\ldots, g_{r-c}}$ 
involving precisely the monomials 
$T_i^{l_i}$ for $i = i_0, i_1, \ldots, i_{c+1}$.
Then $g(z) = 0$ implies $z_{i_0j} = 0$ 
for some $j = 1, \ldots, n_{i_0}$;
a contradiction.

We verify~(ii) and~(iii).
Let $z \in \KK^{n+m}$. If $z$ is of 
big type, then we obviously have 
$g_i(z) = 0$ for $i = 1, \ldots, r-c$.
Thus, $z \in \bar{X}$.
Now, assume that $z$ is of leaf type.
First consider the case $I_z = \{1,\ldots,c\}$.
Then, suitably scaling $z_{c+1, 1}$,
we achieve $g_1(z) = 0$. 
Next we scale $z_{c+2, 1}$ to 
ensure $g_2(z) = 0$, and so on,
until we have also $g_{r-c}(z) = 0$.
Then we have found our $t \in \TT^{n+m}$ 
with $t \cdot z \in \bar{X}$.
Given an arbitrary~$I_z$, 
Remarks~\ref{rem:kerA} and~\ref{rem:kerA2rel}
yield a suitable system 
$g_1', \ldots, g_{r-c}'$ 
of ideal generators that allows us 
to argue analogously.
\end{proof}

\begin{lemma}
\label{lem:singloc}
Situation as in Construction~\ref{constr:R(A,P_0)}.
Let $\bar{X} = V(g_1,\ldots,g_{r-c}) \subseteq \KK^{n+m}$
and denote by $J$ the Jacobian of $g_1,\ldots,g_{r-c}$.
Then, for any $z \in \bar{X}$, the following 
statements are equivalent:
\begin{enumerate}
\item 
The Jacobian $J(z)$ is not of full rank, 
i.e., we have $\rk(J(z)) < r-c$.
\item
The point $z \in \bar{X}$ is of big type 
and there are $i_1 < \ldots < i_{c+2}$ such that 
each of these $i_q$ fulfills one of the subsequent 
two conditions:
\begin{itemize}
\item
$z_{i_qj_q} = 0$ and $l_{i_qj_q} \ge 2$ hold
for at least one $1 \le j_q \le n_{i_q}$,
\item
$z_{i_qj} = 0$ and $l_{i_qj} = 1$ hold for
at least two $1 \le j \le n_{i_q}$.
\end{itemize}
\end{enumerate}
In particular, the set of points $z \in \bar{X}$
with~$J(z)$ not of full rank is of codimension at least 
$c+1$ in $\bar{X}$.
\end{lemma}

\begin{proof}
Assertion~(ii) directly implies the supplement 
and, by a simple computation, also~(i).
We are left with proving ``(i)$\Rightarrow$(ii)''.
So, let $z \in \bar{X}$ be a point such 
that~$J(z)$ is not of full rank. Then there is a 
non-trivial linear combination annulating 
the lines of $J(z)$:
$$ 
\eta_1 \grad(g_1)(z) + \ldots + \eta_{r-c} \grad(g_{r-c})(z) 
\ = \ 
0.
$$
The corresponding
$g := \eta_1 g_1 + \ldots + \eta_{r-c} g_{r-c}$
satisfies $\grad(g)(z) = 0$ and 
is of the form $g = g_v$ with a nonzero
$v \in \ker(A)$ as in Remark~\ref{rem:kerA2rel}.
The condition $\grad(g)(z) = 0$ implies
$z_{ij_i} = 0$ for some $1 \le j_i \le n_i$ 
whenever the monomial~$T_i^{l_i}$ shows up 
in $g$.
As observed in Remark~\ref{rem:kerA2rel},
the polynomial $g$ has at least $c+2$ monomials.
Thus, we have $z_{ij_i} = 0$ for at least 
$c+2$ different $i$.
By Lemma~\ref{lem:pointsofbarX},
the point $z \in \bar{X}$ is of 
big type.
Moreover, the two conditions of~(ii)
reflect the fact $\grad(g)(z) = 0$. 
\end{proof}

\begin{proof}[Proof of Theorem~\ref{thm:R(A,P_0)}]
For $c = 1$, the statement is proven 
in~\cite{HaHe}*{Thm.~10.1 and Prop.~10.7}. 
So, assume $c \ge 2$.
First we show that 
$\bar{X} = V(g_1,\ldots,g_{r-c}) \subseteq \KK^{n+m}$
is connected.
By construction, the quasitorus 
$H_0 \subseteq \TT^{n+m}$
is the kernel of the homomorphism  
$\TT^{n+m} \to \TT^{r}$ defined by 
$P_0$.
Consider the multiplicative 
one-parameter subgroup
$\KK^* \to H_0$, $t \mapsto (t^\zeta,t^\xi)$,
where 
$$
\zeta 
\ = \ 
\left( 
\frac{n_0 \cdots n_r l_{01} \cdots l_{rn_r}}{n_0 l_{01}}, 
\ldots, 
\frac{n_0 \cdots n_r l_{01} \cdots l_{rn_r}}{n_r l_{rn_r}}
\right)
\ \in \
\TT^n,
\quad
\xi 
\ =  \
(1, \ldots, 1)
\ \in \ 
\TT^m.
$$
This gives rise to a $\KK^*$-action on $\bar{X}$
having the origin as an attractive fixed point.
Consequently, $\bar{X}$ is connected.
Moreover, we can conclude that all
invertible functions as well as all $H_0$-invariant 
functions are constant on $\bar{X}$.

Now, Lemma~\ref{lem:singloc} allows us to 
apply Serre's criterion and thus we obtain 
that $R(A,P_0)$ is an integral, normal, 
complete intersection.
By construction, the $K_0$-grading is effective
and as seen above, it is pointed.
To obtain factoriality of the $K_0$-grading,
localize $R(A,P_0)$ by the product over all 
generators $T_{ij}$, $S_k$, observe that the 
degree zero part of the resulting ring is 
a polynomial ring and apply~\cite{Be}*{Thm.~1.1}.
Finally, primality of the generators $T_{ij}$ 
follows from Lemma~\ref{lem:RAPreductions}~(i). 
\end{proof}

Now we use the algebras $R(A,P_0)$ obtained 
by Construction~\ref{constr:R(A,P_0)}
to produce general arrangement varieties.
The basic idea is to turn $R(A,P_0)$ into 
a prospective Cox ring via coarsening the 
grading by $K_0 = \ZZ^{n+m}/\im(P_0^*)$ to 
a grading by $K = \ZZ^{n+m}/\im(P^*)$,
where $P$ arises from $P_0$ by 
adding suitable further rows.

\begin{construction}
\label{constr:R(A,P)}
Let $A$ and $P_0$ be input data as in 
Construction~\ref{constr:R(A,P_0)}.
Moreover, fix $1 \le s \le n+m-r$ 
and let $d$ be an integral $s \times (n+m)$ 
matrix such that the columns $v_{ij}$, $v_k$
of the  $(r + s) \times (n+m)$ 
stack matrix 
$$
P
\ := \ 
\left[
\begin{array}{c}
P_0
\\
d
\end{array}
\right]
$$
are pairwise different, primitive and  
generate $\QQ^{r+s}$ as a vector space.
Consider the factor group $K := \ZZ^{n+m}/\im(P^*)$.
Then the projection $Q\colon \ZZ^{n+m} \to K$
factors through $Q_0$ and we obtain 
the \emph{$K$-graded $\KK$-algebra
associated with $(A,P)$}: 
$$ 
R(A,P)
\ := \ 
\KK[T_{ij},S_k] / \bangle{g_1,\ldots,g_{r-c}},
$$
$$
\deg(T_{ij}) :=  w_{ij} := Q(e_{ij}),
\qquad
\deg(S_k) :=  w_k \ := \ Q(e_k).
$$ 
Now, let $\Sigma$ be any fan in $\ZZ^{r+s}$ having
precisely the rays through the columns of~$P$ as 
its one-dimensional cones and let $Z$ be the 
associated toric variety.
Then we have a commutative diagram
$$
\xymatrix{
V(g_1,\ldots,g_{r-c})
\ar@{}[r]|{\qquad =}
&
{\bar{X}}
\ar@{}[r]|\subseteq
\ar@{}[d]|{\rotatebox[origin=c]{90}{$\scriptstyle\subseteq$}}
&
{\bar{Z}}
\ar@{}[r]|=
\ar@{}[d]|{\rotatebox[origin=c]{90}{$\scriptstyle\subseteq$}}
&
{\ZZ^{n+m}}
\\
&
{\hat{X}} 
\ar@{}[r]|\subseteq
\ar[d]_{\quot H}
& 
{\hat{Z}} 
\ar[d]^{\quot H}
&
\\
&
X
\ar@{}[r]|\subseteq
\ar@{-->}[d]
&
Z
\ar@{-->}[d]
\\
&
\PP_{c}
\ar[r]
&
\PP_{r}
&
}
$$ 
with the quasitorus $H = \Spec \, \KK[K]$,
Cox's quotient presentation 
$\hat{Z} \to Z$ and the induced
quotient $\hat{X} \to X$,
where $\hat{Z} := \bar{X} \cap \hat{Z}$.
The resulting variety
$X = X(A,P,\Sigma)$ is normal with 
dimension, invertible functions, divisor class group
and Cox ring given by 
$$ 
\dim(X) = s + c,
\qquad
\Gamma(X,\mathcal{O}^*) = \KK^*,
\qquad
\Cl(X) = K,
\qquad
\mathcal{R}(X) = R(A,P). 
$$
The acting torus $\TT_Z \subseteq Z$ 
splits as $\TT_Z = \TT^r \times \TT^s$
and the factor $\TT  = \{\mathds{1}_r\} \times \TT^s$ 
leaves $X \subseteq Z$ invariant.
The induced $\TT$-action on $X$
is effective and of complexity
$$
c(X) \ = \ c.
$$
Finally, the dashed arrows indicate the maximal 
orbit quotients for the $\TT$-actions on $X$ and 
$Z$ respectively and $\PP_c \subseteq \PP_r$ is 
the linear subspace given by 
$$ 
\PP_c 
\ = \ 
V(h_1,\ldots, h_{r-c}),
\qquad
h_t
\ := \
\det
\left[
\begin{array}{cccc}
a_0 & \ldots & a_c & a_{c+t}
\\
U_0 & \ldots & U_c & U_{c+t}
\end{array}
\right]
\ \in \ 
\KK[U_0, \ldots, U_r].
$$
A collection of doubling divisors 
for the maximal orbit quotient 
$X \dasharrow \PP_c$ 
is given by the intersections of $\PP_c$ 
with the coordinate hyperplanes of $\PP_r$
which form the general hyperplane 
arrangement 
$$ 
H_0, \ldots, H_r
\ \subseteq \ 
\PP_c,
\qquad
H_i 
\ := \ 
\{z \in \PP_c; \ a_{i0}z_0 + \ldots + a_{ic}z_c = 0\}.
$$
\end{construction}

\begin{proposition}
\label{rem:gen2hyp}
Let $X = X(A,P,\Sigma)$ arise from 
Construction~\ref{constr:R(A,P)}.
Consider the explicit variety 
$Y = \PP_c$ in $Z_\Delta = \PP_r$
with embedding system 
$\alpha = (f_0,\ldots,f_r)$,
where $\Delta$ is the complete fan 
in $\ZZ^r$ with generator matrix 
$B = [-\mathds{1}_r,\mathds{E}_r]$ and
$$ 
f_i 
\ = \ 
a_{i0}U_0 + \ldots + a_{ic}U_c
\ \in \
\KK[U_0, \ldots, U_c] 
\ = \
\mathcal{R}(\PP_c).
$$
Then the variety $X(A,P,\Sigma) \subseteq Z_\Sigma$ 
equals the variety 
$X(\alpha,P,\Sigma) \subseteq Z_\Sigma$ 
arising from Construction~\ref{constr:t-mds}.
In particular, $X(A,P,\Sigma) \subseteq Z_\Sigma$
is an explicit $\TT$-variety.
\end{proposition}

\begin{proof}[Proof of Construction~\ref{constr:R(A,P)} 
and Proposition~\ref{rem:gen2hyp}]
The fact that $X(A,P,\Sigma) \subseteq Z_\Sigma$ 
equals 
$X(\alpha,P,\Sigma) \subseteq Z_\Sigma$
is clear by construction.
Observe that, forgetting for the moment 
about the gradings, we have 
$R(A,P) = R(A,P_0)$.
Thus, Theorem~\ref{thm:R(A,P_0)}
ensures that $R(A,P)$ is normal, integral
with only constant homogeneous units.
Moreover, for $c \ge 2$, the generators
$T_{ij}$ and $S_k$ of $R(A,P)$ are 
pairwise non-associated.
Being prime and $K$-homogeneous, 
they are also $K$-prime.
For $c=1$, we infer $K$-primality of the 
generators from~\cite[Thm.~10.4]{HaHe}. 
So, $X(A,P,\Sigma) \subseteq Z_\Sigma$
satisfies the conditions of 
Definition~\ref{def:expltvar}
and hence is an explicit $\TT$-variety.
This yields in particular the statements on the
divisor class group and the Cox ring.
The statement on the maximal orbit 
quotient is due to Proposition~\ref{thm:t-mds-3}.
\end{proof}

\begin{remark}
\label{rem:irredundant}
According to Lemma~\ref{lem:RAPreductions}~(i),
we may always assume that the defining data $P$ 
of Construction~\ref{constr:R(A,P)} 
is~\emph{irredundant} in the sense that 
$l_{i0} + \ldots + l_{in_i} \ge 2$ 
holds for every $i = 0,\ldots, r$.
In this case, we also say that $X(A,P,\Sigma)$ 
is \emph{irredundant}.
\end{remark}

\begin{definition}
By an \emph{explicit general arrangement variety}
we mean a $\TT$-variety $X = X(A,P,\Sigma)$ 
in $Z = Z_\Sigma$ arising from 
Construction~\ref{constr:R(A,P)}.
\end{definition}

\begin{example}
\label{ex:R(A,P)}
Let $A$ and $P_0$ be as in Example~\ref{ex:R(A,P_0)}.
We enhance $P_0$ by an $1 \times 5$ block as follows
$$
P 
\ =  \
\left[
\begin{array}{c}
P_0
\\
d
\end{array}
\right]
\ = \ 
\left[
\begin{array}{rrrrr}
-1 & -2 & 2 & 0 & 0 
\\
-1 & -2 & 0 & 2 & 0
\\
-1 & -2 & 0 & 0 & 4
\\
-1 & -3 & 1 & 1 & 1 
\end{array}
\right].
$$
So, we chose $s=1$.
We have
$K = \ZZ^5 / \im(P^*) = \ZZ \oplus \ZZ / 2\ZZ \oplus \ZZ / 2\ZZ$
and $Q \colon \ZZ^5 \to K$ is represented 
by the degree matrix, having 
$w_{ij} = Q(e_{ij}) \in K$ as its columns: 
$$ 
Q
\ =  \
[w_{01},w_{02},w_{11},w_{21},w_{31}]
\ =  \
\left[
\begin{array}{ccccc}
2 & 1 & 2 & 2 & 1
\\
\bar 0 & \bar 0 & \bar 1 & \bar 1 & \bar 0 
\\
\bar 0 & \bar 1 & \bar 0 & \bar 1 & \bar 0 
\end{array}
\right].
$$
Let $\Sigma$ be the unique complete fan in 
$\ZZ^4$ with $P$ as its generator matrix.
Then we arrive at a projective explicit 
arrangement variety 
$X = X(A,P,\Sigma)$ in $Z = Z_\Sigma$
with 
$$ 
\dim(X) \ = \ 3, 
\qquad
c(X) \ = \ 2,
\qquad
\Cl(X)
\ = \ 
\ZZ \oplus \ZZ / 2\ZZ \oplus \ZZ / 2\ZZ.
$$
Moreover, assigning to each generator $T_{ij}$ 
the divisor class $Q(e_{ij})$, we obtain 
a representation of the Cox ring by homogeneous
generators and relations: 
$$
\mathcal{R}(X)
\ = \
\KK[T_{01},T_{02},T_{11},T_{21},T_{31}]
\ / \ 
\langle 
T_{01}T_{02}^2 + T_{11}^2 + T_{21}^2 + T_{31}^4 
\rangle.
$$
As a maximal orbit quotient, we have 
$\pi \colon X \dasharrow \PP_2$ 
and the doubling divisors form the general 
line configuration in $\PP_2$ given by 
$$ 
V(T_0), 
\quad
V(T_1), 
\quad
V(T_2), 
\quad
V(T_0+T_1+T_2).
$$
\end{example}

\begin{theorem}
\label{thm:arrTvar}
Let $X$ be an $A_2$-maximal general 
arrangement variety. 
Then $X$ admits a presentation as an
explicit general arrangement variety.
\end{theorem}

\begin{proof}
According to Definition~\ref{def:genarrvardef}, 
there is a maximal orbit quotient 
$\pi \colon X \dasharrow Y$ with $Y = \PP_c$ 
admitting a general hyperplane arrangement 
$C_0, \ldots, C_r$ as a collection of 
doubling divisors. 
Then the canonical sections
$1_{C_0}, \ldots, 1_{C_r}$ are of 
degree one in the Cox ring 
$
\mathcal{R}(Y) 
= 
\KK[U_0,\ldots,U_c].
$
Suitably enhancing the general 
hyperplane arrangement
$C_0, \ldots, C_r$, we achieve 
that $1_{C_0}, \ldots, 1_{C_r}$ 
generate $\mathcal{R}(Y)$.
Regard $\mathcal{R}(Y)$ 
as a graded subalgebra of 
$\mathcal{R}(X)$ as in 
Corollary~\ref{cor:RYinRX}
and let $\alpha = (f_0,\ldots,f_q)$
be pairwise non-associated 
$\Cl(X)$-prime generators 
of the Cox ring $\mathcal{R}(X)$ 
such that the $f_i$ lying in 
$\mathcal{R}(Y)$ are precisely 
$1_{C_0}, \ldots, 1_{C_r}$.
Then, following the lines of 
the proof of Theorem~\ref{thm:t-mds-2}, 
one reproduces~$X$ as an explicit 
$\TT$-variety 
$X(\alpha,P,\Sigma)$ in $Z = Z_\Sigma$.
Thus, Proposition~\ref{rem:gen2hyp} 
gives the assertion.
\end{proof}

\begin{remark}
Let $X$ be a general arrangement variety of complexity~$c$.
Then the torus action of $X$ has $\PP_c$ as Chow quotient; 
use~\cite{BaHaKe}*{Props.~2.4 and~2.5} for a proof.
Using the conversions Remark~\ref{rem:expl2poldiv}
and~\cite[Thm.~4.8]{HaSu}, we see that the general 
arrangement varieties are precisely the $\TT$-varieties 
arising from a divisorial fan $\Xi$ on a projective 
space $\PP_c$ in the 
sense of~\cite{AlHaSu} such that the prime divisors 
$D \subseteq \PP_c$ with non-trivial slices $\Xi_D$
form a general hyperplane arrangement in~$\PP_c$.
\end{remark}

\section{Examples and first properties}
\label{sec:exfirstprops}

We discuss examples and study basic structural
properties of general arrangement varieties.
For instance, we investigate torsion in the 
divisor class group, describe the canonical
class and give a combinatorial characterization 
of the $X$-cones which in turn leads to the
combinatorial smoothness criterion provided
by Corollary~\ref{cor:genarrvarsmooth}.
Moreover, we specify constraints on 
the defining data of an explicit general 
arrangement variety imposed by conditions 
on the singularities, preparing the 
classification performed in 
Section~\ref{sec:FanoClass}.
As a first concrete application, we prove 
at the end of this section that the smooth 
projective general arrangement varieties 
of Picard number one are just the classical 
smooth projective quadrics; see 
Proposition~\ref{prop:arrvar-rho-one}.

We begin with the examples.
The first one shows how to realize
intrinsic quadrics as explicit general 
arrangement varieties.
Recall from~\cite{BerHa2} that an
\emph{intrinsic quadric} is a normal 
projective variety with a Cox ring 
defined by a single quadratic relation.
The intrinsic quadrics form a 
playground immediately adjacent to 
the one given by the projective toric 
varieties, which have a polynomial 
ring as Cox ring.
We mention Bourqui's work~\cite{Bou} proving 
Manin's conjecture for the full intrinsic 
quadrics and the classification results 
on smooth (Fano) intrinsic quadrics 
of low Picard number in~\cite{FaHa}
as examples for research in this field.

\begin{example}
\label{ex:intrquad}
The normal form for graded quadrics
provided by~\cite{FaHa}*{Prop.~2.1}
shows that we can represent 
every intrinsic quadric as an explicit 
general arrangement variety $X \subseteq Z$ 
with defining matrix $P$ having left 
upper block 
$$ 
\left[
\begin{array}{rrrr}
-l_0 & l_1 & & 0
\\
\vdots & & \ddots & 
\\
-l_0 &  0 & & l_r
\end{array}
\right],
\qquad
l_0 = \ldots = l_q = (1,1),
\quad
l_{q+1} = \ldots = l_r = (2),
$$ 
where $-1 \le q \le r$ and the 
variables $T_{i1}$ with $i = q+1, \ldots, r$
have pairwise distinct $K$-degrees.
Moreover, for the dimension of $X$, 
the rank of the divisor class group 
and the complexity of the torus action on~$X$,
we have
$$ 
\dim(X) \ = \ r-1+s,
\qquad
\rk(\Cl(X)) \ = \ m+q+2-s,
\qquad
c(X) \ = \  r-1.
$$
\end{example}

In the second example we exhibit a series 
of general arrangement varieties producing 
many Fano examples.
We pick up these varieties again in 
Example~\ref{ex:PrtimesPr2}, 
when the necessary methods 
are available to figure out the smooth 
Fano varieties.

\begin{example}
\label{ex:PrtimesPr}
Fix integers $r > c \ge 1$. 
Consider the product $Z = \PP_{r} \times \PP_{r}$
and the intersection 
$X = V(g_1) \cap \ldots \cap V(g_{r-c}) \subseteq Z$ 
of the $r-c$ divisors of bidegree $(a,b)$ 
in $Z$ given by
\begin{eqnarray*}
g_1  
& = & 
\lambda_{1,0}T_{01}^a T_{02}^b+\lambda_{1,1}T_{11}^aT_{12}^b+ \ldots + \lambda_{1,c}T_{c1}^aT_{c2}^b + T_{c+1,1}^aT_{c+1,2}^b,
\\
& \vdots &
\\
g_{r-c}  
& = & 
\lambda_{r-c,0}T_{01}^aT_{02}^b+ \lambda_{r-c,1} T_{11}^aT_{12}^b + \ldots + \lambda_{r-c,c}T_{c1}^aT_{c2}^b +T_{r1}^aT_{r2}^b,
\end{eqnarray*}
where $a,b > 0$ are coprime integers
and any $c+1$ of the vectors 
$\lambda_i = (\lambda_{i,0}, \ldots, \lambda_{i,c})$
are linearly independent.
Observe that for $r > c+1$,
the divisors 
$V(g_i) \subseteq Z$ are singular.
We realize $X \subseteq Z$ as an
explicit general arrangement variety.
Let~$P$ be the stack matrix
with upper and lower blocks
\begin{align*}
P_0 
\ &= \  
\left[
\begin{array}{rrrr}
-l_0 & l_1 & & 0
\\
\vdots & & \ddots & 
\\
-l_0 &  0 & & l_r
\end{array}
\right],
\qquad
l_0 = \ldots = l_r = (a,b),
\\
d 
\ &= \  
\left[
\begin{array}{rrrr}
-d_0 & d_1 & & 0
\\
\vdots & & \ddots & 
\\
-d_0 &  0 & & d_r
\end{array}
\right],
\qquad
d_0 = \ldots = d_r = (v,u),
\end{align*}
where $u$ and $v$ are integers with 
$ua - vb = 1$.
We claim that there is precisely one 
complete fan $\Sigma$ with generator 
matrix $P$ and the associated 
toric variety $Z = Z_\Sigma$ is the 
product $\PP_r \times \PP_r$.
Indeed, consider the matrices
$$
\left[
\begin{array}{rr}
u\cdot \mathds{E}_r & - b\cdot \mathds{E}_r
\\
- v\cdot \mathds{E}_r & a \cdot \mathds{E}_r
\end{array}
\right],
\qquad\qquad
\left[
\begin{array}{rrrr}
- \mathds{1}_r & \mathds{E}_r & 0 & 0
\\
0 & 0 & - \mathds{1}_r & \mathds{E}_r
\end{array}
\right].
$$
The first one is unimodular and
multiplying it from the left to 
$P$ yields, after suitably renumbering 
columns, the second one.
Now, choosing a suitable $(c+1) \times (r+1)$
matrix $A$, we obtain the above relations 
as the output of Construction~\ref{constr:R(A,P)}.
Thus, $X = X(A,P,\Sigma)$ is of dimension $r+c$ and 
comes with an effective $r$-torus action.
\end{example}

We enter the study of structural properties
of explicit general arrangement varieties
$X \subseteq Z$ as provided by
Construction~\ref{constr:R(A,P)}.
We will freely use the notation fixed there.
Our first observation is that there may occur 
unavoidable torsion in the divisor class 
group.

\begin{proposition}
\label{prop:zwangstors}
Let $X \subseteq Z$ be an
explicit general arrangement
variety.
Then $\ZZ^r / \mathrm{im}(P_0)$ 
is a finite subgroup of the divisor 
class group $\Cl(X)$. 
\end{proposition}

\begin{proof}
The divisor class group of $X$ equals 
$K = \ZZ^{n+m} /  \mathrm{im}(P^*)$. 
Moreover, $\ZZ^r / \mathrm{im}(P_0)$ 
is the torsion part $K_0^{\mathrm{tors}}$
of the factor group $K_0 = \ZZ^{n+m} /  \mathrm{im}(P_0^*)$. 
Applying the snake Lemma to the exact 
sequences arising from $P_0^*$ and $P^*$ 
yields that the kernel of $K_0 \to K$ 
injects into $\ZZ^s$.
Consequently, the torsion part $K_0^{\mathrm{tors}}$
maps injectively into $K$.
\end{proof}

In Remark~\ref{rem:kerA2rel}, we observed 
that $R(A,P)$ is a complete intersection
ring.
Thus, we can apply Proposition~\ref{prop:ci}
and obtain the following description 
of the canonical class.

\begin{proposition}
\label{prop:genarrcandiv}
Let $X \subseteq Z$ be an explicit general arrangement 
variety of complexity $c(X) = c$. 
Then the canonical class of $X$ is given in 
terms of the generator degrees $w_{ij} = \deg(T_{ij})$ 
and $w_k = \deg(S_k)$ as 
$$ 
\mathcal{K}_X
\ = \
-
\sum_{i = 0}^r \sum_{j = 1}^{n_i} w_{ij} 
\, - \,
\sum_{k = 0}^r w_k
\, + \,
(r-c) \sum_{j=1}^{n_0} l_{0j}w_{0j}
\ \in \ 
K
\ = \ 
\Cl(X).
$$
\end{proposition}

\begin{example}
\label{ex:PrtimesPr2}
Consider again the explicit general 
arrangement varieties $X \subseteq Z$ 
discussed in Example~\ref{ex:PrtimesPr}.
The degree matrix $Q$ is
$$ 
Q
\ = \
[w_{01},w_{02}, \ldots, w_{r1},w_{r2}]
\ = \
[\mathds{E}_2, \ldots,\mathds{E}_2]. 
$$
Thus, Proposition~\ref{prop:PicandCones}
tells us $\Eff(X) = \SAmple(X) = \cone(e_1,e_2)$.
Moreover, the anticanonical class of $X$ is 
given by
$$
-\mathcal{K}_X 
\ = \ 
(r+1 - (r-c)a, \, r+1 - (r-c)b) 
\ \in \ 
\Cl(X) 
\ = \
\ZZ^2.
$$
as we infer from Proposition~\ref{prop:genarrcandiv}.
In particular, $X$ is a Fano variety
if and only if the following two conditions 
are satisfied
$$
a 
\ < \
\frac{r+1}{r-c},
\qquad\qquad
b
\ < \ 
\frac{r+1}{r-c}.
$$ 
\end{example}

Recall that in Definition~\ref{def:Xcone}
we introduced for any explicit variety 
$X \subseteq Z$ the $X$-cones as those 
cones $\sigma \in \Sigma$ of the defining 
fan of $Z$ such that $X$ intersects the 
corresponding orbit $\TT_Z \cdot z_\sigma$ 
non-trivially.
For explicit general arrangement 
varieties $X \subseteq Z$, we may 
determine the $X$-cones in a simple 
purely combinatorial way.

\begin{definition}
Consider the setting of
Construction~\ref{constr:R(A,P)}
and let $\sigma\in \Sigma$.
We say that the cone $\sigma$ is
\begin{enumerate}
\item
\emph{big (elementary big)} 
if $\sigma$ contains at least 
(precisely) one column $v_{ij}$ 
of~$P$ for every $i = 0, \dots, r$,
\item
a \emph{leaf cone}
if there is a set $I_\sigma = \{i_1, \ldots, i_c\}$ 
of indices $0 \le i_1 < \ldots < i_c \le r$ 
such that for any $i$, we have 
$v_{ij} \in \sigma \Rightarrow i \in I_\sigma$.
\end{enumerate}
\end{definition}

\begin{remark}
\label{rem:big2bigtype}
Situation as in Construction~\ref{constr:R(A,P)}.
Given $\sigma \in \Sigma$, let 
$\gamma_0 \preccurlyeq \gamma$ be the 
corresponding face, that means that
$\sigma = P(\gamma_0^*)$ holds.
Then $\sigma$ is a big (leaf) cone 
if and only if the toric orbit
$\TT^{n+m} \cdot z_{\gamma_0^*} \subseteq \KK^{n+m}$ 
consists of points of big (leaf) type 
in the sense of Definition~\ref{def:bigleaftype}.
\end{remark}

\begin{proposition}
\label{prop:bigleaf2rel}
Let $X \subseteq Z$ be an
explicit general arrangement
variety.
Then, for every $\sigma \in \Sigma$, 
the following statements are 
equivalent.
\begin{enumerate}
\item
The cone $\sigma$ is an $X$-cone.
\item
The cone $\sigma$ is big or a 
leaf cone.
\end{enumerate}
\end{proposition}

\begin{proof}
Consider the face $\gamma_0 \preccurlyeq \gamma$ 
with $P(\gamma_0^*) = \sigma$.
By Remark~\ref{rem:big2bigtype}, our $\sigma$
is a big (leaf) cone if and only if 
$\bar{X}(\gamma_0)$ consists of points
of big (leaf) type.  
The assertion thus follows from 
Lemma~\ref{lem:pointsofbarX}.
\end{proof}

\begin{example}
\label{ex:gorensteinfano}
We look again at $X = X(A,P,\Sigma)$ in 
$Z = Z_\Sigma$ from 
Examples~\ref{ex:R(A,P_0)} and~\ref{ex:R(A,P)}.
Recall that we have 
$$
P 
\ =  \
[v_{01},v_{02},v_{11},v_{21},v_{31}]
\ = \ 
\left[
\begin{array}{rrrrr}
-1 & -2 & 2 & 0 & 0 
\\
-1 & -2 & 0 & 2 & 0
\\
-1 & -2 & 0 & 0 & 4
\\
-1 & -3 & 1 & 1 & 1 
\end{array}
\right].
$$
Except $\cone(v_{01},v_{02},v_{11},v_{21},v_{31})$,
every cone generated by some of the $v_{ij}$ 
occurs in the fan $\Sigma$. 
In particular, $\Sigma$ has two big cones
$$ 
\sigma_1
\ = \
\cone(v_{01},v_{11},v_{21},v_{31}),
\qquad\qquad
\sigma_2
\ = \
\cone(v_{02},v_{11},v_{21},v_{31}),
$$
and six maximal leaf cones
$$ 
\tau_1
\ = \ 
\cone(v_{01},v_{02},v_{11}),
\quad
\tau_2
\ = \ 
\cone(v_{01},v_{02},v_{21}),
\quad
\tau_3
\ = \ 
\cone(v_{01},v_{02},v_{31}),
$$
$$
\tau_4
\ = \ 
\cone(v_{11},v_{21}),
\qquad
\tau_5
\ = \ 
\cone(v_{11},v_{31}),
\qquad
\tau_6
\ = \ 
\cone(v_{21},v_{31}).
$$
Thus, by Proposition~\ref{prop:bigleaf2rel}
the $X$-cones of $\Sigma$ are $\sigma_1, \sigma_2$ 
and the faces of $\tau_1,\ldots,\tau_6$.
This allows us to determine the Picard group
$\Pic(X)$. 
Recall the degree matrix
$$ 
Q
\ =  \
[w_{01},w_{02},w_{11},w_{21},w_{31}]
\ =  \
\left[
\begin{array}{ccccc}
2 & 1 & 2 & 2 & 1
\\
\bar 0 & \bar 0 & \bar 1 & \bar 1 & \bar 0 
\\
\bar 0 & \bar 1 & \bar 0 & \bar 1 & \bar 0 
\end{array}
\right],
$$
having the generator degrees 
$w_{ij} = Q(e_{ij}) 
\in 
K = \ZZ \times \ZZ/2\ZZ \times \ZZ/2\ZZ$ 
as its columns.
The $X$-faces corresponding to the $X$-cones 
$\sigma_1,\sigma_2,\tau_1,\tau_2,\tau_3$
are
$$ 
\gamma_1  \ = \ \cone(e_{02}),
\qquad
\gamma_2 \ = \ \cone(e_{02}),
$$
$$
\eta_1 \ = \ \cone(e_{11},e_{21}),
\quad
\eta_2 \ = \ \cone(e_{11},e_{21}),
\quad
\eta_3 \ = \ \cone(e_{21},e_{31}).
$$
Observe that these are precisely the minimal 
ones among all $X$-faces 
$\gamma_0 \preccurlyeq \gamma = \QQ_{\ge 0}^5$. 
Thus, Proposition~\ref{prop:PicandCones} yields
$$ 
\Pic(X) 
\ = \
\bigcap_{i=1}^2 Q(\lin_\QQ(\gamma_i) \cap \ZZ^4)
\, \cap \, 
\bigcap_{i=1}^3 Q(\lin_\QQ(\eta_i) \cap \ZZ^4)
\ = \ 
\ZZ \cdot (4, \bar 0, \bar 0)
\ \subseteq \
\Cl(X).
$$
Using Proposition~\ref{prop:genarrcandiv},
we see that $(4, \bar 0, \bar 0)$ equals 
the anticanonical class~$-\mathcal{K}_X$.
In particular, $X$ is a Gorenstein Fano 
threefold.
\end{example}

Big and leaf cones admit 
also simple characterizations in terms
of the geometry of the defining fan
of the ambient toric variety.

\begin{remark}
\label{rem:bigchar}
Consider the setting of
Construction~\ref{constr:R(A,P)}
and let $L \subseteq \ZZ^{r+s}$ 
be the kernel of the projection
$\pr \colon \ZZ^{r+s} \to \ZZ^{s}$.
Then, for any $\sigma \in \Sigma$,
the following statements 
are equivalent.
\begin{enumerate}
\item 
The cone $\sigma$ is big.
\item
We have $\pr(\sigma) = \QQ^{r}$.
\item
We have $\sigma \not \subseteq L_\QQ$ 
and $\sigma^\circ \cap L_\QQ \ne \emptyset$.
\end{enumerate}
Moreover, $\sigma \in \Sigma$ is a leaf cone 
if and only if its image $\pr(\sigma) \subseteq \QQ^{r}$
is a pointed cone of dimension at most~$c$.
\end{remark}

\begin{proposition}
Situation as in
Construction~\ref{constr:R(A,P)}.
Let $L \subseteq \ZZ^{r+s}$ 
be the kernel of the projection
$\pr \colon \ZZ^{r+s} \to \ZZ^{s}$
and $\Sigma_L$ the fan in 
$\ZZ^{r+s}$ consisting of all the
faces of the cones $\sigma \cap L_\QQ$,
where $\sigma \in \Sigma$.
Then the following statements 
are equivalent.
\begin{enumerate}
\item
$\Sigma_L$ is a subfan of $\Sigma$.
\item
$\Sigma$ contains no big cone.
\item
$\Sigma$ consists of leaf cones.
\end{enumerate}
\end{proposition}

\begin{proof}
The equivalence of~(ii) and~(iii) 
is clear. 
We prove ``(i)$\Rightarrow$(ii)''.
Assume that there is a big cone
$\sigma \in \Sigma$.
Then $\sigma \cap L_\QQ$ belongs 
to $\Sigma_L$ but not to $\Sigma$
according to~\ref{rem:bigchar}~(iii);
a contradiction.
We turn to  ``(ii)$\Rightarrow$(i)''.
The task is to show that for 
every cone $\sigma \in \Sigma$, the
intersection $\sigma \cap L_\QQ$
is a face of $\sigma$.
Let $\tau \preccurlyeq \sigma$ be 
the minimal face containing 
$\sigma \cap L_\QQ$.
Then $\tau^\circ \cap L_\QQ$ is 
non-empty.
Since $\tau \in \Sigma$ is not big, 
we can use~\ref{rem:bigchar}~(iii)
to conclude $\tau \subseteq L_\QQ$.
This means $\sigma \cap L_\QQ = \tau$.
\end{proof}

We use the the concrete description 
of $X$-cones as big cones and leaf 
cones to study (quasi)smoothness 
properties of explicit general 
arrangement varieties $X \subseteq Z$.
First, let us define quasismoothness.

\begin{definition}
Let $X \subseteq Z$ be an explicit 
$\TT$-variety. 
We say that $x \in X$ is a \emph{quasismooth}
point of $X$ if the fiber $p^{-1}(x) \subseteq \hat X$ 
consists of smooth points of $\bar X$.
\end{definition}

\begin{remark}
\label{rem:quasismoothchar}
Let $X \subseteq Z$ be an explicit 
$\TT$-variety, $\sigma \in \Sigma$ 
an $X$-cone and 
$\gamma_0 \preccurlyeq \gamma$ the face 
with $P(\gamma_0^*) = \sigma$. 
Then, for $x \in X(\sigma)$, 
the intersection 
$p^{-1}(x) \cap \bar X(\gamma_0)$ 
equals the closed orbit 
$H \cdot z$ of $p^{-1}(x)$.
In particular, $x \in X$ is 
quasismooth if and only if
$z \in \bar X$ is smooth.
Moreover, $X(\sigma)$ consists 
of quasismooth points of $X$ 
if and only if $\bar X(\gamma_0)$
consists of smooth points of $\bar X$.
\end{remark}

\begin{proposition}
\label{prop:genarrvarqsmooth}
Let $X \subseteq Z$ be an explicit general 
arrangement variety.
\begin{enumerate}
\item
For every big cone $\sigma \in \Sigma$,
the following statements are equivalent.
\begin{enumerate}
\item
There is a quasismooth point of $X$ in the piece 
$X(\sigma) \subseteq X$. 
\item
The piece $X(\sigma) \subseteq X$ 
consists of quasismooth points of $X$.
\item
Every sequence $0 \le i_1 < \ldots < i_{c+2} \le r$
admits $1 \le q \le c+2$ and $1 \le j \le n_{i_q}$
such that
$v_{i_qj} \in \sigma$, $l_{i_qj} = 1$
and $v_{i_qk} \not\in \sigma$
for all $k \ne j$.
\end{enumerate}
\item
For every leaf cone $\sigma \in \Sigma$,
the piece $X(\sigma) \subseteq X$ 
consists of quasismooth points of $X$.
\end{enumerate}
\end{proposition}

\begin{proof}
According to Remark~\ref{rem:quasismoothchar},
we just have to care about smoothness of the
points of~$\bar X(\gamma_0)$.
By Remark~\ref{rem:kerA2rel}, 
a point $z \in \bar X(\gamma_0)$ 
is smooth if and only if the 
Jacobian $J(z)$ of $g_1,\ldots, g_{r-c}$ 
is of full rank.
The latter is characterized via
Lemma~\ref{lem:singloc}~(ii).
In particular, we see that 
in the case of a leaf cone $\sigma$,
all points of $\bar X(\gamma_0)$ 
are smooth, proving~(ii).
To show~(i), let $\sigma$ be big.
By the nature of 
Condition~\ref{lem:singloc}~(ii),
there is a smooth point of $\bar X$ 
in $\bar X(\gamma_0)$ if and only if 
every point of $\bar X(\gamma_0)$ 
is smooth in~$\bar X$.
This establishes the equivalence 
of~(a) and~(b).
The equivalence of~(a) and~(c)
is obtained by negating 
Condition~\ref{lem:singloc}~(ii)
for a point $z$ of big type.
\end{proof}

\begin{corollary}
\label{prop:niliji}
Let $X \subseteq Z$ be a quasismooth 
explicit general arrangement variety.
Assume that~$P$ is irredundant
and let $\sigma=\cone(v_{0j_0}+\ldots+ v_{rj_r})$
be an elementary big cone of 
$\Sigma$.
\begin{enumerate}
\item
We have $l_{ij_i} \ge 2$ for at most $c+1$ 
different $i=0,\dots,r$.
\item
We have $n_i = 1$ for at most  $c+1$ 
different $i=0,\dots,r$.
\end{enumerate}
\end{corollary}

Combining Proposition~\ref{prop:genarrvarqsmooth},
Remark~\ref{rem:quasismoothchar} and 
Proposition~\ref{prop:smoothchar}
leads to the following purely combinatorial
smoothness criterion for explicit general 
arrangement varieties.

\begin{corollary}
\label{cor:genarrvarsmooth}
Let $X \subseteq Z$ be an explicit general 
arrangement variety.
\begin{enumerate}
\item
Let $\sigma \in \Sigma$ be a big cone and 
$\gamma_0 \preccurlyeq \gamma$ 
the corresponding $X$-face.
Then the following statement are 
equivalent.
\begin{enumerate}
\item 
There is a smooth point of $X$ in the piece 
$X(\sigma) = X(\gamma_0) \subseteq X$. 
\item
The piece $X(\sigma) = X(\gamma_0) \subseteq X$ 
consists of smooth points of $X$.
\item
The cone $\sigma$ is regular 
and~\ref{prop:genarrvarqsmooth}~(i)~(c) holds.
\item
We have $K = Q(\lin_\QQ(\gamma_0) \cap \ZZ^{n+m})$
and~\ref{prop:genarrvarqsmooth}~(i)~(c) holds.
\end{enumerate}
\item
Let $\sigma \in \Sigma$ be a leaf cone and 
$\gamma_0 \preccurlyeq \gamma$ 
the corresponding $X$-face.
Then the following statements are 
equivalent.
\begin{enumerate}
\item 
There is a smooth point of $X$ in the piece 
$X(\sigma) = X(\gamma_0) \subseteq X$. 
\item
The piece $X(\sigma) = X(\gamma_0) \subseteq X$ 
consists of smooth points of $X$.
\item
The cone $\sigma$ is regular.
\item
We have $K = Q(\lin_\QQ(\gamma_0) \cap \ZZ^{n+m})$.
\end{enumerate}
\end{enumerate}
\end{corollary}

\begin{example}
\label{ex:gorensteinfano2}
Let us continue the discussion 
of $X \subseteq Z$  
from~\ref{ex:R(A,P_0)} and~\ref{ex:R(A,P)}.
In~\ref{ex:gorensteinfano}, 
we determined the maximal $X$-cones:
there are two big cones 
$\sigma_1$, $\sigma_2$ 
and six maximal leaf cones 
$\tau_1, \ldots, \tau_6$.
Using Corollary~\ref{cor:genarrvarsmooth},
we see that the associated pieces are 
precisely those consisting of singular 
points of~$X$.
Note that 
$$
\overline{X(\tau_i)}  
\ = \ 
X(\sigma_1) \ \cup X(\tau_i) \cup \ X(\sigma_2),
\qquad i = 1,2,3,
$$ 
are curves, each being the 
closure of the $\TT$-orbit $X(\tau_i)$;
use Proposition~\ref{prop:isogroups}.
The union over these $\overline{X(\tau_i)}$ 
is a connected component of the singular 
locus of $X$. 
In addition, we have $X(\tau_i)$ for $i = 4,5,6$,
each consisting of an isolated singularity.
\end{example}

\begin{example}
\label{ex:PrtimesPr2}
We continue Example~\ref{ex:PrtimesPr}.
Suitably renumbering the variables 
we achieve $a \geq b$. 
We claim that $X$ is smooth if and only if 
one of the following conditions is satisfied:
\begin{enumerate}
\item 
$r= c+1$, $a \geq 1$ and $b = 1$,
\item 
$r=c+2$ and $a = b = 1$.
\end{enumerate}
Indeed, $r \le c+2$ holds, because 
otherwise the big cone $\sigma \in \Sigma$ 
generated by all $v_{ij}$ with $i \le r-2$
and $v_{r-1 \, 1}$, $v_{r2}$ yields a 
singular point in $X(\sigma)$.
Now, the elementary big cones 
of $\Sigma$ are precisely 
$\cone(v_{0j_0}, \ldots,  v_{rj_r})$,
where $\{j_0, \ldots,  j_r\}$ equals 
$\{1,2\}$, and the claim follows
from Corollaries~\ref{prop:niliji}
and~\ref{cor:genarrvarsmooth}.
In particular we obtain smooth Fano
varieties in the cases  
\begin{enumerate}
\item[(iii)] 
$r= c+1$, $1 \le a \le r$ and $b = 1$,
\item[(iv)] 
$r=c+2$ and $a = b = 1$.
\end{enumerate}
\end{example}

Finally, we observe constraints on the defining data 
of an explicit general arrangement variety 
arising from local factoriality and $\QQ$-factoriality.

\begin{proposition}
\label{prop:leafy2ni>2}
Let $X \subseteq Z$ be a
locally factorial explicit 
general arrangement variety,
where~$P$ is irredundant.
Assume that $\Sigma$ consists 
of leaf cones
and each of the sets $cone(v_{i1}) + L_\QQ$ 
is covered by cones of $\Sigma$.
Then $n_i \ge 2$ holds for 
$i = 0, \ldots, r$.
\end{proposition}

\begin{proof}
Assume that $n_i = 1$ holds for some $i$.
Let $\varrho$ denote the ray through 
$v_{i1}$ and consider the cone 
$\tau := \varrho + L_\QQ$.
We claim that for every $\sigma \in \Sigma$,
the intersection $\tau \cap \sigma$ is a 
face of $\sigma$.
Indeed, as $\Sigma$ consists of leaf cones, 
the image of $\pr(\sigma)$ under the projection 
$\pr \colon \QQ^{r+s} \to \QQ^{r}$ is a pointed
cone, having $\pr(\varrho)$ as an extremal ray.
Thus, $\tau = \pr^{-1}(\pr(\varrho))$ cuts out 
a face from $\sigma$.

By our assumptions, the above claim implies that 
$\tau = \varrho + L_\QQ$ is a union of cones of~$\Sigma$.
Any cone of~$\Sigma \setminus \Sigma_L$ 
contained in $\tau$ is necessarily of the form 
$ \varrho + \sigma_L \in \Sigma$
with $\sigma_L \in \Sigma_L$.
We conclude that in particular all the cones 
$\sigma = \varrho + \sigma_L$, where 
$\dim(\sigma_L) = s$, must belong to $\Sigma$.
As $\sigma$ and $\sigma_L$ are leaf cones,
they are $X$-cones by 
Proposition~\ref{prop:bigleaf2rel}.
Thus, Proposition~\ref{prop:smoothchar}
yields that $\sigma$ and $\sigma_L$
are regular.
This implies $l_{i1} = 1$; a contradiction 
to the assumption that $P$ is irredundant.
\end{proof}

\begin{corollary}
\label{cor:leafy2rho}
Let $X \subseteq Z$ be a non-toric, 
projective, locally factorial 
explicit general arrangement
variety.
If $\Sigma$ consists of leaf cones,
then the Picard number and the 
complexity of~$X$ satisfy
$$
\rho(X) 
\ \ge \ 
 r+3
\ \ge \ 
c(X) + 4.
$$
\end{corollary}

\begin{proof}
Since $X$ is non-toric, we may assume 
that $P$ is irredundant with $r > c$.
Moreover, as $X$ is projective, 
we may assume that $\Sigma$ is complete.
Thus, Proposition~\ref{prop:leafy2ni>2}
applies and we obtain $n \ge 2r +2$.
Then Corollary~\ref{cor:upperrhobound1}
yields the desired estimate.
\end{proof}

\begin{proposition}
\label{prop:big2elbig}
Let $X \subseteq Z$ be a $\QQ$-factorial
explicit general arrangement variety.
If $\Sigma$ admits a big cone, then
it admits an elementary big cone.
\end{proposition}

\begin{proof}
Let $\sigma \in \Sigma$ be a big cone.
Then $\sigma$ is an $X$-cone according 
to Proposition~\ref{prop:bigleaf2rel}.
Proposition~\ref{prop:Qfactchar} tells 
us that $\sigma$ is simplicial.
Now, any elementary big face of $\sigma$ 
is as wanted.
\end{proof}

\begin{corollary}
\label{cor:ebc}
Let $X \subseteq Z$ be a non-toric, 
projective, locally factorial 
explicit general arrangement variety.
If $X$ is of Picard number $\rho(X) \le c+3$,
then $\Sigma$ admits an elementary big cone.
\end{corollary}

We conclude the section with a closer look 
at smooth projective general arrangement 
varieties of Picard number one.

\begin{proposition}
\label{prop:arrvar-rho-one}
Let~$X$ be a non-toric, smooth, 
projective general arrangement variety 
of Picard number one.
Then $X$ is a quadric
$V(T_0^2 + \ldots + T_r^2) \subseteq \PP_r$.
\end{proposition}

\begin{proof}
According to~Theorem~\ref{thm:arrTvar},
it suffices to consider the explicit 
general arrangement varieties 
$X \subseteq Z$.
Moreover, we may assume that
$P$ is irredundant
and $n_0 \ge \ldots \ge n_r$ holds.
Finally, we have $K_\QQ = \QQ$ and 
may assume that the effective cone 
of $X$ is~$\QQ_{\ge 0}$.

First we show that the are no variables 
of type $S_k$ in $R(A,P)$.
Otherwise, 
$\gamma_1 = \cone(e_1) \preccurlyeq \gamma$
is an $X$-face and thus we find
a point $x \in \hat X$ 
having $x_1$ as its only nonzero 
coordinate.
By smoothness of $X$, the Jacobian 
of $g_1, \ldots, g_{r-c}$ 
does not vanish at~$x$;  
see Proposition~\ref{prop:smoothchar}.
This implies 
$l_{i1} + \ldots + l_{in_i} = 1$
for some $i$; a contradiction
to irredundance of $P$.

According to Corollary~\ref{cor:ebc},
the fan $\Sigma$ admits an elementary 
big cone. 
Proposition~\ref{prop:niliji} tells us  
$n_0 \ge 2$.
Thus $\gamma_{0j} = \cone(e_{0j}) \preccurlyeq \gamma$ 
is an $X$-face.
Proposition~\ref{prop:smoothchar} 
yields that $\deg(T_{0j})$ generates~$K$. 
We conclude $K = \ZZ$ and $\deg(T_{0j})= 1$.
Additionally, smoothness of $X(\gamma_{01})$ 
implies that $\grad(g_1)(x) \ne 0$ holds 
for every point 
$x \in \bar{X}(\gamma_{01})$.
We conclude $n_0 = 2$ and 
$\deg(g_1) = 2$.
This implies $\deg(T_{ij}) = 1$ 
and for all $i$, $j$, we obtain 
$l_{ij} = 1$ or $l_{ij} = 2$ 
according to $n_i=2$ or $n_i=1$.

Finally, observe that $c = r-1$ holds,
i.e., that there is only one defining 
relation. Indeed, otherwise, we find 
generators $g_1', \ldots, g_{r-c}'$,
each involving precisely $c+2$ monomials
and $g_{r-c}'$ all different from $T_0^{l_0}$.
Then the corresponding Jacobian 
vanishes at any $x \in \bar{X}(\gamma_{01})$,
showing that $X(\gamma_{01})$ is singular.
A contradiction.
\end{proof}

\begin{remark}
Consider a smooth, projective 
explicit general arrangement variety 
$X \subseteq Z$ of Picard number 
one with $P$ being irredundant.
By Proposition~\ref{prop:arrvar-rho-one},
the divisor class group $\Cl(X)$ is torsion 
free.
Thus, Proposition~\ref{prop:zwangstors}
yields
$$ 
P_0
\ = \ 
\left[
\begin{array}{rrrr}
-l_0 & l_1 & & 0
\\
\vdots & & \ddots & 
\\
-l_0 &  0 & & l_r
\end{array}
\right],
\qquad
l_0 = \ldots = l_{r-1} = (1,1),
\quad
l_r
\ = \ 
\begin{cases}
(1,1), &  n \text{ even},
\\
(2), &  n \text{ odd},
\end{cases}
$$ 
where $n = n_0+ \ldots + n_r$.
Moreover, the torus action on $X$ 
is the action of the maximal torus 
of $\Aut(X) = \mathrm{O}(n)$. 
In particular, the torus action on
$X$ is of complexity
$$ 
c 
\ = \ 
\begin{cases}
\frac{n}{2}-2, & \quad n \text{ even},
\\
\frac{n-1}{2}-1, & \quad n \text{ odd}.
\end{cases}
$$
\end{remark}

\section{Smooth arrangement varieties of 
complexity and Picard number~2}

This section is devoted to the proof 
of Theorem~\ref{theo:proj}, which presents 
all smooth projective general arrangement 
varieties of true complexity two and Picard 
number two. 
Theorem~\ref{thm:arrTvar} allows us to 
work in the setting of explicit general 
arrangement varieties 
$X = X(A,P,\Sigma)$ in $Z = Z_\Sigma$.
The task is then to specify further and further 
the defining data until we arrive at the 
list of Theorem~\ref{theo:proj},
where we follow a similar strategy as 
in the case of complexity one~\cite{FaHaNi}.

Let us indicate the basic principles 
behind the proof of Theorem~\ref{theo:proj}
and see how smoothness and Picard number 
two come into the game.
Recall from Definition~\ref{def:Xcone} the 
notions of $\bar X$-faces and $X$-faces
of the orthant $\gamma = \QQ_{\ge 0}^{n+m}$ 
and look at the degree map 
$Q \colon \ZZ^{n+m} \to K = \Cl(X)$,
having the image of~$P^*$ as its kernel.
The arguments for establishing the list 
of Theorem~\ref{theo:proj} rely on the following.
\begin{itemize}
\item 
Due to the explicit nature of the defining 
relations $g_1, \ldots, g_{r-2}$ of $\bar X$,
we can determine the $\bar X$-faces 
$\gamma_0 \preccurlyeq \gamma$ and characterize 
smoothness of the points of 
the associated pieces 
$\bar X(\gamma_0) \subseteq \bar X$
in terms of defining data.
\item 
Since $K_\QQ = \Cl_\QQ(X)$ is of dimension 
two, the positions of the generator 
degrees $Q(e_{ij})$ and $Q(e_k)$ with respect 
to the ample cone are sufficiently accessible, 
so that we can figure out possible 
$X$-faces.
\item
Smoothness of $X$ is characterized
by the fact that for every $X$-face 
$\gamma_0 \preccurlyeq \gamma$,
the image $Q(\gamma_0 \cap \ZZ^{n+m})$ 
generates $K$ as a group and 
$\bar X(\gamma_0)$ consists of smooth points
of $\bar X$.
\end{itemize}
As it turns out, the fact that $X$ is of 
Picard number two allows us to extract 
enough information from the $\bar X$-faces 
of dimension at most two in order to 
determine at the end all possibilities for 
the degree map~$Q$ and the ample class 
of $X$, which in turn lead to the possible 
matrices~$P$ and fans~$\Sigma$.
In order to obtain that all varieties of 
Theorem~\ref{theo:proj} are 
indeed of true complexity two, we have 
show that there no toric examples and 
no examples of complexity one in the list.
This will be done by comparing geometric
data, which in the complexity one case
we import from~\cite{FaHaNi}.

Let us enter the first step towards the list 
of Theorem~\ref{theo:proj}.
We will bound the number~$r$
which in turn bounds the size of the 
matrix~$A$ and the number of defining 
relations of the Cox ring 
$\mathcal{R}(X) = R(A,P)$.
Moreover, this step provides first 
specifications on the numbers
$n_0, \ldots, n_r$ and $m$, which 
sum up to the number of columns of 
the matrix~$P$.
Here comes the precise formulation.

\begin{proposition}
\label{prop:const}
Let $X \subseteq Z$ be a 
smooth, projective explicit general 
arrangement variety of Picard number 
and complexity two.
Assume $n_0 \ge \ldots \ge n_r$ and 
that $P$ is irredundant.
Then we have $\Cl(X) = \ZZ^2$ 
and one of the following statements holds.
\begin{enumerate}
\item[(I)] 
We have $r=3$ and the tuple 
$(n_0, n_1, n_2, n_3)$ together with
the number $m$ fits into one of the cases
below, where $n_0 \ge n_1 \ge 3$:
\vspace{-3mm}
\begin{center}
\setlength{\columnsep}{0cm}
\begin{multicols}{2}
\begin{enumerate}
\item $m \geq 0$ and $(n_0, n_1,2,2)$,
\item $m \geq 0$ and $(n_0, 2,2,2)$,
\item $m \geq 0$ and $(n_0,2,2,1)$,
\item $m=0$ and $(3,2,1,1)$,
\item $m=0$ and $(3,1,1,1)$,
\item $m \geq 0$ and $(2,2,2,2)$,
\item $m \geq 0 $ and $(2,2,2,1)$,
\item $m \geq 0$ and $(2,2,1,1)$,
\item $m >0$ and $(2,1,1,1)$.
\end{enumerate}
\end{multicols}
\end{center}
\vspace{-2mm}
\item[(II)] 
We have $r=4$ and $m=0$ and the tuple $(n_0, n_1, n_2, n_3, n_4)$
is one of 
$$
(2,2,2,2,2),\
(2,2,2,2,1),\
(2,2,2,1,1),\
(2,2,1,1,1).
$$
\end{enumerate}
\end{proposition}

We will obtain this proposition as 
a consequence of the more general 
statements~\ref{prop:4-1-m-pos},
\ref{prop:mNull} and~\ref{cor:clfree},
presented and proven below
after the necessary preparatory work.

The following Lemma identifies small 
$\bar X$-faces and relates smoothness 
of the points of $\bar X(\gamma_0)$ 
to data of the defining matrix $P$.
It applies to arbitrary explicit 
general arrangement varieties 
and generalizes the corresponding 
statement~\cite{FaHaNi}*{Lemma~3.9}
on the case of complexity one.

\begin{lemma}
\label{lem:FF}
Situation as in Construction~\ref{constr:R(A,P)}.
Consider the orthant $\gamma = \QQ^{n+m}_{\ge 0}$, 
its extremal rays 
$\gamma_{ij} := \cone(e_{ij})$
and 
$\gamma_{k} := \cone(e_k)$
and the two-dimensional faces
$$
\gamma_{k_1,k_2} \ := \gamma_{k_1} + \gamma_{k_2},
\quad
\gamma_{ij,k} := \gamma_{ij} + \gamma_{k},
\quad
\gamma_{i_1j_1,i_2j_2} := \gamma_{i_1j_1} + \gamma_{i_2j_2}.
$$

\begin{enumerate}
\item 
All $\gamma_k$, resp.~$\gamma_{k_1,k_2}$, 
are $\bar{X}$-faces and 
each $\bar{X}(\gamma_{k})$, resp.~$\bar{X}(\gamma_{k_1,k_2})$,
consists of singular points of $\bar{X}$.
\item 
A given $\gamma_{ij}$, resp.~$\gamma_{ij,k}$,
is an $\bar{X}$-face if and only if 
$n_i \ge 2$ holds. 
In that case, $\bar{X}(\gamma_{ij})$, resp.~$\bar{X}(\gamma_{ij,k})$,
consists of smooth points of $\bar{X}$ 
if and only if $r=c+1$, $n_i =2$ and $l_{i,3-j} = 1$ hold.
\item 
A given $\gamma_{ij_1,ij_2}$ with $j_1 \ne j_2$ 
is an $\bar{X}$-face if and only if 
$n_i \ge 3$ holds.
In that case, $\bar{X}(\gamma_{ij_1,ij_2})$ consists of 
smooth points of $\bar{X}$ if and only if
$r=c+1$, $n_i=3$ and $l_{ij} = 1$ for the $j \ne j_1,j_2$ hold.
\item 
A given $\gamma_{i_1j_1,i_2j_2}$ with $i_1\neq i_2$ is 
an $\bar{X}$-face if and only if we have either
$n_{i_1}, n_{i_2}\ge 2$ or $n_{i_1}=n_{i_2}=1$ and $r=c+1$.
In the former case $\bar{X}(\gamma_{i_1j_1,i_2j_2})$ consists 
of smooth points of $\bar{X}$ if and only if 
one of the following holds:
\begin{itemize}
\item
$r=c+1$, $n_{i_t}=2$ and $l_{i_t,3-j_t} = 1$ for a $t\in\{1,2\}$,
\item
$r=c+2$, $n_{i_1}=n_{i_2}=2$, $l_{i_1,3-j_1}=l_{i_2,3-j_2} = 1$.
\end{itemize}
\end{enumerate}
\end{lemma}

\begin{proof}
Recall that for any face $\gamma_0 \preccurlyeq \gamma$,
we set 
$\bar X (\gamma_0) = \bar X \cap \TT^{n+m} \cdot z$,
where the coordinates $z_{ij},z_k$ of the 
point $z = z_{\gamma_0^*} \in \KK^{n+m}$ satisfy
$z_{ij} = 1$ if $e_{ij} \in \gamma_0$ 
and  
$z_k = 1$ if $e_k \in \gamma_0$ 
and equal zero otherwise.
Now consider the $\gamma_0 \preccurlyeq \gamma$ 
from Assertions~(i) to~(iii).
In these cases, we infer from 
Lemma~\ref{lem:pointsofbarX} that 
$\gamma_0$ is an $\bar X$-face 
if and only if 
$z = z_{\gamma_0^*}$ is of big type.
The latter is trivially satisfied
in~(i), means $n_i \ge 2$ in~(ii)
and $n_i \ge 3$ in~(iii).
Moreover, Lemma~\ref{lem:singloc}
yields the desired characterizations 
of smoothness.
If $\gamma_0$ is as in Assertion~(iv)
with $n_{i_1},n_{i_2} \ge 2$, then  
we use Lemmas~\ref{lem:pointsofbarX} 
and~\ref{lem:singloc} as in the 
preceding cases. 
If $\gamma_0$ is as in Assertion~(iv)
one of $n_{i_1},n_{i_2}$ equals one,
then $z = z_{\gamma_0^*}$, then
Lemma~\ref{lem:pointsofbarX} 
that tells us that 
$\gamma_0$ is an $\bar X$-face 
if and only if $z = z_{\gamma_0^*}$ 
is of leaf type, which amounts to
$n_{i_1}=n_{i_2}=1$ and $r=c+1$.
\end{proof}

Observe that the above statements~(iii), (iv)
and~(v) depend on the complexity~$c$.
To proceed, we take a look at the ample 
cone and the position of generator 
degrees of the Cox ring.
Propositions~\ref{prop:Qfactchar} 
and~\ref{prop:PicandCones} 
lead to the following picture.

\begin{remark}
\label{rem:projFF}
Let $X \subseteq Z$ be a projective 
explicit general arrangement variety
with divisor class group $\Cl(X)$ 
of rank two.
Then the effective cone of $X$ is of 
dimension two and decomposes as
$$
\Eff(X) 
\ = \ 
\tp \cup \tx \cup \tm,
$$
where $\tx \subseteq \Eff(X)$ is the 
ample cone, $\tp$, $\tm$ are closed cones 
not intersecting $\tx$ and 
$\tp \cap \tm$ consists of the origin.
Due to $\tx \subseteq \Mov(X)$,
each of the cones $\tp$ and~$\tm$ 
contains at least two of the weights
$$
w_{ij} \ = \ \deg(T_{ij}) \ = \ Q(e_{ij}),
\qquad\qquad
w_k \ = \ \deg(S_k) \ = \ Q(e_k).
$$
Moreover, for every $\bar{X}$-face
$\{0\} \ne \gamma_0 \preccurlyeq \gamma$
precisely one of the following 
inclusions holds:
$$
Q(\gamma_0) \ \subseteq \ \tp,
\qquad
\tx \ \subseteq \ Q(\gamma_0)^\circ,
\qquad
Q(\gamma_0) \ \subseteq \ \tm.
$$
The $X$-faces are 
exactly the $\bar{X}$-faces
$\gamma_0 \preccurlyeq \gamma$
with $\tx  \subseteq  Q(\gamma_0)^\circ$.
Note that the ample cone $\tx$ is 
of dimension two 
if and only if $X$ is $\QQ$-factorial.
\end{remark}

The following Lemma is a direct generalization 
of~\cite{FaHaNi}*{Lemma~3.11} and provides 
first constraints for the possible positions 
of the generator degrees with respect to the 
ample cone.

\begin{lemma}
\label{lem:tau}
Let $X \subseteq Z$ be a projective 
explicit general arrangement variety
with divisor class group of rank two. 
\begin{enumerate}
\item 
Suppose that $X$ is $\QQ$-factorial.
Then $w_{k} \notin \tx$ holds for all $1\le k \le m$
and for all $0\le i \le r$ with $n_i\ge2$ 
we have $w_{ij} \notin \tx$, where 
$1 \leq j \leq n_i$.
\item 
Suppose that $X$ is quasismooth, $m > 0$ 
holds and there is
$0 \le i_1 \le r$ with $n_{i_1} \ge 3$. 
Then the $w_{ij}, w_k$ with $n_i \ge 3$, 
$j=1, \ldots , n_i$ and $k=1, \ldots , m$
lie either all in $\tp$ or all in $\tm$.
\item 
Suppose that $X$ is quasismooth and there 
is $0 \le i_1 \le r$ with $n_{i_1} \ge 4$. 
Then the $w_{ij}$ with $n_i \geq 4$ and 
$j=1, \ldots , n_i$
lie either all in $\tp$ or all in $\tm$.
\item
Suppose that $X$ is quasismooth and there 
exist $0 \le i_1 < i_2 \le r$ with 
$n_{i_1}, n_{i_2} \ge 3$.
Then the $w_{ij}$ with $n_i \ge 3$, 
$j=1, \ldots , n_i$ lie either all 
in $\tp$ or all in $\tm$.
\item
Suppose that $X$ is quasismooth.
Then $w_1, \ldots, w_m$ lie
either all in $\tp$ or all in $\tm$.
\end{enumerate} 
\end{lemma}

\begin{proof}
We prove~(i).
By Lemma~\ref{lem:FF}~(i) and~(ii),
all rays $\gamma_{k}$ and all rays
$\gamma_{ij}$ with $n_i \ge 2$ are 
$\bar X$-faces. 
By $\QQ$-factoriality, 
$\tx \subseteq K_\QQ$ is of dimension 
two which excludes
$\tx \subseteq Q(\gamma_{k})^\circ$
and
$\tx \subseteq Q(\gamma_{ij})^\circ$. 
Taking $\gamma_0 = \gamma_{k}$ and 
$\gamma_0 = \gamma_{ij}$ in
Remark~\ref{rem:projFF} 
gives the claim.
We show~(ii). 
By Lemma~\ref{lem:FF}~(i) and~(ii),
all $\gamma_{k}$ 
and all $\gamma_{ij}, \gamma_{ij,k}$
with $n_i \ge 3$ are $\bar X$-faces 
and the corresponding pieces 
in $\bar {X}$ consist of singular points 
of $\bar X$. 
By quasismoothness, none of 
these $\bar X$-faces is an $X$-face.
For $\gamma_0 = \gamma_{i_11}$,
Remark~\ref{rem:projFF} yields
$w_{i_11} \in \tp$ or $w_{i_11} \in \tm$,
say $w_{i_11} \in \tp$.
Applying Remark~\ref{rem:projFF} to 
$\gamma_0 = \gamma_{ij,k}$ gives
$w_k,w_{ij} \in \tp$ for 
$k = 1,\ldots, m$ and
all $i$ with $n_i \ge 3$ and $j=1,\ldots,n_i$.
Assertions~(iii), (iv) and~(v) are seen
by similar arguments.
\end{proof}

\begin{proposition}
\label{prop:4-1-m-pos}
Let $X \subseteq Z$ be a non-toric projective 
quasismooth explicit general arrangement 
variety with divisor class group of rank two,
where $P$ is irredundant and 
$n_0 \ge  \ldots \ge n_r$.
Assume that $m > 0$ holds and 
$\Sigma$ admits an elementary 
big cone.
\begin{enumerate}
\item 
We have $r = c+1$ and 
are in one of the following situations:
\begin{enumerate}
\item
We have $n_0 = 2$ and there
exist indices $i$ and $j$ such that 
$n_i = 2$ holds and $\gamma_{ij,k}$ is 
an $X$-face for all $k$.
\item
We have $n_0 \geq 3$ and there exist indices 
$i_1 \ne i_2$ and $j_1, j_2$ such that
$n_{i_1} = n_{i_2} = 2$ holds and 
$\gamma_{i_1j_1,k}, \gamma_{i_2j_2,k}$ are $X$-faces 
for all~$k$.
\end{enumerate}
\item 
Assume $c=2$. Then we have $r=3$ and the 
constellation of the $n_i$ is
$(n_0,n_1,2,2)$,
$(n_0,2,2,2)$,
$(n_0,2,2,1)$
$(2,2,2,2)$,
$(2,2,2,1)$,
$(2,2,1,1)$
or
$(2,1,1,1)$,
where $n_0 \ge n_1 \ge 3$.
\end{enumerate}
\end{proposition}

\begin{proof}
Due to Lemma~\ref{lem:tau}~(v), 
we may assume $w_1, \dots, w_m \in \tau^+$.
As $X$ is non-toric we have at least one 
relation~$g_1$. 
Thus, $r \geq c+1$ holds and
Proposition \ref{prop:niliji}~(ii)
yields $n_0 \ge 2$.
Lemma~\ref{lem:tau}~(i) says that none 
of the $w_{ij}$ with $n_i \ge 2$ lies in 
$\tau_X$.
Moreover, at least one of the $w_{ij}$ 
with $n_i \ge 2$ lies in $\tau^-$;
otherwise, since all relations~$g_i$ 
share the same degree, we had
$w_{i1} \in \tau^+$ for all $i$ with 
$n_i = 1$, meaning that $\tau^-$ 
contains no weights at all; a 
contradiction. 
In particular, if $n_0 =2$ holds,
then there exists a $w_{ij} \in \tau^-$ 
with $n_{i} = 2$ and all $\gamma_{ij,k}$ 
are $X$-faces.
Assume $n_0 \geq 3$.
Then Lemma~\ref{lem:tau}~(ii) yields 
$w_{ij} \in \tau^+$ whenever $n_i \geq 3$.
Moreover, because all relations~$g_i$ 
have the same degree,
$w_{ij} \in \tau^+$ holds for all $i$ with 
$n_i = 1$. 
Since~$\tau^-$ contains at least two weights,
we find $i_1, i_2$ and $j_1, j_2$ 
with $n_{i_1} = n_{i_2} = 2$ and 
$w_{i_1j_1}, w_{i_2j_2} \in \tau^-$.
Note that all $\gamma_{i_1j_1,k}, \gamma_{i_2j_2,k}$ 
are $X$-faces.
Now, Lemma~\ref{lem:FF}~(ii) yields
$r = c+1$.
Thus, Assertion~(i) is proven. Assertion~(ii) 
is a direct consequence.
\end{proof}

\begin{proposition}
\label{prop:mNull}
Let $X \subseteq Z$ be a non-toric projective 
quasismooth explicit general arrangement 
variety with divisor class group of rank two,
where $P$ is irredundant and 
$n_0 \ge  \ldots \ge n_r$.
Assume that $m = 0$ holds and $\Sigma$ 
admits an elementary big cone.
\begin{enumerate}
\item 
We are in one of the following situations:
\begin{enumerate}
\item 
We have $r = c+1$, $n_0 = 3 > n_1$ and there exists 
an index~$j$ such that $\gamma_{01,0j}$ is an $X$-face.
\item 
We have $r = c+1$ and there exist indices $0 \leq i_1 < i_2$
with ${n_{i_1} = n_{i_2} = 2}$
and indices $j_0, j_2$
such that $\gamma_{0j_0,i_2j_2}$ is an $X$-face.
\item 
We have $r = c+2$ and $n_0 = n_1 = 2$
and there exist indices $0 < i_1$ and $j_0, j_1$ 
such that $\gamma_{0j_0,i_1j_1}$ is an $X$-face.
\end{enumerate}
\item
Assume $c= 2$. Then the constellation
of the $n_i$ is one of the following,
where $n_0 \geq n_1 \ge 3$ holds:
\begin{align*}
r = 3: \quad 
&(n_0, n_1, 2, 2),
(n_0, 2, 2, 2),
(n_0, 2, 2, 1),
(3, 2, 1, 1),
(3, 1, 1, 1),
\\
&(2, 2, 2, 2),
(2, 2, 2, 1),
(2, 2, 1, 1).
\\
r  = 4: \quad 
&(2,2,2,2,2),
(2,2,2,2,1),
(2,2,2,1,1),
(2,2,1,1,1). 
\end{align*}
\end{enumerate}
\end{proposition}

\begin{proof}
Only for the first assertion, there is 
something to show.
As $X$ is non-toric 
we have at least one relation
$g_1$ and conclude
$r \geq c+1$. Moreover,
Proposition \ref{prop:niliji}~(ii)
yields $n_0 \geq 2$.
Finally, Lemma~\ref{lem:tau}~(i) 
shows that none of the $w_{ij}$ with 
$n_i \ge 2$ lies in $\tau_X$.
We distinguish the following cases.

First, let $n_0 \ge 4$ or $n_0 =  n_1 = 3$.
By Assertions~(iii) and~(iv) of Lemma~\ref{lem:tau},
we may assume ${w_{ij} \in \tau^+}$
for all~$i$ with $n_i  \ge 3$. 
Then $w_{ij} \in \tau^+$ holds as well 
for all~$i$ with $n_i = 1$.
Since $\tau^-$ contains at least two weights, 
there are $i_1 < i_2$ and $j_1, j_2$ with 
$n_{i_1} = n_{i_2} = 2$ and
$w_{i_1j_1}, w_{i_2j_2} \in \tau^-$.
Observe that $\gamma_{01,i_2j_2}$ is an $X$-face.
Moreover, Lemma~\ref{lem:FF}~(iv) shows 
$r = c+1$. 
We arrive at Case~(b) of~(i).

Next, let $n_0 = 3 > n_1$.
If all weights $w_{0j}$ lie 
either in $\tau^+$ or in $\tau^-$,
then we can argue as above
and end up in Case~(b) of~(i).
Otherwise, $w_{01}$ and some 
$w_{0j}$ for $j=2,3$ lie on 
different sides of $\tau_X$. 
Then $\gamma_{01,0j}$ is 
an $X$-face.
Lemma~\ref{lem:FF}~(iii) 
yields $r= c+1$ and
we are in Case~(a) of~(i).

Finally, let $n_0 = 2$.
The common degree of 
$g_1, \ldots, g_{r-c}$ and 
hence all $w_{ij}$ with $n_i = 1$
lie in precisely one of the cones 
$\tau^+$, $\tau^-$ or $\tau_X$,
where we may assume that this 
is not $\tau^-$.
Then no pair $w_{i1}, w_{i2}$ lies 
in~$\tau^-$.
As there must be at least two weights 
in $\tau^-$, we conclude $n_1=2$
and find the desired $\gamma_{0j_0,1j_1}$.
Lemma \ref{lem:FF}~(iv) yields
$r \leq c+2$.
Thus, we are in one of the Cases~(b) 
or~(c) of~(i).
\end{proof}

\begin{corollary}
\label{cor:clfree}
Let $X$ be a smooth projective general 
arrangement 
variety of Picard number two.
Then we have $\Cl(X) = \Pic(X) = \ZZ^2$.
\end{corollary}

\begin{proof}
We may assume that $X \subseteq Z$ 
is explicit.
Corollary~\ref{cor:ebc} tells us that 
$\Sigma$ admits an elementary big cone.
Thus Propositions~\ref{prop:4-1-m-pos} 
and~\ref{prop:mNull} provide
a two-dimensional $X$-face
$\gamma_0 \preccurlyeq \gamma$.
According to Proposition~\ref{prop:smoothchar},
the two weights stemming from 
$\gamma_0$ generate $K$ as a group.
This implies $\Cl(X) \cong K \cong \ZZ^2$.
\end{proof}

\begin{proof}[Proof of Proposition~\ref{prop:const}]
According to Corollary~\ref{cor:ebc}, the fan 
$\Sigma$ admits an elementary big cone.
Thus, Propositions~\ref{prop:4-1-m-pos}~(ii)
and~\ref{prop:mNull}~(ii) apply. 
Together with Corollary~\ref{cor:clfree} they
provide the desired statements.
\end{proof}

In order to show that Theorem~\ref{theo:proj}
lists all smooth projective general 
arrangement varieties of true complexity two,
we have to go through the cases of 
Proposition~\ref{prop:const}.
After some further preparation,
we will treat exemplarily 
Cases~\ref{prop:const}~(I)(a) and~(b).
The detailed discussion of the 
remaining cases is given in~\cite[Sec.~6.1]{Wr}.

\begin{remark}
\label{rem:Q}
Let $X \subseteq Z$ be a smooth, 
projective explicit general arrangement 
variety of Picard number two.
Corollary~\ref{cor:clfree} ensures
$\Cl(X) =  \ZZ^2$ and we will write
$$
\deg(T_{ij}) 
\ = \ 
Q(e_{ij}) 
\ = \ 
w_{ij} 
\ = \ 
(x_{ij},y_{ij}) 
\ \in \ 
\ZZ^2,
$$
$$
\deg(T_{k}) 
\ = \ 
Q(e_{k})
\ = \ 
w_{k} 
\ = \ 
(x_{k},y_{k}) 
\ \in \ 
\ZZ^2
$$
for the weights.
Moreover, the (common) degree of the 
relations $g_1, \ldots, g_{r-c}$ 
will be denoted as 
$\deg(g_i) = \mu  = (\mu_1, \mu_2) \in \ZZ^2$.
Recall that for each $i = 0,\ldots,r$ we have
$$
\mu_1 \ = \ \sum_{j=1}^{n_i} l_{ij}x_{ij},
\qquad\qquad 
\mu_2 \ = \ \sum_{j=1}^{n_i} l_{ij}y_{ij}.
$$
Consider the decomposition of the effective cone 
$\mathrm{Eff}(X) = \tau^- \cup \tau_X \cup \tau^+$
from Remark~\ref{rem:projFF}.
Choosing names suitably, we can fix the following 
orientation:

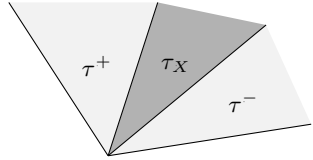
\begin{center}
\begin{tikzpicture}[scale=0.6]
    \path[fill=gray!60!] (0,0)--(3.5,2.9)--(1.1,3.4)--(0,0);
    \path[fill, color=black] (1.5,2) circle (0.0ex)  node[]{\small{$\tx$}};
    \path[fill=gray!10!] (0,0)--(1.1,3.4)--(-2.2,3.4)--(0,0);
    \path[fill, color=black] (-0.25,2) circle (0.0ex)  node[]{\small{$\tp$}};
    \draw (0,0)--(1.1,3.4);
    \draw (0,0) --(-2.2,3.4);
     \path[fill=gray!10!] (0,0)--(3.5,2.9)--(4.5,0.7)--(0,0);
    \path[fill, color=black] (3,1.2) circle (0.0ex)  node[]{\small{$\tm$}};
    \draw (0,0)  -- (3.5,2.9);
    \draw (0,0)  -- (4.5,0.7);
  \end{tikzpicture}
\end{center}

\noindent
If $(w,w')$ is a positively oriented
pair in $\QQ^2$,
for instance $w \in \tau^-$ and $w' \in \tau^+$,
then $\det(w,w')$ is positive.
If, furthermore, $w,w'$ are the weights stemming from 
a two-dimensional $X$-face 
$\gamma_0 \preccurlyeq \gamma$, then we have
$\det(w,w')=1$ by Proposition~\ref{prop:smoothchar}.
In that case, we can achieve 
$$ 
w \ = \ (1,0), 
\qquad
\qquad
w' \ = \ (0,1)
$$
by a suitable unimodular coordinate change on 
$\ZZ^2$. 
Then $w'' = (x'',1)$ holds whenever $w,w''$ 
stems from a two-dimensional $X$-face
and, similarly, 
$w'' = (1,y'')$ holds whenever $w'',w'$ 
stems from  a two-dimensional $X$-face.
\end{remark}

\begin{lemma}
\label{lem:CaseABC} 
In the situation of Proposition~\ref{prop:const}, 
consider the case $r = 3$, $m \geq 0$ and 
$n_0 \ge 3 > n_1 = n_2 = 2 \ge n_3$. 
Then the following constellation of weights 
cannot occur:
$$ 
w_{01}, \ldots, w_{0n_0},w_{12},w_{22} 
\ \in \
\tau^+,
\qquad \qquad
w_{11}, w_{21} 
\ \in \ 
\tau^-.
$$
\end{lemma}
\begin{proof}
We may assume 
$w_{02}, \dots, w_{0n_0}, w_{21} \in \cone(w_{01}, w_{11})$.
Applying Remark~\ref{rem:Q} first to 
the $X$-face $\gamma_{01,11}$ and then to all
$\gamma_{01,21}, \gamma_{22,11},\gamma_{0j,11}, \gamma_{i,11}\in\rlv(X)$,
where ${j = 1, \ldots, n_0}$ and $i =1, \ldots, m$, 
turns the degree matrix $Q$ into the shape
$$
Q
=
\left[ 
\begin{array}{cccc|cc|cc|ccc||ccc}
0 & x_{02} & \dots & x_{0n_0}  & 1 & x_{12} & 1 & x_{22} & x_{31} & \dots & x_{3n_1} & x_1 & \dots & x_m
\\
1 & 1 & \dots & 1  & 0 & y_{12} & y_{21} &1 & y_{31} & \dots & y_{3n_1} & 1 & \dots  & 1 
\end{array}
\right],
$$
where $x_{0j} \ge 0$ and $y_{21} \ge 0$.
Moreover, $\gamma_{01,11}, \gamma_{01,21} \in \rlv(X)$ 
implies $l_{12} = l_{22} =1$ due to Lemma~\ref{lem:FF}~(iv).
Using $\gamma_{21,12} \in \rlv(X)$ 
we infer $y_{12} = 1+y_{21}x_{12}$ 
from $\det(w_{21},w_{12}) = 1$ 
and, by the shape of $Q$, obtain
$$
3 
\ \le \
l_{01} + \dots+ l_{0n_0} 
\ = \ 
\mu_2 
\ = \ 
y_{12} 
\ = \ 
1+y_{21}x_{12}.
$$
We conclude $x_{12} > 0$. 
Using $\gamma_{0j,21} \in \rlv(X)$
gives $\det(w_{21},w_{0j}) = 1$
and thus $x_{0j}y_{21} = 0$.
As the effective cone of $X$ 
is pointed, $w_{21} \in  \tau^-$ 
implies $y_{21} > 0$.
We arrive at $x_{0j} = 0$ and thus 
$\mu_1 = 0 = l_{11} + x_{12}$. 
This contradicts to the fact that 
$l_{11}$ and $x_{12}$ are strictly 
positive.
\end{proof}

We are ready to establish the list of 
Theorem~\ref{theo:proj}. 
As announced, we restrict ourselves here 
to the discussion of two cases of 
Proposition~\ref{prop:const} and refer 
to~\cite[Sec.~6.1]{Wr} for the remaining 
ones.

\begin{namedthm*}{Case~\ref{prop:const}~(I)(a)}
We have $r= 3$, $m \geq 0$ and $n_0 \ge n_1 \ge 3 > n_2 = n_3 = 2$.
This setting allows no examples satisfying the assumptions 
of Theorem~\ref{theo:proj}.
\end{namedthm*}

\begin{proof}
By Assertions~(iv) and~(ii) of Lemma~\ref{lem:tau}, 
we may assume that the weights 
$w_{01}, \dots, w_{0n_0}$, $w_{11}, \dots,  w_{1n_1}$ 
and $w_1,\dots, w_m$ all lie in $\tau^+$.
At least two other weights lie in~$\tau^-$. 
Renumbering suitably, we arrive at 
$w_{21}, w_{31} \in \tau^-$ and $w_{22}, w_{32} \in \tau^+$
because of $\mu \in \tau^+$.
Thus, Lemma~\ref{lem:CaseABC} gives the assertion. 
\end{proof}

\begin{namedthm*}{Case~\ref{prop:const}~(I)(b)}
We have $r= 3$, $m \geq 0$ and $n_0 \ge 3 > n_1 = n_2 = n_3 = 2$.
This gives the varieties Nos.~1 and~2 
of Theorem~\ref{theo:proj}. 
\end{namedthm*}

\begin{proof}
We claim that each of $\tau^+$ and $\tau^-$
contains weights from $w_{01}, \ldots, w_{0n_0}$.
Otherwise, due to Lemma~\ref{lem:tau}~(i),
we may assume that all $w_{0j}$ lie in $\tau^+$.
If $m > 0$ holds, Lemma~\ref{lem:tau}~(ii) yields
$w_1, \ldots, w_m \in \tau^+$. 
As $\tau^-$ contains at least two weights,
we can achieve 
$w_{11}, w_{21} \in \tau^-$ and $w_{12}, w_{22} \in \tau^+$
by suitable renumbering; 
note that $w_{i1}, w_{i2} \in \tau^-$ is not possible 
for $i = 1,2,3$ because of $\mu \in \tau^+$.
Lemma~\ref{lem:CaseABC} then verifies the claim.

By the claim, we may assume $w_{01}, w_{02} \in \tau ^+$ 
and $w_{03} \in \tau^-$.
Lemma~\ref{lem:tau}~(ii) shows $m=0$ 
and Lemma~\ref{lem:tau}~(iii) gives $n_0 = 3$. 
There must be at least one more weight in~$\tau^-$, 
say $w_{11}$.
Applying Lemma~\ref{lem:FF}~(iii) to 
$\gamma_{0j,03} \in \rlv(X)$ 
we obtain $l_{01} = l_{02} = 1$. 
Moreover, looking at 
suitable $X$-faces $\gamma_{0j,i_2j_2}$
and using Lemma~\ref{lem:FF}~(iv), 
we obtain
$$
l_{11} = l_{12} = l_{21} = l_{22} = l_{31} = l_{32} = 1.
$$
We may assume $w_{02} \in \cone(w_{01},w_{03})$.
Then, applying Remark~\ref{rem:Q} to $\gamma_{01,03}\in\rlv(X)$ 
and afterwards to $\gamma_{01,11}, \gamma_{02,03}\in\rlv(X)$ 
we achieve that our degree matrix looks as follows:
$$
Q
\ = \
\left[ 
\begin{array}{ccc|cc|cc|cc}
0 & x_{02} & 1 &1 &x_{12} & x_{21} & x_{22} & x_{31} & x_{32} 
\\
1 & 1 &0 &y_{11} & y_{12} & y_{21} & y_{22} & y_{31} &y_{32}
\end{array}
\right].
$$
Note that we have $x_{02} \ge 0$ because of 
$w_{02} \in  \cone(w_{01},w_{03})$.
We distinguish the following three cases
according to the possible positions of 
the weights $w_{21}$ and~$w_{22}$.

\medskip
\noindent
\emph{We have $w_{21}, w_{22} \in \tau^-$}. 
Then $\mu \in \tau^-$ holds and we may assume 
$w_{31} \in \tau^-$.
Moreover, we have 
$\gamma_{01,21}, \gamma_{01,22}, \gamma_{01,31} \in \rlv(X)$
and conclude 
$$
x_{21} = x_{22} = x_{31} = 1,
\qquad
\mu = (2,2),
\qquad
x_{12} = x_{22} = x_{32} = 1.
$$
The determinants corresponding to
$\gamma_{02,21},\gamma_{02,22} \in \rlv(X)$
both equal one, which implies
$y_{21}x_{02} = 0$ and $y_{22}x_{02} = 0$.
Because of $y_{21} + y_{22} = \mu_2 = 2$,
we obtain 
$$
x_{02} \ = \ 0,
\qquad\qquad
l_{03} =  \ \mu_1  \ =  \ 2.
$$
The considerations performed so far 
determine the defining relation $g_1$ 
and, up to integers $a_1,a_2,a_3$, 
also the degree matrix $Q$ as
$$
g_1 \ = \ T_{01}T_{02}T_{03}^2 +T_{11}T_{12}+T_{21}T_{22}+T_{31}T_{32},
$$
$$
Q \ = \ 
\left[ 
\begin{array}{ccc|cc|cc|cc}
0 & 0 & 1 & 1 & 1 & 1 & 1 & 1 & 1
\\
1 & 1 & 0 & a_1 & 2-a_1 & a_2 & 2-a_2 & a_3 & 2-a_3
\end{array}
\right].
$$
We claim that all $w_{ij}$, where $i = 1,2,3$,
lie in $\tau^-$.
That means that we have to show 
$w_{12}, w_{32} \in \tau^-$.
Otherwise, if $w_{12} \in \tau^+$ holds, 
then $\gamma_{03,12} \in \rlv(X)$ 
leads to 
$$ 
1 
\ = \ 
\det(w_{03},w_{12})
\ = \ 
a_1.
$$
This implies $w_{11} = w_{21} \in \tau^- \cap \tau^+$,
which is impossible. 
Analogously, one excludes  $w_{32} \in \tau^+$.
Thus, we may assume  $a_1 \leq a_2 \leq a_3$ and
$a_i \geq 2-a_i$. 
Consequently, $a_i \geq 1$ holds and we obtain 
$$ 
\SAmple(X) 
\ = \ 
\overline{\tau_X} 
\ = \ 
\cone( (1,a_3), (0,1)).
$$

\medskip
\noindent
\emph{We have $w_{21}, w_{22} \in \tau^+$}. 
Thus, we have $\mu \in \tau^+$ and thus
$w_{12} \in \tau^+$. 
That means that 
$\gamma_{03,12}, \gamma_{03,21}, \gamma_{21,22}$
are $X$-faces and we can conclude 
$$
y_{12} = y_{21} = y_{22} = 1,
\qquad
\mu_2 = 2,
\qquad
y_{11} = 1. 
$$
Looking at the determinants
associated with 
$\gamma_{02,11},\gamma_{11,21}, \gamma_{11,22} \in \rlv(X)$
we see
$x_{02} = x_{21} =x_{22} = 0$. 
This gives
$l_{03} = \mu_1 = x_{21} +x_{22} = 0$;
a contradiction.

\medskip
\noindent
\emph{We have $w_{21} \in \tau^-$ and $w_{22} \in \tau^+$}.
Then we may assume $w_{31} \in \tau^-$ and $w_{32} \in \tau^+$,
as otherwise, up to renumbering, we are in one of the 
preceding cases.
Remark~\ref{rem:Q} applied to
$\gamma_{01,21}, \gamma_{01,31}, \gamma_{03,22},
\gamma_{03,32} \in \rlv(X)$ 
and $\mu_2 = 2$ yield
$$
x_{21} = x_{31} = y_{22} = y_{32} = 1,
\qquad \qquad
y_{21} = y_{31}  = 1.
$$
We claim $y_{11} \ne 0$. 
Otherwise, $y_{12} =  \mu_2 = 2$.
This implies
$\det(w_{03},w_{12}) = 2$,
hence
$\gamma_{03,12}\notin \rlv(X)$ 
and thus $w_{12} \in \tau^-$.
Then $\gamma_{01,12} \in \rlv(X)$
leads to 
$x_{12} =1$ and $\mu_1 = 2$. 
Thus $w_{22} = (1,1) = w_{21} \in \tau^-$,
contradicting the setup.
Now, $y_{11} \ne 0$ yields
$$ 
x_{02} = x_{22} = x_{32} = 0,
\qquad
\mu = (1,2),
\qquad
l_{03} = 1,
\qquad
x_{12} = 0
$$
due to 
$\gamma_{11,02},\gamma_{11,22},\gamma_{11,32} \in \rlv(X)$
and homogeneity of the relation $g_1$.
We conclude $w_{12} = (0,y_{12}) \in \tau^+$
and 
$\gamma_{03,12} \in \rlv(X)$ shows $y_{12} =1$.
Finally, 
$y_{11} = \mu_2- y_{12} = 1$ holds. 
For the relation, the degree matrix 
and the ample cone this means
$$
g_1 
\ = \ 
T_{01}T_{02}T_{03} + T_{11}T_{12}+ T_{21}T_{22} +T_{31}T_{32},
$$
$$
Q
\ = \ 
\left[ 
\begin{array}{ccc|cc|cc|cc}
0 & 0 & 1 & 1 & 0 & 1 & 0 & 1 & 0
\\
1 & 1 & 0 &1 & 1 & 1 & 1 & 1 & 1
\end{array}
\right],
$$
$$
\SAmple(X) 
\ = \
\overline{\tau_X}
\ = \ 
\cone((0,1), (1,1)).
$$
\end{proof}

Finally, Theorem~\ref{theo:proj} claims that 
the listed varieties are all of true 
complexity two.
The following Proposition establishes this
statement. 
As before, we present an example case; 
the full discussion will be made available 
elsewhere.

\begin{proposition}
\label{prop:truly}
Each of the varieties listed in Theorem~\ref{theo:proj} 
is of true complexity two, i.e., does not admit torus 
actions of lower complexity.
\end{proposition}

\begin{proof}
First observe that each of the varieties 
listed in Theorem~\ref{theo:proj} has a 
singular total coordinate space and hence
is not toric.
Thus, we have to show that none of 
the varieties from Theorem~\ref{theo:proj} 
is isomorphic to a smooth non-toric 
variety of Picard number two with torus action 
of complexity one, which in turn are all
given in~\cite{FaHaNi}*{Thm.~1.1}.
We do this exemplarily for the 
variety~$X$ listed as No.~5 in our
Theorem~\ref{theo:proj}.
Recall that Cox ring, degree matrix 
and an ample class of $X$ are 
\begin{eqnarray*}
R
& = &
\KK[T_1,\ldots,T_8,S_1,\ldots,S_{m}]
/ 
\bangle{T_1T_2+T_3^2T_4+T_5^2T_6+T_7^2T_8},
\quad 
m \ge 0,
\\
Q 
& = &
\left[
\begin{array}{cccccccc|ccc}
0&2b+1&b&1&b&1&b&1&1&\ldots&1
\\
1&1&1&0&1&0&1&0&0&\ldots&0
\end{array}
\right],
\quad
b \ge 0,
\\
u
& = &
(2b+2,1).
\end{eqnarray*}
The total coordinate space $\Spec(R)$ of $X$
is of dimension~$m+7$ 
with singular locus of codimension~$4$.
Computing these data also for the 
varieties from~\cite{FaHaNi}*{Thm.~1.1}, we see 
that~$X$ can be isomorphic at most 
to varieties from Nos.~4, 6, 10, 11 or~12 
of~\cite{FaHaNi}*{Thm.~1.1}.
We now go through these cases.

Assume that $X$ is isomorphic to the 
variety~\cite{FaHaNi}*{Thm.~1.1, No.~4},
which we denote by $X'$.
Cox ring, degree matrix and an ample class of $X'$ 
are given by 
\begin{eqnarray*}
R'
& = &
\KK[T_1,\ldots,T_6,S_1,\ldots,S_{m'}]
/ 
\bangle{T_1T_2^{l_2}+T_3T_4^{l_4}+T_5T_6^{l_6}},
\quad 
m' \ge 0,
\\
Q' 
& = &
\left[
\begin{array}{cccccc|ccc}
0&1&a_1&1&a_2&1&c_1&\ldots&c_{m'}
\\
1&0&1&0&1&0&1&\ldots&1
\end{array}
\right],
\quad
\begin{array}{l}
0 \le a_1 \le a_2, \ 
\\
c_1 \le \ldots \le c_{m'},
\\
1 \le l_2 = a_1+l_4 = a_2+l_6b,
\end{array}
\\
u'
& = &
(\max(a_2,c_m')+1,1).
\end{eqnarray*}
The total coordinate space $\Spec(R')$ of $X'$
is of dimension~$m'+ 5$ and the codimension
of its singular locus equals~$5$ minus the 
number of $i$ with $l_i \ge 2$. 
Consequently, we obtain
$$ 
m'  = m+2, 
\quad
l_4 = l_6 = 1,
\quad
l_2 = a_1 + 1 = a_2 + 1 \ge 2,
\quad a_1 = a_2.
$$
We write $w_i$ for the $i$-th column of $Q$ 
and denote by $\mu_i$ the number of times 
it shows up as a column of $Q$.
Analogously, we define $w_i'$ and $\mu_i'$.
Then we have 
$$ 
\mu_1 \in \{1,4\},
\quad
\mu_4 = 3+m,
\qquad
\mu_2' = 3,
\quad
\mu_1' \le 1+m'.
$$
Observe that $w_1,w_4$ are the primitive 
generators of the extremal rays 
of the effective cone of $X$ 
and $w_4$ is a semiample class, 
whereas $w_1$ is not semiample.
Moreover, $w_2'$ is a semiample 
primitive generator of the effective 
cone of $X'$. 
We conclude 
$$
3+m = \mu_4 = \dim(R_{w_4}) 
= \dim(R'_{w_2'})  = \mu_2' = 3.
$$ 
Thus, $m=0$ and $m'= 2$ hold. 
Comparing the multiplicities 
$\dim(R_w)$ and $\dim(R'_{w'})$
for $w$ and $w'$ being the primitive 
generators differing from $(1,0)$
of the respective effective,
moving and semiample cones
of $X$ and $X'$, we obtain
$$
b,c_1,c_2 > 0,
\quad
\mu_1' = \mu_1 = 1
\quad 
b = a_1 = a_2 = c_1 < c_2 = 2b+1.
$$
But then the anticanonical class 
$-\mathcal{K}_{X} = (3a+3,3)$ is divisible 
by $3$, whereas $-\mathcal{K}_{X'} = (4b+3,3)$ 
is not; a contradiction.

Assume that $X$ is isomorphic to a 
variety $X'$ as in~\cite{FaHaNi}*{Thm.~1.1, No.~6}.
Here, Cox ring, the degree matrix and an 
ample class look as follows: 
\begin{eqnarray*}
R'
& = &
\KK[T_1,\ldots,T_6,S_1,\ldots,S_{m'}]
/ 
\bangle{T_1T_2+T_3T_4+T_5^{2}T_6},
\quad 
m' \ge 1,
\\
Q' 
& = &
\left[
\begin{array}{cccccc|ccc}
0&2a_3+1&a_1&a_2&a_3&1&1&\ldots &1\\
1&1&1&1&1&0&0&\ldots &0
\end{array}
\right],
\quad
\begin{array}{l}
a_1,a_2,a_3 \ge 0,
\\
a_1 < a_2,
\\
a_1+a_2 = 2a_3+1,
\end{array}
\\
u'
& = &
(2a_3+2,1).
\end{eqnarray*}
The dimension of the total coordinate space $\Spec(R')$ 
of $X'$ equals $m'+5$ and hence $m' = m +2$ must hold.
As before, let $w_i$ be the $i$-th column of $Q$
and $\mu_i$ the number of times it shows up as 
a column of $Q$.
Define $w_i'$ and $\mu_i'$ analogously.
We obtain 
$$
\mu_1 \in \{1,4\},
\qquad 
\mu_4 = \mu_6' = m+3,
\qquad 
\mu_1' \in \{1,2,3\}.
$$
For $X$ as well as for $X'$, we find 
precisely one  semiample primitive 
generator of the effective cone, 
namely $w_4$ and $w_6'$.
Consequently we obtain 
$$
1  =  \mu_1  = \mu_1',
\qquad 
a_1,a_2,a_3 > 0.
$$
Comparing the multiplicities $\dim(R_w)$ and $\dim(R'_{w'})$ for 
$w$ and $w'$ being the primitive generators differing from $(1,0)$ 
of the effective, movable and semiample cones 
of~$X$ and $X'$, we arrive at
$a_1 = a_2 = a_3 = b$, which contradicts,
for instance, $a_1<a_2$.

Assume that $X$ is isomorphic to the variety $X'$ 
as in~\cite{FaHaNi}*{Thm.~1.1, No.~10}.
In this case, the Cox ring, the degree matrix and an 
ample class of $X'$ are given as 
\begin{eqnarray*}
R'
& = &
\KK[T_1,\ldots,T_5,S_1,\ldots,S_{m'}]
/ 
\bangle{T_1T_2+T_3T_4+T_5^{2}},
\quad 
m' \ge 1,
\\
Q' 
& = &
\left[
\begin{array}{ccccc|ccc}
1&1&1&1&1&0&\ldots &0\\
0&2&1&1&1&1&\ldots &1
\end{array}
\right],
\\
u'
& = &
(2,1).
\end{eqnarray*}
Let $w_i$, $w_i'$ and $\mu_i$, $\mu_i'$ be 
as before.
Then $w_4$ and $w_1'$ are the only semiample 
primitive generators of the effective cones 
of $X$ and $X'$, respectively.
Thus, we obtain $1 = \mu_1' = \mu_4 = 3+m$; 
a contradiction.

Assume that $X$ is isomorphic to the variety $X'$ 
as in~\cite{FaHaNi}*{Thm.~1.1, No.~11}.
The Cox ring, the degree matrix and an 
ample class of $X'$ are then given by
\begin{eqnarray*}
R'
& = &
\KK[T_1,\ldots,T_5,S_1,\ldots,S_{m'}]
/ 
\bangle{T_1T_2+T_3T_4+T_5^{2}},
\quad 
m' \ge 2,
\\
Q' 
& = &
\left[
\begin{array}{ccccc|cccc}
1&1&1&1&1&0&a_2&\ldots &a_{m'}\\
0&0&0&0&0&1&1&\ldots &1
\end{array}
\right],
\quad
\begin{array}{l}
0 \le a_2 \le \ldots \le a_{m'},
\\
a_{m'} > 0,
\end{array}
\\
u'
& = &
(a_{m'}+1,1).
\end{eqnarray*}
With $w_i$, $w_i'$ and $\mu_i$, $\mu_i'$ as 
before,
we see that, again, $w_4$, $w_1'$  
are the only semiample primitive generators 
of the effective cones of $X$, $X'$,
respectively, and conclude
$$ 
5 = \mu_1' = \mu_4 = 3+ m.
$$
This implies $m=2$. Thus, $\Spec(R)$ is 
of dimension $7+m = 9$.
Consequently,  $\Spec(R')$ is 
of dimension $9 = 4 + m'$,
showing $m' = 5$.
Looking for $R$ and $R'$ at the homogeneous 
components of degrees $2w_4$ and $2w_1'$ 
respectively, we arrive at a contradiction:
$$
15 \ = \ \dim(R_{2w_4}) \ = \  \dim(R'_{2w_1'}) \ = \ 14.
$$

Finally, assume that $X$ is isomorphic to the 
variety $X'$ as in~\cite{FaHaNi}*{Thm.~1.1, No.~12}.
The Cox ring, the degree matrix and an 
ample class of the latter are given as 
\begin{eqnarray*}
R'
& = &
\KK[T_1,\ldots,T_5,S_1,\ldots,S_{m'}]
/ 
\bangle{T_1T_2+T_3T_4+T_5^{2}},
\quad 
m' \ge 2,
\\
Q' 
& = &
\left[
\begin{array}{ccccc|cccc}
0&2a_3&a_1&a_2&a_3&1&1&\ldots &1
\\
1&1&1&1&1&0&0&\ldots &0
\end{array}
\right],
\quad
\begin{array}{l}
0 \le a_1 \le a_3 \le a_2,
\\
a_1+a_2 = 2a_3,
\end{array}
\\
u'
& = &
(2a_3+1,1).
\end{eqnarray*}
Comparing the primitive generators $w$, $w'$ 
and the corresponding multiplicities 
$\dim(R_w)$, $\dim(R_w')$ 
of the effective, moving and semiample cones of 
$X$ and $X'$, we arrive at 
$$
m' = 3+m, 
\qquad
0 < b,
\qquad
0 < a_1 = a_2 = a_3.
$$
Now, comparing the determinants of the Mori chambers 
of $X$ and $X'$ leads to a contradiction: we obtain
$$ 
b = \det(w_3,w_1) = \det(w_5',w_2) = a_3.
\qquad
b+1 = \det(w_2,w_3) = \det(w_2',w_3') = a_3.
$$
\end{proof}

\section{Fano and almost Fano arrangement varieties}
\label{sec:FanoClass}

In this section, we present the Fano 
and almost Fano varieties from 
Theorem~\ref{theo:proj} and discuss 
geometric properties of the Fano varieties.
The key is the description of the 
anticanonical class of an explicit 
general arrangement variety 
$X \subseteq Z$ given~by
$$ 
-\mathcal{K}_X
\ = \
\sum_{i = 0}^r \sum_{j = 1}^{n_i} w_{ij} 
+ 
\sum_{k = 0}^r w_k
\, - \,
(r-c) \sum_{j=1}^{n_0} l_{0j}w_{0j}
\ \in \ 
K
\ = \ 
\Cl(X),
$$
where $c=c(X)$ is the complexity 
and $w_{ij} = \deg(T_{ij})$ as well 
as $w_k = \deg(T_k)$ are the 
Cox ring generator degrees,
see Proposition~\ref{prop:genarrcandiv}.
Going through the list of Theorem~\ref{theo:proj} 
and picking the cases with $-\mathcal{K}_X$
lying in the ample cone, we obtain the 
following.

\begin{theorem}
\label{theo:Fano}
Every smooth Fano general arrangement variety 
of true complexity two and Picard number two 
is isomorphic to precisely one of the following 
varieties~$X$, specified by their Cox ring 
$\mathcal{R}(X)$ and the matrix $[w_1, \dots, w_r]$ 
of generator degrees $w_i \in \Cl(X) = \ZZ^2$. 

\medskip

{\centering
{\small
\setlength{\tabcolsep}{-1pt}
\begin{longtable}{ccccc}
No.
&
\small{$\mathcal{R}(X)$}
&
\small{$[w_1,\ldots, w_r]$}
&
\small{$-\mathcal{K}_X$}
&
\small{$\dim(X)$}
\\
\toprule
1
&
$
\frac
{\KK[T_1, \ldots , T_9]}
{\langle T_{1}T_{2}T_{3}^2+T_{4}T_{5}+T_6T_7+T_8T_9 \rangle}
$
&
\tiny{
\setlength{\arraycolsep}{1pt}
$
\left[
\begin{array}{ccccccccc}
0 & 0 & 1 & 1 & 1 & 1 & 1 & 1 & 1
\\ 
1 & 1 & 0 & 1 & 1 &1 & 1 & 1 & 1
\end{array}
\right]
$
}
&
\tiny{
\setlength{\arraycolsep}{1pt}
$
\left[
\begin{array}{c}
5
\\
6
\end{array}
\right]
$
}
&
\small{$6$}
\\
\midrule
2
&
$
\frac
{\KK[T_1, \ldots , T_9]}
{\langle T_{1}T_{2}T_{3}+T_{4}T_{5}+T_6T_7 +T_8T_9\rangle}
$
&
\tiny{
\setlength{\arraycolsep}{1pt}
$
\left[
\begin{array}{ccccccccc}
0 & 0 & 1 & 1 & 0 & 1 & 0  & 1 & 0
\\ 
1 & 1 & 0 & 1 & 1 & 1 & 1 & 1 & 1
\end{array}
\right]
$
}
&
\tiny{
\setlength{\arraycolsep}{1pt}
$
\left[
\begin{array}{c}
3
\\
6
\end{array}
\right]
$
}
&
\small{$6$} 
\\
\midrule
3
&
$
\frac
{\KK[T_1, \ldots , T_8]}
{\langle T_{1}T_{2}T_{3}^2+T_{4}T_{5}+T_6T_7 +T_8^2\rangle}
$
&
\tiny{
\setlength{\arraycolsep}{1pt}
$
\left[
\begin{array}{cccccccc}
0 & 0 & 1 & 1 & 1 & 1 & 1 & 1
\\ 
1 & 1 & 0 & 1 & 1 & 1 & 1 & 1
\end{array}
\right]
$
}
&
\tiny{
\setlength{\arraycolsep}{1pt}
$
\left[
\begin{array}{c}
4
\\
5
\end{array}
\right]
$
}
&
\small{$5$} 
\\
\midrule
4A
&
$
\begin{array}{c}
\frac
{\KK[T_1, \ldots , T_8, S_1,\ldots,S_m]}
{\langle T_{1}T_{2}^3+T_{3}T_{4}+T_5T_{6}+T_7T_8 \rangle}
\\
\scriptstyle m \geq 0
\end{array}
$
&
\tiny{
\setlength{\arraycolsep}{1pt}
$
\left[
\begin{array}{cccccccc|cc}
0 & 1 & 2 & 1 & 2 & 1 & 2 & 1 & 2 \ldots & 2
\\
1 & 0 & 1 & 0 & 1 & 0 & 1 & 0 & 1  \ldots & 1
\end{array}
\right]
$
}
&
\tiny{
\setlength{\arraycolsep}{1pt}
$
\left[
\begin{array}{c}
7+2m
\\
3+m 
\end{array}
\right]
$
}
&
\small{$m+5$}
\\
\midrule
4B
&
$
\begin{array}{c}
\frac
{\KK[T_1, \ldots , T_8, S_1,\ldots,S_m]}
{\langle T_{1}T_{2}^3+T_{3}T_{4}^2+T_5T_{6}^2+T_7T_8^2 \rangle}
\\
\scriptstyle m \geq 0
\end{array}
$
&
\tiny{
\setlength{\arraycolsep}{1pt}
$
\left[
\begin{array}{cccccccc|cc}
0 & 1 & 1 & 1 & 1 & 1 & 1 & 1 & 1 \ldots & 1
\\
1 & 0 & 1 & 0 & 1 & 0 & 1 & 0 & 1  \ldots & 1
\end{array}
\right]
$
}
&
\tiny{
\setlength{\arraycolsep}{1pt}
$
\left[
\begin{array}{c}
4+m 
\\
3+m
\end{array}
\right]
$
}
&
\small{$m+5$}
\\
\midrule
4C
&
$
\begin{array}{c}
\frac
{\KK[T_1, \ldots , T_8, S_1,\ldots,S_m]}
{\langle T_{1}T_{2}^2+T_{3}T_{4}^2+T_5T_{6}+T_7T_8 \rangle}
\\
\scriptstyle m \geq 0
\end{array}
$
&
\tiny{
\setlength{\arraycolsep}{1pt}
$
\left[
\begin{array}{cccccccc|cc}
0 & 1 & 0 & 1 & 1 & 1 & 1 & 1 & 1 \ldots & 1
\\
1 & 0 & 1 & 0 & 1 & 0 & 1 & 0 & 1 \ldots & 1
\end{array}
\right]
$
}
&
\tiny{
\setlength{\arraycolsep}{1pt}
$
\left[
\begin{array}{c}
4+m 
\\
3+m
\end{array}
\right]
$
}
&
\small{$m+5$}
\\
\midrule
4D
&
$
\begin{array}{c}
\frac
{\KK[T_1, \ldots , T_8, S_1,\ldots,S_m]}
{\langle T_{1}T_{2}^2+T_{3}T_{4}+T_5T_{6}+T_7T_8 \rangle}
\\
\scriptstyle m \geq 0
\end{array}
$
&
\tiny{
\setlength{\arraycolsep}{1pt}
$
\begin{array}{c}
\left[
\begin{array}{cccccccc|ccc}
0 & 1 & 1 & 1 & 1 & 1 & 1 & 1 & d_1 & 1 \ldots & 1
\\
1 & 0 & 1 & 0 & 1 & 0 & 1 & 0 & 1 & 1 \ldots & 1
\end{array}
\right]
\\[1em]
d_1 \in \left\{0, 1\right\}
\end{array}
$
}
&
\tiny{
\setlength{\arraycolsep}{1pt}
$
\left[
\begin{array}{c}
5\!+\!m\!-\!1\!+\!d_1
\\
3+m 
\end{array}
\right]
$
}
&
\small{$m+5$}
\\
\midrule
4E
&
$
\begin{array}{c}
\frac
{\KK[T_1, \ldots , T_8, S_1,\ldots,S_m]}
{\langle T_{1}T_{2}^3+T_{3}T_{4}^3+T_5T_{6}^3+T_7T_8^3 \rangle}
\\
\scriptstyle m \geq 0
\end{array}
$
&
\tiny{
\setlength{\arraycolsep}{1pt}
$
\left[
\begin{array}{cccccccc|cc}
0 & 1 & 0 & 1 & 0 & 1 & 0 & 1 & 0 \ldots & 0
\\
1 & 0 & 1 & 0 & 1 & 0 & 1 & 0 & 1 \ldots & 1
\end{array}
\right]
$
}
&
\tiny{
\setlength{\arraycolsep}{1pt}
$
\left[
\begin{array}{c}
3 
\\
3+m
\end{array}
\right]
$
}
&
\small{$m+5$}
\\
\midrule
4F
&
$
\begin{array}{c}
\frac
{\KK[T_1, \ldots , T_8, S_1,\ldots,S_m]}
{\langle T_{1}T_{2}^2+T_{3}T_{4}^2+T_5T_{6}^2+T_7T_8^2 \rangle}
\\
\scriptstyle m \geq 0
\end{array}
$
&
\tiny{
\setlength{\arraycolsep}{1pt}
$
\begin{array}{c}
\left[
\begin{array}{cccccccc|ccc}
0 & 1 & 0 & 1 & 0 & 1 & 0 & 1 & d_1  & 0 \ldots & 0
\\
1 & 0 & 1 & 0 & 1 & 0 & 1 & 0 & 1 & 1 \ldots & 1
\end{array}
\right]
\\[1em]
d_1 \in \left\{-1,0\right\}
\end{array}
$
}
&
\tiny{
\setlength{\arraycolsep}{1pt}
$
\left[
\begin{array}{c}
2+d_1
\\
3+m 
\end{array}
\right]
$
}
&
\small{$m+5$}
\\
\midrule
4G
&
$
\begin{array}{c}
\frac
{\KK[T_1, \ldots , T_8, S_1,\ldots,S_m]}
{\langle T_{1}T_{2}+T_{3}T_{4}+T_5T_{6}+T_7T_8 \rangle}
\\
\scriptstyle m \geq 0
\end{array}
$
&
\tiny{
\setlength{\arraycolsep}{1pt}
$
\begin{array}{c}
\left[
\begin{array}{cccccccc|cccc}
0 & 1 & 0 & 1 & 0 & 1 & 0 & 1 & d_1 & d_2 & 0 \ldots & 0
\\
1 & 0 & 1 & 0 & 1 & 0 & 1 & 0 & 1 & 1 & 1 \ldots & 1
\end{array}
\right]
\\[1em]
d_1,d_2 \leq 0, \ d_1+d_2 \geq -2
\end{array}
$
}
&
\tiny{
\setlength{\arraycolsep}{1pt}
$
\left[
\begin{array}{c}
3+d_1+d_2
\\
3+m 
\end{array}
\right]
$
}
&
\small{$m+5$}
\\
\midrule
5
&
$
\begin{array}{c}
\frac
{\KK[T_1, \ldots , T_8, S_1,\ldots,S_m]}
{\langle T_{1}T_{2}+T_{3}^2T_{4}+T_5^2T_{6}+T_7^2T_8 \rangle}
\\
\scriptstyle m \geq 1
\end{array}
$
&
\tiny{
\setlength{\arraycolsep}{1pt}
$
\begin{array}{c}
\left[
\begin{array}{cccccccc|cc}
0 & 2a+1 & a & 1 & a & 1 & a & 1 & 1 \ldots & 1
\\
1 & 1 & 1 & 0 & 1 & 0 & 1 & 0 & 0 \ldots & 0
\end{array}
\right]
\\[1em]
a \geq 0, \ m > 3a
\end{array}
$
}
&
\tiny{
\setlength{\arraycolsep}{1pt}
$
\left[
\begin{array}{c}
3a+ 3 + m
\\
3
\end{array}
\right]
$
}
&
\small{$m+5$}
\\
\midrule
6
&
$
\begin{array}{c}
\frac
{\KK[T_1, \ldots , T_8, S_1,\ldots,S_m]}
{\langle T_{1}T_{2}+T_{3}T_{4}+T_5^2T_{6}+T_7^2T_8 \rangle}
\\
\scriptstyle m \geq 1
\end{array}
$
&
\tiny{
\setlength{\arraycolsep}{1pt}
$
\begin{array}{c}
\left[
\begin{array}{cccccccc|cc}
0 & 2a_3+1 & a_1 & a_2 & a_3 & 1 & a_3 & 1 & 1 \ldots & 1
\\
1 & 1 & 1 & 1 & 1 & 0 & 1 & 0 & 0 \ldots & 0
\end{array}
\right]
\\[1em]
0 \leq a_1 \leq a_2, \ a_1 +a_2 = 2 a_3+1 
\\
m > 4a_3+1
\end{array}
$
}
&
\tiny{
\setlength{\arraycolsep}{1pt}
$
\left[
\begin{array}{c}
4a_3 + 3 + m
\\
4
\end{array}
\right]
$
}
&
\small{$m+5$}
\\
\midrule
7
&
$
\begin{array}{c}
\frac
{\KK[T_1, \ldots , T_8, S_1,\ldots,S_m]}
{\langle T_{1}T_{2}+T_{3}T_{4}+T_5T_{6}+T_7^2T_8 \rangle}
\\
\scriptstyle m \geq 1
\end{array}
$
&
\tiny{
\setlength{\arraycolsep}{1pt}
$
\begin{array}{c}
\left[
\begin{array}{cccccccc|cc}
0 & 2a_5+1 & a_1 & a_2 & a_3 & a_4 & a_5 & 1 & 1 \ldots & 1
\\
1 & 1 & 1 & 1 & 1 & 1 & 1 & 0 & 0 \ldots & 0
\end{array}
\right]
\\[1em]
a_i \geq 0, \ m > 5a_5+2,
\\
a_1 + a_2 = a_3 + a_4 = 2a_5 +1
\end{array}
$
}
&
\tiny{
\setlength{\arraycolsep}{1pt}
$
\left[
\begin{array}{c}
5a_5 + 3 + m
\\
5
\end{array}
\right]
$
}
&
\small{$m+5$}
\\
\midrule
8
&
$
\begin{array}{c}
\frac
{\KK[T_1, \ldots , T_8, S_1,\ldots,S_m]}
{\langle T_{1}T_{2}+T_{3}T_{4}+T_5T_{6}+T_7T_8 \rangle}
\\
\scriptstyle 1 \leq m \leq 5
\end{array}
$
&
\tiny{
\setlength{\arraycolsep}{1pt}
$
\left[
\begin{array}{cccccccc|cc}
0 & 0 & 0 & 0 & 0 & 0 & -1 & 1 & 1  \ldots & 1
\\
1 & 1 & 1 & 1 & 1 & 1 & 1 & 1 & 0 \ldots & 0
\end{array}
\right]
$
}
&
\tiny{
\setlength{\arraycolsep}{1pt}
$
\left[
\begin{array}{c}
m
\\
6
\end{array}
\right]
$
}
&
\small{$m+5$}
\\
\midrule
9
&
$
\begin{array}{c}
\frac
{\KK[T_1, \ldots , T_8, S_1,\ldots,S_m]}
{\langle T_{1}T_{2}+T_{3}T_{4}+T_5T_{6}+T_7T_8 \rangle}
\\
\scriptstyle m \geq 2
\end{array}
$
&
\tiny{
\setlength{\arraycolsep}{1pt}
$
\begin{array}{c}
\left[
\begin{array}{cccccccc|cc}
0 & a_1 & a_2 & a_3 & a_4 & a_5 & a_6 & a_7 & 1 \ldots & 1
\\
1 & 1 & 1 & 1 & 1 & 1 & 1 & 1 & 0 \ldots & 0
\end{array}
\right]
\\[1em]
a_i \geq 0, \ m > 3a_1,
\\
a_1 =a_2+a_3=a_4+a_5=a_6+a_7
\end{array}
$
}
&
\tiny{
\setlength{\arraycolsep}{1pt}
$
\left[
\begin{array}{c}
3a_1+m
\\
6
\end{array}
\right]
$
}
&
\small{$m+5$}
\\
\midrule
10
&
$
\begin{array}{c}
\frac
{\KK[T_1, \ldots , T_8, S_1,\ldots,S_m]}
{\langle T_{1}T_{2}+T_{3}T_{4}+T_5T_{6}+T_7T_8 \rangle}
\\
\scriptstyle m \geq 2
\end{array}
$
&
\tiny{
\setlength{\arraycolsep}{1pt}
$
\begin{array}{c}
\left[
\begin{array}{cccccccc|ccc}
0 & 0 & 0 & 0 & 0 & 0 & 0 & 0 & 1 & 1 \ldots & 1
\\
1 & 1 & 1 & 1 & 1 & 1 & 1 & 1 & 0 & d_2 \ldots & d_m
\end{array}
\right]
\\[1em]
0 \leq d_2\leq\dots \leq d_m, \ 0< d_m \leq 5,
\\
m d_m < 6 + d_2 + \ldots + d_m
\end{array}
$
}
&
\tiny{
\setlength{\arraycolsep}{1pt}
$
\left[
\begin{array}{c}
m 
\\
6\!+ \!\sum \!d_k
\end{array}
\right]
$
}
&
\small{$m+5$}
\\
\midrule
11
&
$
\begin{array}{c}
\frac
{\KK[T_1, \ldots , T_7, S_1,\ldots,S_m]}
{\langle T_{1}T_{2}+T_{3}T_{4}+T_5T_{6}+T_7^2 \rangle}
\\
\scriptstyle 1\leq m \leq 4
\end{array}
$
&
\tiny{
\setlength{\arraycolsep}{1pt}
$
\begin{array}{c}
\left[
\begin{array}{ccccccc|ccc}
-1 & 1 & 0 & 0 & 0 & 0 & 0 & 1 & \ldots & 1
\\
1 & 1 & 1 & 1 & 1 & 1 & 1 & 0& \ldots & 0
\end{array}
\right]
\end{array}
$
}
&
\tiny{
\setlength{\arraycolsep}{1pt}
$
\left[
\begin{array}{c}
m
\\
5
\end{array}
\right]
$
}
&
\small{$m+4$}
\\
\midrule
12
&
$
\begin{array}{c}
\frac
{\KK[T_1, \ldots , T_7, S_1,\ldots,S_m]}
{\langle T_{1}T_{2}+T_{3}T_{4}+T_5T_{6}+T_7^2 \rangle}
\\
\scriptstyle m \geq 2
\end{array}
$
&
\tiny{
\setlength{\arraycolsep}{1pt}
$
\begin{array}{c}
\left[
\begin{array}{ccccccc|ccc}
0 & 2a_5 & a_1 & a_2 & a_3 & a_4 & a_5 &  1  & \ldots & 1
\\
1 & 1 & 1 & 1 & 1 & 1 & 1 & 0 & \ldots & 0
\end{array}
\right]
\\[1em]
a_1+a_2 = a_3+a_4 = 2a_5,
\\
a_i \geq 0, \  m > 5a_5 
\end{array}
$
}
&
\tiny{
\setlength{\arraycolsep}{1pt}
$
\left[
\begin{array}{c}
m+5a_5
\\
5
\end{array}
\right]
$
}
&
\small{$m+4$}
\\
\midrule
13
&
$
\begin{array}{c}
\frac
{\KK[T_1, \ldots , T_7, S_1,\ldots,S_m]}
{\langle T_{1}T_{2}+T_{3}T_{4}+T_5T_{6}+T_7^2 \rangle}
\\
\scriptstyle m \geq 2
\end{array}
$
&
\tiny{
\setlength{\arraycolsep}{1pt}
$
\begin{array}{c}
\left[
\begin{array}{ccccccc|cccc}
0 & 0 & 0 & 0 & 0 &0 &0 &  1  & 1 & \ldots & 1
\\
1 & 1 & 1 & 1 & 1 & 1 & 1 & 0 &d_2 & \ldots & d_m
\end{array}
\right]
\\[1em]
0 \leq d_2 \leq \ldots \leq d_m, \ 0 < d_m,
\\
m  d_m < 5 +d_2+ \ldots + d_m
\end{array}
$
}
&
\tiny{
\setlength{\arraycolsep}{1pt}
$
\left[
\begin{array}{c}
m
\\
5\!+ \!\sum \! d_k
\end{array}
\right]
$
}
&
\small{$m+4$}
\\
\midrule
14
&
\quad
$
\frac
{\KK[T_1, \ldots , T_{10}]}
{\left\langle\!\!
\tiny{
\begin{array}{c}
T_{1}T_{2}+T_{3}T_{4}+T_5T_{6}+T_7T_8, 
\\
\lambda_1T_{3}T_{4}+\lambda_2T_5T_{6}+T_7T_8+T_9T_{10} 
\end{array}
}
\!\!\right\rangle}
$
&
\tiny{
\setlength{\arraycolsep}{1pt}
$
\left[
\begin{array}{cccccccccc}
1 & 0 & 1 & 0 & 1 & 0 & 1 & 0 & 1 & 0
\\
0 & 1 & 0 & 1 & 0 & 1 & 0 & 1 & 0 & 1
\end{array}
\right]
$
}
&
\tiny{
\setlength{\arraycolsep}{1pt}
$
\left[
\begin{array}{c}
3
\\
3
\end{array}
\right]
$
}
&
$6$
\\
\bottomrule
\end{longtable}
}
}
\noindent
Moreover, each of the listed data defines 
a smooth Fano general arrangement variety of 
true complexity two and Picard number two.
\end{theorem}

\begin{remark}
Some of the above Fano varieties are 
intrinsic quadrics. 
Here is the overlap of Theorem~\ref{theo:Fano}
with~\cite{FaHa}*{Cor.~1.2}:
\begin{enumerate}
\item Cases~10 and~13 are intrinsic quadrics of Type~1,
\item Cases~9 and~12 are intrinsic quadrics of Type~2,
\item Cases~8 and~11 are intrinsic quadrics of Type~3,
\item Case~4.G is an intrinsic quadric of Type~4.
\end{enumerate}
\end{remark}

Let us discuss some aspects of the geometry of the 
Fano varieties listed in Theorem~\ref{theo:Fano}.
We take a look at elementary contractions, i.e.,
the morphisms obtained by passing to facets of 
the ample cone with respect to the Mori chamber 
decomposition, which is directly computable 
in terms of the data listed in Theorem~\ref{theo:Fano};
use Proposition~\ref{prop:PicandCones} 
and Remark~\ref{rem:Moridecomp}.
Moreover, we look at small degenerations, that means 
degenerations with fibers all sharing the same
divisor class group.
In fact, degenerating the quadrinomial equations 
of the Cox ring into trinomial ones, reflects a 
degeneration of Cox rings inducing a small 
degeneration of the underlying Fano variety into 
a possibly singular variety with a torus action of 
complexity one. 

We now explicitly go through the list of 
Theorem~\ref{theo:Fano} and provide 
basic information in the subsequent table; 
more details will be made available 
elsewhere.
Let us explain how to read the table.
By $Q_k$, we denote the smooth projective 
quadric of dimension $k$
and by $Q_{k,l} \subseteq \PP_{l}$
the projective quadric of rank $k$ in $\PP_l$.
We write $Y_{a;1^k,d^l}$ for a hypersurface 
of degree $a$ in the weighted projective 
space $\PP(1^k,d^l)$, where we do not specify 
the equation, and we set 
$$
Y_{4B} 
\  = \ 
V(T_0^3 + T_1T_2^2 + T_3T_4^2  + T_5T_6^2)
\ \subseteq \
\PP_{m+6},
$$
$$
Y_{4F} 
\  = \ 
V(T_0T_1^2 + T_2T_3^2 + T_4T_5^2  + T_6T_7^2)
\ \subseteq \
\PP_{m+6}.
$$
As we consider smooth Fano varieties of 
Picard number two, there will be at most 
two elementary contractions for each. 
If we have a birational elementary contraction,
then a prime divisor gets contracted.
In this case we write $X \sim Y$ and denote 
by $C \subseteq Y$ the center of this contraction.
The other possibility is that we have a Mori 
fiber space.
Then we write $X \to Y$ and denote by 
$F_{\mathrm{gen}}$ the general fiber.
If there are no special fibers, then we
write just $F$ for the fiber.
Moreover, when we say that a variety is Gorenstein,
terminal, etc. then we mean that it is singular
but has at most Gorenstein, terminal, etc. 
singularities.
We computed small degenerations 
for every case in the lowest dimensions.
The resulting varieties are always normal
and Fano with a torus action of complexity one.
The properties of being Gorenstein, terminal 
etc. have been checked  
using~\cites{BeHaHuNi,HaKe}.
If we say two, three, etc. degenerations, 
then this means that we found small 
degenerations into two, three, etc. 
non-isomorphic Fano $\TT$-varieties of complexity 
one.
\goodbreak

\begin{remark}
The following table lists the 
elementary contractions and 
small degenerations obtained 
via degenerating the Cox ring 
for the Fano varieties of 
Theorem~\ref{theo:Fano}.

\begin{center}
{\small
\setlength{\tabcolsep}{3pt}
\begin{longtable}{ccccl}
No.
&
$\dim(X)$
&
Contraction~1
&
Contraction~2
&
$\quad$
Small~Degenerations 
\\
\toprule
1
&
$6$
&
$
\begin{array}{c}
X \sim Q_6
\\
C = Q_4
\end{array}
$
&
$
\begin{array}{c}
X \to \PP_1
\\
F_{\mathrm{gen}} = Q_5
\end{array}
$
&
$
\begin{array}{l}
{\text{two Gorenstein, terminal}}
\\
{\text{locally factorial}}
\end{array}
$
\\
\midrule
2
&
$6$
&
$
\begin{array}{c}
X \sim Q_6
\\
C = \PP_2
\end{array}
$
&
$
\begin{array}{c}
X \to \PP_4
\\
F = \PP_2
\end{array}
$
&
$
\begin{array}{l}
{\text{two Gorenstein, terminal}}
\\
{\text{locally factorial}}
\end{array}
$
\\
\midrule
3
&
$5$
&
$
\begin{array}{c}
X \sim Q_5
\\
C = Q_3
\end{array}
$
&
$
\begin{array}{c}
X \to \PP_1
\\
F_{\mathrm{gen}} = Q_4
\end{array}
$
&
$
\begin{array}{l}
{\text{three Gorenstein, terminal}}
\\
{\text{locally factorial}}
\end{array}
$
\\
\midrule
4A
&
$m+5$
&
$
\begin{array}{c}
X \sim Y_{3;1^4,2^{m+3}}
\\
C = \PP_{m+2}
\end{array}
$
&
$
\begin{array}{c}
X \to \PP_3
\\
F = \PP_{m+2}
\end{array}
$
&
$
\begin{array}{l}
{\text{$\dim(X) \le 6$: two Gorenstein,}}
\\
{\text{terminal, locally factorial}}
\end{array}
$
\\
\midrule
4B
&
$m+5$
&
$
\begin{array}{c}
X \sim Y_{\mathrm{4B}}
\\
C = \PP_{m+2}
\end{array}
$
&
$
\begin{array}{c}
X \to \PP_3
\\
F = \PP_{m+2}
\end{array}
$
&
$
\begin{array}{l}
{\text{$\dim(X) \le 6$: two Gorenstein,}}
\\
{\text{log terminal, locally factorial}}
\end{array}
$
\\
\midrule
4C
&
$m+5$
&
---
&
$
\begin{array}{c}
X \to \PP_3
\\
F = \PP_{m+2}
\end{array}
$
&
$
\begin{array}{l}
{\text{$\dim(X) \le 6$: two Gorenstein,}}
\\
{\text{terminal, locally factorial}}
\end{array}
$
\\
\midrule
4D
&
$m+5$
&
$
\begin{array}{c}
\scriptstyle{\text{if } 
d_1 = 1 \text{ or } m=0\text{:}}
\\
X \sim Q_{7,m+6}
\\
C = \PP_{m+2}
\end{array}
$
&
$
\begin{array}{c}
X \to \PP_3
\\
F = \PP_{m+2}
\end{array}
$
&
$
\begin{array}{l}
{\text{$\dim(X) \le 6$: two Gorenstein,}}
\\
{\text{terminal, locally factorial}}
\end{array}
$
\\
\midrule
4E
&
$m+5$
&
$
\begin{array}{c}
X \to \PP_{m+3}
\\
F_{\mathrm{gen}} = Y_{3;1^4}
\end{array}
$
&
$
\begin{array}{c}
X \to \PP_3
\\
F = \PP_{m+2}
\end{array}
$
&
$
\begin{array}{l}
{\text{$\dim(X) \le 6$: one Gorenstein,}}
\\
{\text{locally factorial}}
\end{array}
$
\\
\midrule
4F
&
$m+5$
&
$
\begin{array}{c}
\scriptstyle{\text{if } 
d_1 = 0 \text{ or } m=0\text{:}}
\\
X \to \PP_{m+3}
\\
F_{\mathrm{gen}} = \PP_1 \times \PP_1
\\
\scriptstyle{\text{if } 
d_1 = -1 \text{:}}
\\
X \sim Y_{4F}
\\
C = \PP_{m+2}
\end{array}
$
&
$
\begin{array}{c}
X \to \PP_3
\\
F = \PP_{m+2}
\end{array}
$
&
$
\begin{array}{l}
{\text{$\dim(X) \le 6$: one Gorenstein,}}
\\
{\text{log terminal, locally factorial}}
\end{array}
$
\\
\midrule
4G
&
$m+5$
&
$
\begin{array}{c}
\scriptstyle{\text{if } 
d_i = 0 \text{ or } m = 0 \text{:}}
\\
X \to \PP_{m+3}
\\
F_{\mathrm{gen}} = \PP_2
\\
\scriptstyle{\text{if } 
d_1 = -1 \text{ and } d_2 = 0 \text{:}}
\\
X \sim Q_{7,m+6}
\\
\scriptstyle{\text{if } 
d_1 = -2 \text{ and } d_2 = 0 \text{:}}
\\
X \sim Y_{3;1^4,2^{m+3}}
\end{array}
$
&
$
\begin{array}{c}
X \to \PP_3
\\
F = \PP_{m+2}
\\
\end{array}
$
&
$
\begin{array}{l}
{\text{$\dim(X) \le 6$: one Gorenstein,}}
\\
{\text{terminal, locally factorial}}
\end{array}
$
\\
\midrule
5
&
$m+5$
&
$
\begin{array}{c}
X \to \PP_{m+2}
\\
F_{\mathrm{gen}} = Q_3
\end{array}
$
&
---
&
$
\begin{array}{l}
{\text{$\dim(X) = 6$: one Gorenstein,}}
\\
{\text{terminal, locally factorial;}}
\\
{\text{one Gorenstein, log terminal}}
\\
{\text{locally factorial}}
\end{array}
$
\\
\midrule
6
&
$m+5$
&
$
\begin{array}{c}
X \to \PP_{m+1}
\\
F_{\mathrm{gen}} = Q_4
\end{array}
$
&
---
&
$
\begin{array}{l}
{\text{$\dim(X) = 7$: two Gorenstein,}}
\\
{\text{terminal, locally factorial}}
\end{array}
$
\\
\midrule
7
&
$m+5$
&
$
\begin{array}{c}
X \to \PP_{m}
\\
F_{\mathrm{gen}} = Q_5
\end{array}
$
&
---
&
$
\begin{array}{l}
{\text{$\dim(X) = 8$: two Gorenstein,}}
\\
{\text{terminal, locally factorial}}
\end{array}
$
\\
\midrule
8
&
$m+5$
&
$
\begin{array}{c}
X \sim \PP_{m+5}
\\
C = Q_4
\end{array}
$
&
$
\begin{array}{c}
\scriptstyle{\text{if } 
d_1 = -1 \text{:}}
\\
X \sim Q_6
\\
C = \{\mathrm{pt}\}
\end{array}
$
&
$
\begin{array}{l}
{\text{$\dim(X) = 6$: one Gorenstein,}}
\\
{\text{terminal, locally factorial;}}
\\
{\text{one of Gorenstein index~2,}}
\\
{\text{terminal, $\QQ$-factorial}}
\end{array}
$
\\
\midrule
9
&
$m+5$
&
$
\begin{array}{c}
X \to \PP_{m-1}
\\
F_{\mathrm{gen}} = Q_6
\end{array}
$
&
$
\begin{array}{c}
\scriptstyle{\text{if } 
a_1 = \ldots = a_7 = 0 \text{:}}
\\
X \to Q_6
\\
F = \PP_{m-1}
\end{array}
$
&
$
\begin{array}{l}
{\text{$\dim(X) = 7$: one Gorenstein,}}
\\
{\text{terminal, locally factorial}}
\end{array}
$
\\
\midrule
10
&
$m+5$
&
$
\begin{array}{c}
X \to Q_6
\\
F_{\mathrm{gen}} = \PP_{m-1}
\end{array}
$
&
$
\begin{array}{c}
\scriptstyle{\text{if } 
0 < d_2 = \ldots = d_m \text{:}}
\\
X \sim Y_{2;1^8,d_2^{m-1}}
\\
C = \PP_{m-2}
\end{array}
$
&
$
\begin{array}{l}
{\text{$\dim(X) = 7$: one Gorenstein,}}
\\
{\text{terminal, locally factorial}}
\end{array}
$
\\
\midrule
11
&
$m+4$
&
$
\begin{array}{c}
X \sim \PP_{m+4}
\\
C = Q_3
\end{array}
$
&
$
\begin{array}{c}
\scriptstyle{\text{if } 
m = 1 \text{:}}
\\
X \sim Q_5
\\
C = \{\mathrm{pt}\}
\end{array}
$
&
$
\begin{array}{l}
{\text{$\dim(X) =5$: two Gorenstein,}}
\\
{\text{terminal, locally factorial;}}
\\
{\text{one Gorenstein, terminal}}
\\
{\text{$\QQ$-factorial.}}
\\
{\text{$\dim(X) =6$: two Gorenstein,}}
\\
{\text{terminal, locally factorial;}}
\\
{\text{one of Gorenstein index 2,}}
\\
{\text{terminal, $\QQ$-factorial}}
\end{array}
$
\\
\midrule
12
&
$m+4$
&
$
\begin{array}{c}
X \to \PP_{m-1}
\\
F_{\mathrm{gen}} = Q_5
\end{array}
$
&
$
\begin{array}{c}
\scriptstyle{\text{if } 
a_1 = \ldots = a_5 = 0 \text{:}}
\\
X \to \PP_{m-1}
\\
F = Q_5
\end{array}
$
&
$
\begin{array}{l}
{\text{$\dim(X) = 6$: two Gorenstein, }}
\\
{\text{terminal, locally factorial}}
\end{array}
$
\\
\midrule
13
&
$m+4$
&
$
\begin{array}{c}
X \to Q_5
\\
F_{\mathrm{gen}} = \PP_{m-1}
\end{array}
$
&
$
\begin{array}{c}
\scriptstyle{\text{if } 
0 < d_2 = \ldots = d_m \text{:}}
\\
X \sim Y_{2;1^7,d_2^{m-1}}
\\
C = \PP_{m-2}
\end{array}
$
&
$
\begin{array}{l}
{\text{$\dim(X) = 6$: two Gorenstein, }}
\\
{\text{terminal, locally factorial}}
\end{array}
$
\\
\midrule
14
&
$6$
&
$
\begin{array}{c}
X \to \PP_4
\\
F_{\mathrm{gen}} = \PP_2
\end{array}
$
&
$
\begin{array}{c}
X \to \PP_4
\\
F_{\mathrm{gen}} = \PP_2
\end{array}
$
&
$
\begin{array}{l}
{\text{one Gorenstein, terminal}}
\\
{\text{locally factorial}}
\end{array}
$
\\
\bottomrule
\end{longtable}
}
\end{center}
\end{remark}

\begin{remark}
\label{rem:duplication}
All varieties of Theorem~\ref{theo:Fano} can be 
constructed out of a finite set of starting 
varieties listed in Theorem~\ref{theo:proj} via 
iterated \emph{duplication of free weights}
as introduced in~\cite{FaHaNi}*{Constr.~5.1}.
In this procedure, one takes a Cox ring 
generator~$S_k$ of $X$ not occurring in the 
defining relations and 
constructs a new Cox ring by adding a further 
free generator $S_k'$ of the same degree as $S_k$.
The resulting variety $X'$ is of one dimension higher.
In terms of birational geometry, the duplication 
of a free weight means taking an elementary 
contraction $\widetilde X_1 \to X$ with fiber 
$\PP_1$, passing via a series of small 
quasimodifications to $\widetilde X_t$ and then 
performing a contraction of a prime divisor
$\widetilde X_t \to X'$, see~\cite{FaHaNi}*{Prop.~5.3}.
It follows from~\cite{FaHaNi}*{Prop.~5.4, Thm.~5.5} 
that every smooth Fano variety of true complexity 
one and Picard number two is of dimension 4 to~7
or arises via iterated duplications of free weights 
from a finite set of smooth projective varieties 
of true complexity one and Picard number two of 
dimension 4 to~7.
For the Fano general arrangement varieties of true 
complexity two listed in Theorem~\ref{theo:Fano}, 
one directly establishes the analogous statement  
with a finite set of starting varieties of dimensions 
5 to~8.
It would be interesting to see if the smooth Fano general 
arrangement varieties of Picard number two but higher 
complexity behave similarly.
\end{remark}

\begin{theorem}
\label{theo:AlFano}
Every smooth projective truly almost Fano 
general arrangement variety of true complexity 
two and Picard number two 
is isomorphic to precisely one of the following 
varieties $X$, specified by their Cox ring 
$\mathcal{R}(X)$, 
the matrix $[w_1, \dots, w_r]$ of generator degrees 
and an ample class $u \in \Cl(X) = \ZZ^2$. 

\medskip

{\centering
{\small
\setlength{\tabcolsep}{0pt}
\begin{longtable}{ccccc}
No.
&
\small{$\mathcal{R}(X)$}
&
\small{$[w_1,\ldots, w_r]$}
&
\small{$u$}
&
\small{$\dim(X)$}
\\
\toprule
4A
&
$
\begin{array}{c}
\frac
{\KK[T_1, \ldots , T_8, S_1,\ldots,S_m]}
{\langle T_{1}T_{2}^4+T_{3}T_{4}+T_5T_{6}+T_7T_8 \rangle}
\\
\scriptstyle m \geq 0
\end{array}
$
&
\tiny{
\setlength{\arraycolsep}{1pt}
$
\left[
\begin{array}{cccccccc|cc}
0 & 1 & 3 & 1 & 3 & 1 & 3 & 1 & 3 \ldots & 3
\\
1 & 0 & 1 & 0 & 1 & 0 & 1 & 0 & 1  \ldots & 1
\end{array}
\right]
$
}
&
\tiny{
\setlength{\arraycolsep}{1pt}
$
\left[
\begin{array}{c}
4
\\
1
\end{array}
\right]
$
}
&
\small{$m+5$}
\\
\midrule
4B
&
$
\begin{array}{c}
\frac
{\KK[T_1, \ldots , T_8, S_1,\ldots,S_m]}
{\langle T_{1}T_{2}^3+T_{3}T_{4}+T_5T_{6}+T_7T_8 \rangle}
\\
\scriptstyle m \geq 0
\end{array}
$
&
\tiny{
\setlength{\arraycolsep}{1pt}
$
\left[
\begin{array}{cccccccc|cccc}
0 & 1 & 2 & 1 & 2 & 1 & 2 & 1 & 1 & 2&  \ldots & 2
\\
1 & 0 & 1 & 0 & 1 & 0 & 1 & 0 & 1 & 1& \ldots & 1
\end{array}
\right]
$
}
&
\tiny{
\setlength{\arraycolsep}{1pt}
$
\left[
\begin{array}{c}
3
\\
1
\end{array}
\right]
$
}
&
\small{$m+5$}
\\
\midrule
4C
&
$
\begin{array}{c}
\frac
{\KK[T_1, \ldots , T_8, S_1,\ldots,S_m]}
{\langle T_{1}T_{2}^3+T_{3}T_{4}^2+T_5T_{6}+T_7T_8 \rangle}
\\
\scriptstyle m \geq 0
\end{array}
$
&
\tiny{
\setlength{\arraycolsep}{1pt}
$
\left[
\begin{array}{cccccccc|ccc}
0 & 1 & 1 & 1 & 2 & 1 & 2 & 1 & 2&  \ldots & 2
\\
1 & 0 & 1 & 0 & 1 & 0 & 1 & 0 & 1& \ldots & 1
\end{array}
\right]
$
}
&
\tiny{
\setlength{\arraycolsep}{1pt}
$
\left[
\begin{array}{c}
3
\\
1
\end{array}
\right]
$
}
&
\small{$m+5$}
\\
\midrule
4D
&
$
\begin{array}{c}
\frac
{\KK[T_1, \ldots , T_8, S_1,\ldots,S_m]}
{\langle T_{1}T_{2}^4+T_{3}T_{4}^2+T_5T_{6}^2+T_7T_8^2 \rangle}
\\
\scriptstyle m \geq 0
\end{array}
$
&
\tiny{
\setlength{\arraycolsep}{1pt}
$
\left[
\begin{array}{cccccccc|ccc}
0 & 1 & 2 & 1 & 2 & 1 & 2 & 1 &2&  \ldots & 2
\\
1 & 0 & 1 & 0 & 1 & 0 & 1 & 0 & 1& \ldots & 1
\end{array}
\right]
$
}
&
\tiny{
\setlength{\arraycolsep}{1pt}
$
\left[
\begin{array}{c}
3
\\
1
\end{array}
\right]
$
}
&
\small{$m+5$}
\\
\midrule
4E
&
$
\begin{array}{c}
\frac
{\KK[T_1, \ldots , T_8, S_1,\ldots,S_m]}
{\langle T_{1}T_{2}+T_{3}T_{4}+T_5T_{6}+T_7T_8 \rangle}
\\
\scriptstyle m \geq 0
\end{array}
$
&
\tiny{
\setlength{\arraycolsep}{1pt}
$
\left[
\begin{array}{cccccccc|ccc}
0 & 1 & 0 & 1 & 0 & 1 & 0 & 1 & 1&  \ldots & 1
\\
1 & 0 & 1 & 0 & 1 & 0 & 1 & 0 & 1& \ldots & 1
\end{array}
\right]
$
}
&
\tiny{
\setlength{\arraycolsep}{1pt}
$
\left[
\begin{array}{c}
2
\\
1
\end{array}
\right]
$
}
&
\small{$m+5$}
\\
\midrule
4F
&
$
\begin{array}{c}
\frac
{\KK[T_1, \ldots , T_8, S_1,\ldots,S_m]}
{\langle T_{1}T_{2}^2+T_{3}T_{4}+T_5T_{6}+T_7T_8 \rangle}
\\
\scriptstyle m \geq 0
\end{array}
$
&
\tiny{
\setlength{\arraycolsep}{1pt}
$
\left[
\begin{array}{cccccccc|ccccc}
0 & 1 & 1 & 1 & 1 & 1 & 1 & 1 & 0 & 0 & 1&  \ldots & 1
\\
1 & 0 & 1 & 0 & 1 & 0 & 1 & 0 & 1 & 1& 1 & \ldots & 1
\end{array}
\right]
$
}
&
\tiny{
\setlength{\arraycolsep}{1pt}
$
\left[
\begin{array}{c}
2
\\
1
\end{array}
\right]
$
}
&
\small{$m+5$}
\\
\midrule
4G
&
$
\begin{array}{c}
\frac
{\KK[T_1, \ldots , T_8, S_1,\ldots,S_m]}
{\langle T_{1}T_{2}^2+T_{3}T_{4}+T_5T_{6}+T_7T_8 \rangle}
\\
\scriptstyle m \geq 0
\end{array}
$
&
\tiny{
\setlength{\arraycolsep}{1pt}
$
\left[
\begin{array}{cccccccc|cccc}
0 & 1 & 1 & 1 & 1 & 1 & 1 & 1 & -1 & 1&  \ldots & 1
\\
1 & 0 & 1 & 0 & 1 & 0 & 1 & 0 & 1 & 1& \ldots & 1
\end{array}
\right]
$
}
&
\tiny{
\setlength{\arraycolsep}{1pt}
$
\left[
\begin{array}{c}
2
\\
1
\end{array}
\right]
$
}
&
\small{$m+5$}
\\
\midrule
4H
&
$
\begin{array}{c}
\frac
{\KK[T_1, \ldots , T_8, S_1,\ldots,S_m]}
{\langle T_{1}T_{2}^2+T_{3}T_{4}^2+T_5T_{6}^2+T_7T_8 \rangle}
\\
\scriptstyle m \geq 0
\end{array}
$
&
\tiny{
\setlength{\arraycolsep}{1pt}
$
\left[
\begin{array}{cccccccc|ccc}
0 & 1 & 0 & 1 & 0 & 1 & 1 & 1 & 1 & \ldots & 1
\\
1 & 0 & 1 & 0 & 1 & 0 & 1 & 0 & 1& \ldots & 1
\end{array}
\right]
$
}
&
\tiny{
\setlength{\arraycolsep}{1pt}
$
\left[
\begin{array}{c}
2
\\
1
\end{array}
\right]
$
}
&
\small{$m+5$}
\\
\midrule
4I
&
$
\begin{array}{c}
\frac
{\KK[T_1, \ldots , T_8, S_1,\ldots,S_m]}
{\langle T_{1}T_{2}^3+T_{3}T_{4}^2+T_5T_{6}^2+T_7T_8^2 \rangle}
\\
\scriptstyle m \geq 0
\end{array}
$
&
\tiny{
\setlength{\arraycolsep}{1pt}
$
\left[
\begin{array}{cccccccc|cccc}
0 & 1 & 1 & 1 & 1 & 1 & 1 & 1 & 0 & 1&  \ldots & 1
\\
1 & 0 & 1 & 0 & 1 & 0 & 1 & 0 & 1 & 1& \ldots & 1
\end{array}
\right]
$
}
&
\tiny{
\setlength{\arraycolsep}{1pt}
$
\left[
\begin{array}{c}
2
\\
1
\end{array}
\right]
$
}
&
\small{$m+5$}
\\
\midrule
4J
&
$
\begin{array}{c}
\frac
{\KK[T_1, \ldots , T_8, S_1,\ldots,S_m]}
{\langle T_{1}T_{2}^2+T_{3}T_{4}^2+T_5T_{6}+T_7T_8 \rangle}
\\
\scriptstyle m \geq 0
\end{array}
$
&
\tiny{
\setlength{\arraycolsep}{1pt}
$
\left[
\begin{array}{cccccccc|cccc}
0 & 1 & 0 & 1 & 1 & 1 & 1 & 1 & 0 &1 &  \ldots & 1
\\
1 & 0 & 1 & 0 & 1 & 0 & 1 & 0 & 1 & 1& \ldots & 1
\end{array}
\right]
$
}
&
\tiny{
\setlength{\arraycolsep}{1pt}
$
\left[
\begin{array}{c}
2
\\
1
\end{array}
\right]
$
}
&
\small{$m+5$}
\\
\midrule
4K
&
$
\begin{array}{c}
\frac
{\KK[T_1, \ldots , T_8, S_1,\ldots,S_m]}
{\langle T_{1}T_{2}^4+T_{3}T_{4}^3+T_5T_{6}^3+T_7T_8^3 \rangle}
\\
\scriptstyle m \geq 0
\end{array}
$
&
\tiny{
\setlength{\arraycolsep}{1pt}
$
\left[
\begin{array}{cccccccc|ccc}
0 & 1 & 1 & 1 & 1 & 1 & 1 & 1 & 1 &  \ldots & 1
\\
1 & 0 & 1 & 0 & 1 & 0 & 1 & 0 & 1 &\ldots & 1
\end{array}
\right]
$
}
&
\tiny{
\setlength{\arraycolsep}{1pt}
$
\left[
\begin{array}{c}
2
\\
1
\end{array}
\right]
$
}
&
\small{$m+5$}
\\
\midrule
4L
&
$
\begin{array}{c}
\frac
{\KK[T_1, \ldots , T_8, S_1,\ldots,S_m]}
{\langle T_{1}T_{2}^3+T_{3}T_{4}^2+T_5T_{6}^2+T_7T_8^2 \rangle}
\\
\scriptstyle m \geq 0
\end{array}
$
&
\tiny{
\setlength{\arraycolsep}{1pt}
$
\left[
\begin{array}{cccccccc|cccc}
0 & 1 & 0 & 1 & 1 & 1 & 1 & 1 & 1 &  \ldots & 1
\\
1 & 0 & 1 & 0 & 1 & 0 & 1 & 0 & 1 & \ldots & 1
\end{array}
\right]
$
}
&
\tiny{
\setlength{\arraycolsep}{1pt}
$
\left[
\begin{array}{c}
2 
\\
1
\end{array}
\right]
$
}
&
\small{$m+5$}
\\
\midrule
4M
&
$
\begin{array}{c}
\frac
{\KK[T_1, \ldots , T_8, S_1,\ldots,S_m]}
{\langle T_{1}T_{2}+T_{3}T_{4}+T_5T_{6}+T_7T_8 \rangle}
\\
\scriptstyle m \geq 0
\end{array}
$
&
\tiny{
\setlength{\arraycolsep}{1pt}
$
\begin{array}{c}
\left[
\begin{array}{cccccccc|ccc}
0 & 1 & 0 & 1 & 0 & 1 & 0 & 1 & d_1 &  \ldots & d_m
\\
1 & 0 & 1 & 0 & 1 & 0 & 1 & 0 & 1 & \ldots & 1
\end{array}
\right]
\\[1em]
d_k\leq 0,\ \sum d_k = -3
\end{array}
$
}
&
\tiny{
\setlength{\arraycolsep}{1pt}
$
\left[
\begin{array}{c}
1
\\
1
\end{array}
\right]
$
}
&
\small{$m+5$}
\\
\midrule
4N
&
$
\begin{array}{c}
\frac
{\KK[T_1, \ldots , T_8, S_1,\ldots,S_m]}
{\langle T_{1}T_{2}^2+T_{3}T_{4}^2+T_5T_{6}^2+T_7T_8^2 \rangle}
\\
\scriptstyle m \geq 0
\end{array}
$
&
\tiny{
\setlength{\arraycolsep}{1pt}
$
\begin{array}{c}
\left[
\begin{array}{cccccccc|ccc}
0 & 1 & 0 & 1 & 0 & 1 & 0 & 1 & d_1 &  \ldots & d_m
\\
1 & 0 & 1 & 0 & 1 & 0 & 1 & 0 & 1& \ldots & 1
\end{array}
\right]
\\[1em]
d_k\leq 0,\ \sum d_k = -2
\end{array}
$
}
&
\tiny{
\setlength{\arraycolsep}{1pt}
$
\left[
\begin{array}{c}
1
\\
1
\end{array}
\right]
$
}
&
\small{$m+5$}
\\
\midrule
4O
&
$
\begin{array}{c}
\frac
{\KK[T_1, \ldots , T_8, S_1,\ldots,S_m]}
{\langle T_{1}T_{2}^3+T_{3}T_{4}^3+T_5T_{6}^3+T_7T_8^3\rangle}
\\
\scriptstyle m \geq 0
\end{array}
$
&
\tiny{
\setlength{\arraycolsep}{1pt}
$
\left[
\begin{array}{cccccccc|cccc}
0 & 1 & 0 & 1 & 0 & 1 & 0 & 1 & -1 & 0 &  \ldots & 0
\\
1 & 0 & 1 & 0 & 1 & 0 & 1 & 0 & 1 & 1& \ldots & 1
\end{array}
\right]
$
}
&
\tiny{
\setlength{\arraycolsep}{1pt}
$
\left[
\begin{array}{c}
1 
\\
1
\end{array}
\right]
$
}
&
\small{$m+5$}
\\

\midrule
4P
&
$
\begin{array}{c}
\frac
{\KK[T_1, \ldots , T_8, S_1,\ldots,S_m]}
{\langle T_{1}T_{2}^4+T_{3}T_{4}^4+T_5T_{6}^4+T_7T_8^4\rangle}
\\
\scriptstyle m \geq 0
\end{array}
$
&
\tiny{
\setlength{\arraycolsep}{1pt}
$
\left[
\begin{array}{cccccccc|cccc}
0 & 1 & 0 & 1 & 0 & 1 & 0 & 1 & 0 &  \ldots & 0
\\
1 & 0 & 1 & 0 & 1 & 0 & 1 & 0 & 1 & \ldots & 1
\end{array}
\right]
$
}
&
\tiny{
\setlength{\arraycolsep}{1pt}
$
\left[
\begin{array}{c}
1
\\
1
\end{array}
\right]
$
}
&
\small{$m+5$}
\\
\midrule
5
&
$
\begin{array}{c}
\frac
{\KK[T_1, \ldots , T_8, S_1,\ldots,S_m]}
{\langle T_{1}T_{2}+T_{3}^2T_{4}+T_5^2T_{6}+T_7^2T_8 \rangle}
\\
\scriptstyle m \geq 0
\end{array}
$
&
\tiny{
\setlength{\arraycolsep}{1pt}
$
\begin{array}{c}
\left[
\begin{array}{cccccccc|cc}
0 & 2a+1 & a & 1 & a & 1 & a & 1 & 1 \ldots & 1
\\
1 & 1 & 1 & 0 & 1 & 0 & 1 & 0 & 0 \ldots & 0
\end{array}
\right]
\\[1em]
a \geq 0, m = 3a
\end{array}
$
}
&
\tiny{
\setlength{\arraycolsep}{1pt}
$
\left[
\begin{array}{c}
2a+2
\\
1
\end{array}
\right]
$
}
&
\small{$m+5$}
\\
\midrule
6
&
$
\begin{array}{c}
\frac
{\KK[T_1, \ldots , T_8, S_1,\ldots,S_m]}
{\langle T_{1}T_{2}+T_{3}T_{4}+T_5^2T_{6}+T_7^2T_8 \rangle}
\\
\scriptstyle m \geq 0
\end{array}
$
&
\tiny{
\setlength{\arraycolsep}{1pt}
$
\begin{array}{c}
\left[
\begin{array}{cccccccc|cc}
0 & 2a_3+1 & a_1 & a_2 & a_3 & 1 & a_3 & 1 & 1 \ldots & 1
\\
1 & 1 & 1 & 1 & 1 & 0 & 1 & 0 & 0 \ldots & 0
\end{array}
\right]
\\[1em]
0 \leq a_1 \leq a_2, \ 0 \leq a_3, 
\\
a_1+a_2=2a_3+1,
\ m = 4a_3+1
\end{array}
$
}
&
\tiny{
\setlength{\arraycolsep}{1pt}
$
\left[
\begin{array}{c}
2a_3+2
\\
1
\end{array}
\right]
$
}
&
\small{$m+5$}
\\
\midrule
7
&
$
\begin{array}{c}
\frac
{\KK[T_1, \ldots , T_8, S_1,\ldots,S_m]}
{\langle T_{1}T_{2}+T_{3}T_{4}+T_5T_{6}+T_7^2T_8 \rangle}
\\
\scriptstyle m \geq 1
\end{array}
$
&
\tiny{
\setlength{\arraycolsep}{1pt}
$
\begin{array}{c}
\left[
\begin{array}{cccccccc|cc}
0 & 2a_5+1 & a_1 & a_2 & a_3 & a_4 & a_5 & 1 & 1 \ldots & 1
\\
1 & 1 & 1 & 1 & 1 & 1 & 1 & 0 & 0 \ldots & 0
\end{array}
\right]
\\[1em]
a_i \geq 0, \ m = 5a_5+2,
\\
a_1 + a_2 = a_3 + a_4 = 2a_5 +1
\end{array}
$
}
&
\tiny{
\setlength{\arraycolsep}{1pt}
$
\left[
\begin{array}{c}
2a_5+2
\\
1
\end{array}
\right]
$
}
&
\small{$m+5$}
\\
\midrule
8
&
$
\begin{array}{c}
\frac
{\KK[T_1, \ldots , T_8, S_1,\ldots,S_m]}
{\langle T_{1}T_{2}+T_{3}T_{4}+T_5T_{6}+T_7T_8 \rangle}
\\
\scriptstyle m = 6
\end{array}
$
&
\tiny{
\setlength{\arraycolsep}{1pt}
$
\left[
\begin{array}{cccccccc|cc}
0 & 0 & 0 & 0 & 0 & 0 & -1 & 1 & 1  \ldots & 1
\\
1 & 1 & 1 & 1 & 1 & 1 & 1 & 1 & 0 \ldots & 0
\end{array}
\right]
$
}
&
\tiny{
\setlength{\arraycolsep}{1pt}
$
\left[
\begin{array}{c}
1
\\
2
\end{array}
\right]
$
}
&
\small{$m+5$}
\\
\midrule
9
&
$
\begin{array}{c}
\frac
{\KK[T_1, \ldots , T_8, S_1,\ldots,S_m]}
{\langle T_{1}T_{2}+T_{3}T_{4}+T_5T_{6}+T_7T_8 \rangle}
\\
\scriptstyle m \geq 2
\end{array}
$
&
\tiny{
\setlength{\arraycolsep}{1pt}
$
\begin{array}{c}
\left[
\begin{array}{cccccccc|cc}
0 & a_1 & a_2 & a_3 & a_4 & a_5 & a_6 & a_7 & 1 \ldots & 1
\\
1 & 1 & 1 & 1 & 1 & 1 & 1 & 1 & 0 \ldots & 0
\end{array}
\right]
\\[1em]
a_i \geq 0, \ m = 3a_1,
\\
a_1 =a_2+a_3=a_4+a_5=a_6+a_7
\end{array}
$
}
&
\tiny{
\setlength{\arraycolsep}{1pt}
$
\left[
\begin{array}{c}
a_1+1
\\
1
\end{array}
\right]
$
}
&
\small{$m+5$}
\\
\midrule
10
&
$
\begin{array}{c}
\frac
{\KK[T_1, \ldots , T_8, S_1,\ldots,S_m]}
{\langle T_{1}T_{2}+T_{3}T_{4}+T_5T_{6}+T_7T_8 \rangle}
\\
\scriptstyle m \geq 2
\end{array}
$
&
\tiny{
\setlength{\arraycolsep}{1pt}
$
\begin{array}{c}
\left[
\begin{array}{cccccccc|ccc}
0 & 0 & 0 & 0 & 0 & 0 & 0 & 0 & 1 & 1 \ldots & 1
\\
1 & 1 & 1 & 1 & 1 & 1 & 1 & 1 & 0 & d_2 \ldots & d_m
\end{array}
\right]
\\[1em]
0 \leq d_2\leq\dots \leq d_m, d_m \leq 6
\\
m  a_m = 6 + d_2 + \ldots + d_m
\end{array}
$
}
&
\tiny{
\setlength{\arraycolsep}{1pt}
$
\left[
\begin{array}{c}
1
\\
d_m+1
\end{array}
\right]
$
}
&
\small{$m+5$}
\\
\midrule
11
&
$
\begin{array}{c}
\frac
{\KK[T_1, \ldots , T_7, S_1,\ldots,S_m]}
{\langle T_{1}T_{2}+T_{3}T_{4}+T_5T_{6}+T_7^2 \rangle}
\\
\scriptstyle m = 5
\end{array}
$
&
\tiny{
\setlength{\arraycolsep}{1pt}
$
\left[
\begin{array}{ccccccc|cc}
1 & 1 & 1 & 1 & 1 & 1 & 1 & 0 \ldots & 0
\\
-1 & 1 & 0 & 0 & 0 & 0 & 0 & 1 \ldots & 1
\end{array}
\right]
$
}
&
\tiny{
\setlength{\arraycolsep}{1pt}
$
\left[
\begin{array}{c}
1
\\
2
\end{array}
\right]
$
}
&
\small{$m+4$}
\\
\midrule
12
&
$
\begin{array}{c}
\frac
{\KK[T_1, \ldots , T_7, S_1,\ldots,S_m]}
{\langle T_{1}T_{2}+T_{3}T_{4}+T_5T_{6}+T_7^2 \rangle}
\\
\scriptstyle m \geq 2
\end{array}
$
&
\tiny{
\setlength{\arraycolsep}{1pt}
$
\begin{array}{c}
\left[
\begin{array}{ccccccc|cc}
1 & 1 & 1 & 1 & 1 & 1 & 1 & 0 \ldots & 0
\\
0 & 2a_5 & a_1 & a_2 & a_3 & a_4 & a_5 & 1 \ldots & 1
\end{array}
\right]
\\[1em]
2a_5 = a_1 +a_2=a_3+a_4,
\\
a_i \geq 0, \ m = 5a_5
\end{array}
$
}
&
\tiny{
\setlength{\arraycolsep}{1pt}
$
\left[
\begin{array}{c}
2a_5 +1 
\\
1
\end{array}
\right]
$
}
&
\small{$m+4$}
\\
\midrule
13
&
$
\begin{array}{c}
\frac
{\KK[T_1, \ldots , T_7, S_1,\ldots,S_m]}
{\langle T_{1}T_{2}+T_{3}T_{4}+T_5T_{6}+T_7^2 \rangle}
\\
\scriptstyle m \geq 2
\end{array}
$
&
\tiny{
\setlength{\arraycolsep}{1pt}
$
\begin{array}{c}
\left[
\begin{array}{ccccccc|ccc}
1 & 1 & 1 & 1 & 1 & 1 & 1& 0 & d_2 \ldots & d_m
\\
0 & 0 & 0 & 0 & 0 & 0 & 0 & 1 & 1 \ldots & 1
\end{array}
\right]
\\[1em]
d_2 \leq \ldots \leq d_m,
\\
m d_m = 5 +d_2+ \ldots + d_m
\end{array}
$
}
&
\tiny{
\setlength{\arraycolsep}{1pt}
$
\left[
\begin{array}{c}
1
\\
d_m +1 
\end{array}
\right]
$
}
&
\small{$m+4$}
\\
\bottomrule
\end{longtable}
}
}
\noindent
Moreover, each of the listed data defines 
a smooth truly almost Fano general arrangement 
variety of true complexity two and Picard 
number two.
\end{theorem}

\end{document}